\documentclass[12pt]{article}

\title{Embedded contact homology of prequantization bundles}
\author{Jo Nelson\footnote{Partially supported by NSF grants DMS-1810692, DMS-2104411, and  DMS-2103245.}\; and Morgan Weiler\footnote{Partially supported by NSF grants DMS-1745670 and DMS-2103245.}}
\date{}
\usepackage{amssymb}
\usepackage{latexsym}
\usepackage{amsmath}
\usepackage{amsthm}
\usepackage[scr=rsfso]{mathalfa}
\usepackage{graphicx}
\usepackage[abs]{overpic}

\usepackage[all,cmtip]{xy}

\usepackage[margin=1in]{geometry}

\usepackage[matha,  mathx]{mathabx}

\usepackage{wasysym}

\usepackage{enumerate}
\usepackage[usenames,dvipsnames,svgnames,table]{xcolor}

\usepackage{hyperref}
\hypersetup{colorlinks}

\definecolor{indigo}{RGB}{51,0,102}
\definecolor{brightpurple}{RGB}{102,0,153}
\definecolor{fuchsia}{RGB}{180,51,180}
\definecolor{jolightpurple}{RGB}{188,171,240}

\hypersetup{colorlinks,
linkcolor=brightpurple,
filecolor=brightpurple,
urlcolor=indigo,
citecolor=fuchsia}

\normalbaselines
\usepackage{tikz}

\newcommand{\mc}[1]{{\mathcal #1}}

\numberwithin{equation}{section}
\numberwithin{figure}{section}

\newtheorem{theorem}{Theorem}[section]
\newtheorem{proposition}[theorem]{Proposition}
\newtheorem{corollary}[theorem]{Corollary}
\newtheorem{lemma}[theorem]{Lemma}
\newtheorem{lemma-definition}[theorem]{Lemma-Definition}

\newtheorem{conjecture}[theorem]{Conjecture}

\theoremstyle{definition}
\newtheorem{definition}[theorem]{Definition}
\newtheorem{remark}[theorem]{Remark}

\newtheorem{example}[theorem]{Example}

\newcommand{\floor}[1]{\left\lfloor #1 \right\rfloor}
\newcommand{\ceil}[1]{\left\lceil #1 \right\rceil}

\newcommand{\C}{{\mathbb C}}

\newcommand{\Q}{{\mathbb Q}}
\newcommand{\R}{{\mathbb R}}
\newcommand{\N}{{\mathbb N}}
\newcommand{\Z}{{\mathbb Z}}

\newcommand{\ds}{{\dot{\Sigma}}}

\newcommand{\fp}{\mathfrak{p}}

\newcommand{\fj}{\mathfrak{J}}
\newcommand{\sj}{\mathscr{J}}
\newcommand{\cj}{\mathcal{J}}
\newcommand{\bj}{\mathbb{J}}

\newcommand{\B}{\mathscr{B}}
\newcommand{\E}{\mathscr{E}}
\newcommand{\W}{W^{k,p,\delta}}

\newcommand{\cs}{\mathcal{S}}
\newcommand{\cy}{\mathcal{Y}}

\newcommand{\D}{\mathbf{D}}
\newcommand{\bfi}{{\mathbf{I}}}
\newcommand{\bfj}{{\mathbf{J}}}

\newcommand{\conn}{\mathscr{C}}
\newcommand{\mgn}{\mathscr{M}_{g,n}}
\newcommand{\sm}{\mathscr{M}}

\newcommand{\CC}{{\mathcal C}}

\newcommand{\ga}{\gamma}
\newcommand{\deebar}{\bar{\partial}_{J}}

\newcommand{\op}{\operatorname}

\newcommand{\M}{\mc{M}}

\newcommand{\Aut}{\op{Aut}} 

\newcommand{\ind}{\op{ind}}

\newcommand{\CZ}{\op{CZ}}

\newcommand{\union}{\bigcup}

\newcommand{\bpm}{\begin{pmatrix}}
\newcommand{\epm}{\end{pmatrix}}

\newcommand{\czm}{{CZ}}

\begin{document}

\maketitle

\begin{abstract}
The 2011 PhD thesis of Farris \cite{farris} demonstrated that the ECH of a prequantization bundle over a Riemann surface is isomorphic as a $\Z_2$-graded group to the exterior algebra of the homology of its base.  We extend this result by computing the $\mathbb{Z}$-grading on the chain complex, permitting a finer understanding of this isomorphism and a stability result for ECH.  We fill in a number of technical details, including the Morse-Bott direct limit argument and the classification of certain $J$-holomorphic buildings.  The former requires the isomorphism between filtered Seiberg-Witten Floer cohomology and filtered ECH as established by Hutchings--Taubes \cite{cc2}.  The latter requires the work on higher asymptotics of pseudoholomorphic curves by Cristofaro-Gardiner--Hutchings--Zhang \cite{CGHZ} to obtain the writhe bounds necessary to appeal to an intersection theory argument of Hutchings--Nelson \cite{dc}. 

\end{abstract}

\tableofcontents

\bigskip


\section{Introduction}

Embedded contact homology (ECH) is an invariant of three dimensional contact manifolds, due to Hutchings \cite{Hu2}, with powerful applications to dynamics and symplectic embedding problems.  Most computations of ECH rely on enumerative toric methods, and Morse-Bott methods for ECH are quite subtle \cite{yao1, yao2}.\footnote{ In \cite{yao1}, Yao shows how to degenerate $J$-holomorphic curves and glue together transverse cascades  subject to certain transversality conditions.   In \cite{yao2}, Yao works out how to apply these methods in the context of Morse-Bott computations of ECH, notably in the case when the Morse-Bott families are one dimensional, as in the setting of toric domains.  Adapting Yao's methods to prequantization bundles requires establishing the necessary transversality conditions for cascades in the presence of higher genus curves.  This is a quite subtle endeavor for higher genus curves and branched covers of trivial cylinders, and would require require our results in \S \ref{sec:modspcs}-\ref{handleslides}, namely, those excluding contributions of higher genus curves to the ECH differential, as well as the ECH index calculations in \S \ref{sec:ECHI}.}  Following the framework given by Farris \cite{farris} and providing additional details in preparation for our future work, we show that for prequantization bundles, there is an appropriately  filtered ECH differential which only counts cylinders corresponding to unions of fibers over Morse flow lines of a perfect Morse function on the base. We then make use of direct limits for filtered ECH, as established in \cite{cc2}, to provide a Morse-Bott means of computing ECH for prequantization bundles over closed Riemann surfaces.   This permits us to conclude that the ECH of a prequantization bundle over a Riemann surface is isomorphic as a $\Z_2$-graded group to the exterior algebra of the homology of this base.

\subsection{Definitions and overview of ECH}

Let $Y$ be a closed three-manifold with a contact form $\lambda$. Let $\xi=\ker(\lambda)$ denote the associated contact structure, and let $R$ denote the associated Reeb vector field, which is uniquely determined by 
\[
\lambda(R)=1, \ \ \ d\lambda(R, \cdot)=0.
\]
A {\em Reeb orbit\/} is a map $\gamma:\R/T\Z\to Y$ for some $T>0$ such that $\gamma'(t)=R(\gamma(t))$, modulo reparametrization.  A Reeb orbit is said to be \textit{embedded} whenever this map is  injective. For a Reeb orbit as above, the linearized Reeb flow for time $T$ defines a symplectic linear map
\begin{equation}
\label{eqn:lrt}
P_\gamma:(\xi_{\gamma(0)},d\lambda) \longrightarrow (\xi_{\gamma(0)},d\lambda).
\end{equation}
The Reeb orbit $\gamma$ is {\em nondegenerate\/} if $P_\gamma$ does not have $1$ as an eigenvalue.  The contact form $\lambda$ is called nondegenerate if all Reeb orbits are nondegenerate; generic contact forms have this property. Fix a nondegenerate contact form below.  A nondegenerate Reeb orbit $\gamma$ is {\em elliptic\/} if $P_\gamma$ has eigenvalues on the unit circle and {\em  hyperbolic\/} if $P_\gamma$ has  real eigenvalues. If $\tau$ is a homotopy class of trivializations of $\xi|_\gamma$, then the {\em Conley-Zehnder index\/} $\CZ_\tau(\gamma)\in\Z$ is defined; see the review in \S \ref{cz-sec}. The parity of the Conley-Zehnder index does not depend on the choice of trivialization $\tau$, and is even when $\gamma$ is positive hyperbolic and odd otherwise.  We say that an almost complex structure $J$ on $\R\times Y$ is {\em $\lambda$-compatible\/} if $J(\xi)=\xi$; $d\lambda(v,Jv)>0$ for nonzero $v\in\xi$; $J$ is invariant under translation of the $\R$ factor; and $J(\partial_s)=R$, where $s$ denotes the $\R$ coordinate.   We denote the set of all $\lambda$-compatible $J$ by $\sj(Y,\lambda)$.

ECH is defined roughly as follows, with a complete description given in \S \ref{basics}. Given a closed 3-manifold $Y$ equipped with a nondegenerate contact form $\lambda$ and generic $\lambda$-compatible $J$, the \emph{ECH chain complex} (with respect to a fixed homology class $\Gamma \in H_1(Y)$) is the $\Z_2$ vector space freely generated by finite sets of pairs $\alpha = \{ (\alpha_i, m_i) \}$ where the $\alpha_i$ are distinct embedded Reeb orbits, the $m_i$ are positive integers, the total homology class $\sum_i m_i[\alpha_i] = \Gamma$, and $m_i=1$ whenever $\alpha_i$ is hyperbolic.    

Let $\M^J(\alpha,\beta)$ denote the set of $J$-holomorphic currents from $\alpha$ to $\beta$.  The \emph{ECH differential} is a mod 2 count of ECH index 1  currents.  The definition of the ECH index is the key nontrivial part of ECH \cite{Hindex}, and under the assumption that $J$ is generic, guarantees that the curves are embedded, except possibly for multiple covers of \emph{trivial cylinders} $\R \times \gamma$ where $\gamma$ is a Reeb orbit.  That the differential squares to zero is a technical endeavor requiring obstruction bundle gluing, \cite{obg1}, \cite{obg2}.  Let $ECH_*(Y,\lambda,\Gamma,J)$ denote the homology of the ECH chain complex.  It turns out that this homology does not depend on the choice of $J$ or on the contact form $\lambda$ for $\xi$, and so defines a well-defined $\Z_2$ module $ECH_*(Y,\xi,\Gamma)$.  The proof of invariance goes through Taubes' isomorphism with Seiberg-Witten Floer cohomology \cite{taubesechswf1}-\cite{taubesechswf5}.

There is a filtration on ECH which enables us to compute it via successive approximations, as explained in \S\ref{subsec:PQBgens} and \S\ref{subsec:directlimit}. The \emph{symplectic action} or \emph{length} of an Reeb current $\alpha=\{(\alpha_i,m_i)\}$ is
\[
\mathcal{A}(\alpha):=\sum_im_i\int_{\alpha_i}\lambda.
\]
If $J$ is $\lambda$-compatible and there is a $J$-holomorphic current from $\alpha$ to $\beta$, then $\mathcal{A}(\alpha)\geq\mathcal{A}(\beta)$ by Stokes' theorem, since $d\lambda$ is an area form on such $J$-holomorphic curves. Since $\partial$ counts $J$-holomorphic currents, it decreases symplectic action, i.e.,
\[
\langle\partial\alpha,\beta\rangle\neq0\Rightarrow\mathcal{A}(\alpha)\geq\mathcal{A}(\beta).
\]

Let $ECC_*^L(Y,\lambda,\Gamma;J)$ denote the subgroup of $ECC_*(Y,\lambda,\Gamma;J)$ generated by Reeb currents of symplectic action less than $L$. Since $\partial$ decreases action, it is a subcomplex; we denote the restriction of $\partial$ to $ECC_*^L$ by $\partial^L$. It is shown in \cite[Theorem 1.3]{cc2} that the homology of $ECC_*(Y,\lambda,\Gamma;J)$ is independent of $J$, therefore we denote its homology by $ECH_*^L(Y,\lambda,\Gamma)$, which we call \emph{filtered ECH}.

\subsection{Main theorem and outline of the proof}\label{ssec:introoutline}

Let $(\Sigma_g, \omega)$ be a closed Riemann surface  of genus $g \geq 0$ such that the cohomology class $\frac{[\omega]}{2\pi}$ admits an integral lift in $H^2(\Sigma_g,\Z)$.  Consider the the principal $S^1$-bundle $\fp: Y \to \Sigma_g$  with Euler class $e = -\frac{1}{2\pi}[\omega]$.  A connection on $Y$ is an imaginary-valued 1-form $A \in \Omega^1(Y, i\R)$ that is $S^1$-invariant and satisfies $A(X) = i$, where the vector field
\[
X(p): = \frac{d}{d\theta}\bigg \vert_{\theta =0} p e^{i\theta}, \ \ \ \theta \in \R/2\pi \Z,
\]
 is the derivative of the $S^1$-action of the bundle.  Since the Euler class of $Y$ is $-\frac{1}{2\pi}[\omega]$, the connection can be chosen such that its curvature is
  \[
F_A = dA = i \fp^*\omega;
\]
see \cite[Thm. 2.7.4]{intro}.  The 1-form 
\[
\lambda:= -i A
\]
is a contact form on $Y$ and satisfies $d\lambda = \fp^*\omega$.  Since $\lambda$ is $S^1$-invariant and $\lambda(X) = 1$, the Reeb vector field associated to $\lambda$ is $R:=X$, and all Reeb orbits are degenerate and have period $2\pi$. This contact manifold $(Y,\lambda)$ is called the \emph{prequantization bundle} of $(\Sigma_g, \omega)$, and the construction can be generalized to a symplectic base of arbitrary dimension.  The Hopf fibration is an example of a prequantization bundle of Euler class -1 over the sphere, while the lens space $L(|e|,1)$ is an example of a prequantization bundle of arbitrary negative Euler class $e$ over the sphere.

Our main result is the following computation of the ECH of prequantization bundles. 

\begin{theorem}\label{thm:mainthm}
Let $(Y, \xi = \mbox{\em ker} \lambda)$ be a prequantization bundle over $(\Sigma_g, \omega)$ of negative Euler class $e$. Then as $\Z_2$-graded $\Z_2$-modules,
\[
\bigoplus_{\Gamma\in H_1(Y;\Z)}ECH_*(Y,\xi,\Gamma)\cong\Lambda^*H_*(\Sigma_g;\Z_2).
\]
Moreover, each $\Gamma\in H_1(Y;\Z)$ satisfying $ECH_*(Y,\xi,\Gamma)\neq0$ corresponds to a number in $\{0,\dots,-e-1\}$, and under this correspondence
\begin{equation}\label{eqn:Lambdacorresp}
ECH_*(Y,\xi,\Gamma)\cong\bigoplus_{d\in\Z_{\geq0}}\Lambda^{\Gamma+(-e)d}H_*(\Sigma_g;\Z_2)
\end{equation}
as $\Z_2$-graded $\Z_2$-modules. 

 When $\Gamma=0$, the $\Z_2$-graded isomorphism (\ref{eqn:Lambdacorresp}) can be upgraded to an isomorphism of $\Z$-graded $\Z_2$-modules; the bijection on gradings is as follows. Let $*$ denote the grading on ECH and $\bullet$ the grading on the right hand side of (\ref{eqn:Lambdacorresp}). Then
\begin{equation}\label{eqn:0grading}
*=-ed^2+(\chi(\Sigma_g)+e)d+\bullet.
\end{equation}
\indent When $\Gamma\neq0$, we have a relatively $\Z$-graded isomorphism: Let $\alpha\in\Lambda^{\Gamma+(-e)d_\alpha}H_*(\Sigma_g;\Z_2)$ and $\beta\in\Lambda^{\Gamma+(-e)d_\beta}H_*(\Sigma_g;\Z_2)$, and let $\alpha$ and $\beta$ also denote the corresponding homogeneous ECH elements under (\ref{eqn:Lambdacorresp}). Let $|\cdot|_*$ denote the ECH grading and $|\cdot|_\bullet$ the grading on the right hand side of (\ref{eqn:Lambdacorresp}). Then
\begin{equation}\label{eqn:relgrading}
|\alpha|_*-|\beta|_*=-e(d_\alpha^2-d_\beta^2)+(\chi(\Sigma_g)+2\Gamma+e)(d_\alpha-d_\beta)+|\alpha|_\bullet-|\beta|_\bullet.
\end{equation}
\end{theorem}

\begin{remark} ECH generally only carries a relative $\Z_d$-grading, but it has a $\Z$-grading if the class $c_1(\xi)+2PD(\Gamma)$ is torsion. The fact that this class is torsion is explained in Remark \ref{rem:grading}. We also implicitly upgrade the relative grading to an absolute grading when $\emptyset$ is a generator by mandating that it has ECH index zero. This occurs when $\Gamma=0$.
\end{remark}

\begin{example}\label{ex:lensspc} As a refinement of Theorem \ref{thm:mainthm}, we compute the $\Z$-graded ECH of the lens spaces $L(|e|,1)$, obtaining the ECH version of \cite[Corollary 3.4]{kmos} as Theorem \ref{thm:echS2}:
\[
ECH_*(Y,\xi,\Gamma)=\begin{cases}
\Z&\text{ if $*\in2\Z_{\geq0}$},
\\0&\text{ else}.
\end{cases}
\]
\end{example}

\begin{remark}\label{rmk:Zcoeff} In cases where the differential $\partial$ vanishes because the moduli spaces of ECH index one currents are empty, as it does in Theorem \ref{thm:echS2} (see the proof in \S\ref{sssec:pfg=0}, where we show that the moduli spaces are empty for index reasons), we may upgrade from $\Z_2$-coefficients to $\Z$-coefficients. ECH is defined with $\Z$-coefficients rather than $\Z_2$-coefficients, but because that definition requires putting orientations on the moduli spaces and $\Z$-coefficients have not yet been necessary for any applications, we do not define orientations in this paper.
\end{remark}

We now give a sketch of the proof, which follows much of Farris' approach as outlined in \cite[\S 1]{farris}, and provide the organization of the paper along the way.   In \S \ref{basics}, we provide a short primer on the ingredients that go into ECH.  Then, in \S \ref{sec:ECHI} we give a detailed computation of the ECH index for certain Reeb currents in a perturbation of the canonical contact form associated to a prequantization bundle by a perfect Morse function $H$. Our computation of the ECH index uses the blueprint given in \cite[\S 3]{farris} and fills in a number of details which were not previously available.  Before we can give the computation, we need to perturb  the degenerate canonical  contact form $\lambda$.  

As in \cite{jo2}, we use the following perturbation in \S \ref{sec:ECHI} to employ Morse-Bott computational methods on action filtered ECH.  We use a Morse-Smale function $H:\Sigma\to\R$ such that $|H|$ is ${C^2}$ close to 1 to perturb the canonical connection contact form $\lambda$,
    \[
    \lambda_\varepsilon:=(1+\varepsilon\mathfrak{p}^*H)\lambda.
    \]
The perturbed Reeb vector field is
    \[
    R_\varepsilon=\frac{R}{1+\varepsilon\mathfrak{p}^*H}+\frac{\varepsilon\tilde X_H}{(1+\varepsilon\mathfrak{p}^*H)^2}
    \]
where $\tilde X_H$ is the horizontal lift of the Hamiltonian\footnote{Our convention for defining the Hamiltonian vector field is $X_H=\omega(dH, \cdot)$.} vector field $X_H$ to $\xi$.  Note that if $p \in \mbox{Crit}(H)$ then $X_H(p)=0$.  Fix $L>0$; there  exists $  \varepsilon >0$ so that the orbit  $\gamma$ of $R_\varepsilon$ satisfies the following dichotomy:
\begin{itemize}
\item if $\mathcal{A}(\gamma)<L$ then $\gamma$ is nondegenerate and projects to $p \in \mbox{Crit}(H)$;
\item if $\mathcal{A}(\gamma)>L$ then $\gamma$ loops around the tori above the orbits of $X_H$, or is a larger iterate of a fiber above $p\in\mbox{Crit}(H)$.
 \end{itemize}
We denote the $k$-fold cover projecting to $p \in \mbox{Crit}(H)$ by $\gamma_p^k$.   We have the following expression for the Conley-Zehnder index (see \cite{vknotes}, \cite[\S 4]{jo2}):
    \[
    CZ_\tau(\gamma_p^k) = RS_\tau( \mbox{fiber}^k) - \frac{\mbox{dim}(\Sigma_g)}{2} + \mbox{ind}_p(H).
    \] 
Using the constant trivialization $\tau$ of $\xi=\mathfrak{p}^*T\Sigma_g$, we have that the Robbin-Salamon index of the degenerate fiber is $RS_\tau( \mbox{fiber}^k) =0$.  Thus
     \[
    CZ_\tau(\gamma_p^k) = \mbox{ind}_p(H) -1.
    \]
 If $\mbox{ind}_p(H)=1$ then $\gamma_p$ is positive hyperbolic.  Since $\mathfrak{p}$ is a bundle, all linearized return maps are close to the identity, thus there are no negative hyperbolic orbits.  If $\mbox{ind}_p(H)=0,2$ then $\gamma_p$ is elliptic.

After taking $H$ to be perfect, we denote the index zero elliptic orbit by $e_-$, the index two elliptic orbit by $e_+$, and the hyperbolic orbits by $h_1,\dots,h_{2g}$.  The critical points of a perfect $H$ form a basis for $H_*(\Sigma_g;\Z_2)$. The generators of $ECC_*^L(Y,\lambda_{\varepsilon(L)},\Gamma)$ for appropriate choices of $L$ and $\varepsilon(L)$ are of the form $e_-^{m_-}h_1^{m_1}\cdots h_{2g}^{m_{2g}}e_+^{m_+}$ where $m_i=0,1$. Hence they form a basis for the exterior algebra $\Lambda^*H_*(\Sigma_g;\Z_2)$ when both are given a $\Z_2$ grading, under the map generated by sending each critical point to its Morse homology class.   

 For homologous Reeb currents $\alpha$ and $\beta$,  let $d$ denote the difference in the number of embedded orbits, counted with multiplicity, appearing in $\alpha$ and $\beta$, divided by $-e$:
\begin{equation}\label{intro:d}
d=\frac{M-N}{|e|},
\end{equation}
where $M:=m_-+m_1+\cdots+m_{2g}+m_+$ and similarly $N:=n_-+n_1+\cdots+n_{2g}+n_+$.

We obtain the following formula for the ECH index, cf. Proposition \ref{prop:indexcalc}:

\begin{proposition}\label{intro:ECHindex}
Let $(Y,\lambda)$ be a prequantization bundle over a surface $\Sigma_g$ with Euler class $e\in\Z_{<0}$. The ECH index in $ECH_*^L(Y,\lambda_{\varepsilon(L)},\Gamma)$ satisfies the following formula for any $\Gamma, L$, and Reeb currents $\alpha$ and $\beta$, where $L$ and $\varepsilon(L)$ are as in Lemma \ref{lem:efromL}, 
\begin{equation}\label{introeqn:indexformula}
I(\alpha,\beta)=\chi(\Sigma_g)d-d^2e+2dN+m_+-m_--n_++n_-.
\end{equation}
\end{proposition}

\begin{remark}
The above formula recovers the ECH index for perturbations of the standard $S^3$. The index formula in \cite{Hu2}, when applied to the contact forms on the ellipsoids $E(1+\delta_n,1-\eta_n)$ where $\delta_n,\eta_n\to0$ and $\frac{1-\delta_n}{1-\eta_n}\in\R-\Q$, converges to our formula with $g=0$ and $e=-1$.
\end{remark}

We prove that the filtered ECH differential $\partial$ only counts unions of cylinders corresponding to Morse flows on $\Sigma_g$, therefore $\partial(e_-^{m_-}h_1^{m_1}\cdots h_{2g}^{m_{2g}}e_+^{m_+})$ is a sum over all ways to apply $\partial^{Morse}$ to $h_i$ or $e_+$.  In \S \ref{sssec:pfg>0}, we show that $\partial$ is zero because $H$ is perfect. 

Before sketching the correspondence between Morse flows and the filtered ECH differential (the content of \S \ref{sec:modspcs}-\ref{handleslides}), we say a few words about how we make the above Morse-Bott formalism precise.   In \S \ref{sec:ECHI} we show that given $L$ there exists $ \varepsilon(L)>0$ so that the generators of $ECC^L_*(Y,\lambda_\varepsilon,J)$ consist solely of orbits which are fibers over critical points.  We have:

\begin{proposition}\label{intro:directlimitcomputesfiberhomology} 
With $Y,\lambda$, and $\varepsilon(L)$ as in Lemma \ref{lem:efromL}, for any $\Gamma \in H_1(Y;\Z)$, there is a direct system consisting of all the $ECH^L_*(Y,\lambda_{\varepsilon(L)},\Gamma)$. The direct limit $\lim_{L\to\infty}ECH^L_*(Y,\lambda_{\varepsilon(L)},\Gamma)$ is the homology of the chain complex generated by Reeb currents $\{(\alpha_i,m_i)\}$ where the $\alpha_i$ are fibers above critical points of $H$ and $\sum_im_i[\alpha_i]=\Gamma$.
\end{proposition}

Then in \S \ref{sec:finalcomp}, rather than considering degenerations of moduli spaces directly, we instead pass to filtered ECH and take direct limits by appealing to the isomorphism with filtered Seiberg-Witten theory \cite{cc2}.  We prove:

\begin{theorem}\label{intro:thm:dlisECH} 
With $Y,\lambda,\varepsilon(L)$ as in Lemma \ref{lem:efromL}, and for any $\Gamma \in H_1(Y;\Z)$,
\[
\lim_{L\to\infty}ECH_*^L(Y,\lambda_{\varepsilon(L)},\Gamma)=ECH_*(Y,\xi,\Gamma).
\]
\end{theorem}

{Now we sketch why the appropriately filtered differential $\partial^{L}$ only counts cylinders which are the union of fibers over Morse flow lines in $\Sigma_g$. Cylinder counts associated to the above perturbation permit the use of fiberwise $S^1$-invariant $\fj = \fp^*j_{\Sigma_g}$,  even for multiply covered curves, by automatic transversality, cf. \cite{jo2, Wtrans}.  However for somewhere injective curves with genus, we cannot obtain transversality using a $S^1$-invariant $\fj$, and following \cite[\S 6]{farris}, we make use of $S^1$-invariant domain dependent almost complex structures $\bj \in  \sj_\ds^{S^1}$ in \S \ref{sec:modspcs} to achieve regularity.   That the ECH differential does not admit noncylindrical contributions, follows in part from the following proposition, which demonstrates that domain dependent positive genus curves do not contribute to ECH index 1 moduli spaces.

\begin{proposition}\label{intro-nogenus}
Let $\alpha$ and $\beta$ be nondegenerate Reeb currents and $\bj \in  \sj_\ds^{S^1}$ be generic.  If $\op{deg}(\alpha,\beta) >0$ and $I(\alpha,\beta)=1$ then $\M^{\bj}(\alpha,\beta) = \emptyset$.
\end{proposition}

In \S \ref{degree}, we define the notion of degree of a completed projected pseudoholomorphic curve in $\Sigma_g$ and show that $\op{deg}(\alpha,\beta) = d$, where $d$ is as in \eqref{intro:d}, c.f. Definition \ref{degree-gen}.   By Lemma \ref{lem:d0g0}, a curve $C$ which contributes nontrivially to the ECH differential has degree zero if and only if it is a cylinder.  We note that Proposition \ref{intro-nogenus} and its proof previously appeared in \cite[Corollary 6.2.3]{farris}, and we repeat his arguments from which this corollary follows, including the regularity result for $S^1$-invariant domain dependent almost complex structures, cf. Theorem \ref{thm:ddacs} and Corollary \ref{nogenus}.

However, the definition of ECH relies on the choice of a generic $\lambda$-compatible domain independent almost complex structure} $J$, rather than a domain dependent $\bj$.  
Thus, as in \cite[\S 7]{farris},  we consider a generic one parameter family of almost complex structures $\{ \cj_t\}_{t\in[0,1]}$ interpolating between $\cj_0:=\bj \in \sj_\ds^{S^1} $ and  $\cj_1:=J \in \sj(Y,\lambda)$. We obtain the following result in \S \ref{handleslides}, which demonstrates that after passing to homology the curve counts are unaffected.

\begin{proposition}\label{intro-handleslides}
Let $\alpha$ and $\beta$ be admissible Reeb currents with $I(\alpha,\beta)=1$ and $\op{deg}(\alpha,\beta) >0.$  For generic paths $\{ \cj_t\}_{t\in[0,1]}$, the moduli space $\mathcal{M}_t := \mathcal{M}^{\cj_t}(\alpha, \beta)$  is cut out transversely save for a discrete number of times $t_0,...,t_\ell \in (0,1)$.    For each such $t_i$,  the ECH differential can change either by:
 \begin{enumerate}[\em (a)]
 \item The creation or destruction of a pair of oppositely signed curves.\footnote{Because we are using $\Z_2$-coefficients, we will not sort through the signs.}
 \item An ``ECH handleslide."  
 \end{enumerate}
 In either case, the homology is unaffected. 
 \end{proposition}
 
From this, we obtain the desired corollary: 
 
 \begin{corollary}\label{intro: nogenusJ}
Let $\alpha$ and $\beta$ be nondegenerate admissible Reeb currents and $J \in \sj(Y,\lambda)$ be generic.  If $\op{deg}(\alpha,\beta) >0$ and $I(\alpha,\beta)=1$ then the mod 2 count $\#_{\Z_2}\M^{J}(\alpha,\beta) = 0$.  If $\alpha$ and $\beta$ are associated to $\lambda_\varepsilon$ as in Lemma \ref{lem:efromL} and $\mathcal{A}(\alpha), \mathcal{A}(\beta) <L(\varepsilon)$ then $\langle \partial^{L(\varepsilon)} \alpha, \beta \rangle =0$.\end{corollary}

This corollary completes the proof that if any moduli space of $J$-holomorphic currents contributes to $\partial^L$ then the currents must consist of trivial cylinders together with unions of fibers over Morse flow lines in the base $\Sigma_g$, which contribute to the Morse differential. We obtain the main result, Theorem \ref{thm:mainthm} as the counts of such cylinders equal the counts of the Morse flow lines which are their images under $\fp$, cf. Proposition \ref{cylinder-to-morse}. We show in \S \ref{subsec:pfmainthm}, under the assumption that the Morse function $H$ is perfect, that the latter counts are all zero.

We conclude with a brief discussion of Proposition \ref{intro-handleslides}, which is proven in \S \ref{handleslides}.  In Proposition \ref{intro-handleslides}(a) the mod 2 counts of curves in $\M_{t_i-\varepsilon}(\alpha,\beta)$ and $\M_{t_i+\varepsilon}(\alpha,\beta)$ are the same.  The differential can change at an ECH handleslide, at which a sequence of Fredholm and ECH index 1 curves $\{C(t)\}$ breaks into a holomorphic building in the sense of \cite{BEHWZ} into components consisting of an ECH and Fredholm index 0 curve, an ECH and Fredholm index 1 curve, and some ``connectors," which are Fredholm index 0 branched covers of a trivial cylinder $\gamma \times \R$.    In \S \ref{sec:connectors}, we demonstrate that connectors cannot appear at the top most or bottom most level of the building via intersection theory arguments similar to \cite[\S 4]{dc}, appealing to refined asymptotic estimates of Cristofaro-Gardiner--Hutchings--Zhang \cite{CGHZ}.   The framework regarding this classification of connectors was previously indicated in \cite[\S 7]{farris}. We additionally explain how to invoke the obstruction bundle gluing theorems of \cite{obg1, obg2} and then review how Farris' inductive argument involving the degree of a completed projected curve finishes the proof of Proposition \ref{intro-handleslides}(b). 

\subsection{Finer points of the isomorphism and future work}\label{ssec:futurework}

In \S\ref{sec:finalcomp}, as a step in the proof of Theorem \ref{thm:mainthm}, we prove the second, stronger conclusion of Theorem \ref{thm:mainthm} in terms of the total $H_1(Y)$ classes $\Gamma$ and the total multiplicity of representatives of the ECH homology classes, namely that as $\Z_2$-graded $\Z_2$-modules
\begin{equation}\label{eqn:splitting}
ECH_*(Y,\xi,\Gamma)\cong\bigoplus_{d\in\Z_{\geq0}}\Lambda^{\Gamma+(-e)d}H_*(\Sigma_g;\Z_2).
\end{equation}
Here we are abusing notation on the right hand side by considering $\Gamma$ as the element of $\{0,\dots,-e-1\}$ corresponding to its homology class (see Lemma \ref{lem:HYreps} (i)). We illustrate the $\Z$-grading (\ref{eqn:0grading}) in an example.

\begin{example}
The following table organizes the first several generators in the case $g=2, e=-1, \Gamma=0$. Its columns have fixed ECH index $I=-2,\dots,4$ while its rows have fixed multiplicity $0,\dots,4$.

\[\begin{array}{r|ccccccc}
&I=-2&I=-1&I=0&I=1&I=2&I=3&I=4
\\\hline
\\\vspace{-.9cm}
\\\Lambda^0&&&\emptyset
\\\Lambda^1&e_-&h_i&e_+
\\\Lambda^2&e_-^2&e_-h_i&e_-e_+, h_ih_j&h_ie_+&e_+^2
\\\Lambda^3&&&e_-^3&e_-^2h_i&e_-^2e_+, e_-h_ih_j&e_-h_ie_+, h_ih_jh_k&e_-e_+^2, h_ih_je_+
\\\Lambda^4&&&&&&&e_-^4
\end{array}\]
We use $I$ to denote $I(\cdot,\emptyset)$, and we use $h_i$ to denote any one of the four hyperbolic generators. Whenever any two or three $h_i$ appear together, they are all different.
\end{example}

Notice that for $*\geq1$, the $ECH_*$ groups are isomorphic to $\Z_2^8$. This is a special case of the following more general result, which is a corollary of Theorem \ref{thm:mainthm}. It relies on the degree $-2$ map
\[
U:ECH_*(Y,\xi,\Gamma)\to ECH_{*-2}(Y,\xi,\Gamma)
\]
induced by a chain map which counts index 2 $J$-holomorphic curves passing through a base point.
\begin{corollary}[Stability of ECH]\label{stability}
For $*$ sufficiently large and $g>0$, the groups $ECH_*(Y,\xi,\Gamma)$ are isomorphic to $\mathbb{Z}_2^{f(g)}$, where $f(g) = 2^{2g-1}$.  
\end{corollary}
(For the computation of all of the ECH groups for $g=0$, see Example \ref{ex:lensspc}.)
\begin{proof} Firstly we show that for $*$ sufficiently large, all of the homology groups are isomorphic. Seiberg-Witten Floer cohomology, denoted $\widehat{HM}^\bullet$ (see \S\ref{subsec:directlimit}), is an invariant of smooth closed three-manifolds $Y$ with spin-c structures $\mathfrak{s}$, for which Taubes \cite{taubesechswf1}-\cite{taubesechswf5} has shown that $\widehat{HM}^{-*}(Y,\mathfrak{s}_{\xi,\Gamma})\cong ECH_*(Y,\xi,\Gamma)$ (the notation $\mathfrak{s}_{\xi,\Gamma}$ denotes a spin-c structure naturally associated to $(\xi,\Gamma)$ as in \cite{cc2}). It fits into a long exact sequence with two other homology theories
\[
\cdots\leftarrow\widecheck{HM}^\bullet(Y,\mathfrak{s})\leftarrow\widehat{HM}^\bullet(Y,\mathfrak{s})\leftarrow\overline{HM}^{\bullet-1}(Y,\mathfrak{s})\leftarrow\cdots
\]
where $\overline{HM}^\bullet$ contains $\widehat{HM}^\bullet$ as a subcomplex, and is constructed to satisfy the property
\[
U:\overline{HM}^\bullet(Y,\mathfrak{s}_{\xi,\Gamma})\to\overline{HM}^{\bullet+2}(Y,\mathfrak{s}_{\xi,\Gamma}) \text{ is an isomorphism}
\]
where $U$ is a map on $\overline{HM}^\bullet$ equalling the image of the $U$-map in ECH under the isomorphism $\widehat{HM}^{-*}(Y,\mathfrak{s}_{\xi,\Gamma})\cong ECH_*(Y,\xi,\Gamma)$. Since $\widecheck{HM}^\bullet$ is zero for $\bullet$ sufficiently small as in the proof of \cite[Cor. 35.1.4]{KMbook}, the isomorphism from $\overline{HM}$ descends to an isomorphism
\[
U:\widehat{HM}^{-*}(Y,\mathfrak{s}_{\xi,\Gamma})\to\widehat{HM}^{-*+2}(Y,\mathfrak{s}_{\xi,\Gamma})
\]
thus to an isomorphism between $ECH_*(Y,\xi,\Gamma)$ and $ECH_{*-2}(Y,\xi,\Gamma)$ for $*$ sufficiently large.

We will compute the dimension of a group $ECH_i(Y,\xi,\Gamma)$ for a well-chosen value of $i$, where if $\Gamma\neq0$ we set $i$ to be the index relative to $e_-^\Gamma$. We first make two observations, both following directly from the index formula (\ref{introeqn:indexformula}):
\begin{itemize}
\item[(i)] Amongst all generators $e_-^{m_-}h_1^{m_1}\cdots h_{2g}^{m_{2g}}e_+^{m_+}$ with $m_-+m_1+\cdots+m_{2g}+m_+=M$, the highest value of $I(e_-^{m_-}h_1^{m_1}\cdots h_{2g}^{m_{2g}}e_+^{m_+},e_-^\Gamma)$ is realized by $e_+^M$, and the lowest by $e_-^M$. We also have
\[
I(e_+^M,e_-^M)=2M.
\]
\item[(ii)] We have
\[
I(e_-^{\Gamma-(d+1)e},e_+^{\Gamma-de})=\chi(\Sigma_g).
\]
\end{itemize}
By (i) and (ii), we can take $d$ large enough compared to $g$ so that all generators $e_-^{m_-}h_1^{m_1}\cdots h_{2g}^{m_{2g}}e_+^{m_+}$ with $I(e_-^{m_-}h_1^{m_1}\cdots h_{2g}^{m_{2g}}e_+^{m_+},e_-^{\Gamma-de})=M$ have
\begin{equation}\label{eqn:allsamed}
m_-+m_1+\cdots+m_{2g}+m_+=M=\Gamma-ed.
\end{equation}
Therefore, to complete the proof we will count all generators with $I(e_-^{m_-}h_1^{m_1}\cdots h_{2g}^{m_{2g}}e_+^{m_+},e_-^{\Gamma-de})=M$ which satisfy (\ref{eqn:allsamed}). Because $e_-^{m_-}h_1^{m_1}\cdots h_{2g}^{m_{2g}}e_+^{m_+}$ satisfies (\ref{eqn:allsamed}), we obtain
\begin{equation}\label{eqn:indexsamed}
I(e_-^{m_-}h_1^{m_1}\cdots h_{2g}^{m_{2g}}e_+^{m_+},e_-^{\Gamma-de})=m_+-m_-+\Gamma-de.
\end{equation}
By (\ref{eqn:allsamed}) and (\ref{eqn:indexsamed}) it suffices to count all generators with $m_-=m_+$ and $2m_-+m_1+\cdots+m_{2g}=M$. We can choose $d$ even larger if necessary so that $d>2g$.

If $\Gamma-de$ is even, then the set of such generators can be grouped into those with $0,2,4,\dots,2k,\dots,2g$ hyperbolic generators. There are ${2g\choose 2k}$ generators with $2k$ hyperbolic generators, therefore there are
\begin{equation}\label{eqn:evenbinomial}
{2g\choose0}+{2g\choose 2}+{2g\choose 4}+\cdots+{2g\choose 2k}+\cdots{2g\choose2g}=2^{2g-1}
\end{equation}
generators with $m_-=m_+$ and $2m_-+m_1+\cdots+m_{2g}=M$ in total. Similarly, if $\Gamma-de$ is odd, then the set of such generators can be grouped into those with $1,3,\dots,2k+1,\dots,2g-1$ hyperbolic generators, and there are
\begin{equation}\label{eqn:oddbinomial}
{2g\choose 1}+{2g\choose 3}+\cdots+{2g\choose 2k+1}+\cdots{2g\choose2g-1}=2^{2g-1}
\end{equation}
such generators in total. Both (\ref{eqn:evenbinomial}) and (\ref{eqn:oddbinomial}) are elementary equalities following from the sum of binomial coefficients
\begin{equation}\label{eqn:binomialsum}
\sum_{k=0}^n{n\choose k}x^k=(1+x)^n
\end{equation}
with $n=2g$. Setting $x=1$ shows that the sum of all ${2g\choose k}$ is $2^{2g}$, and adding (\ref{eqn:binomialsum}) for $x=-1$ gives us (\ref{eqn:evenbinomial}) with both sides doubled. Then (\ref{eqn:oddbinomial}) follows by subtracting (\ref{eqn:evenbinomial}) from (\ref{eqn:binomialsum}) with $x=1$.
\end{proof}

\begin{conjecture}\label{conj:Uperiodicity} For $*$ sufficiently large, the $U$ map is an isomorphism on the chain level of the ECH of prequantization bundles, for the chain complex of Proposition \ref{intro:directlimitcomputesfiberhomology}.
\end{conjecture}

When $g=0$, Conjecture \ref{conj:Uperiodicity} can likely be proved using Example \ref{ex:lensspc} and the methods of \cite[\S4.1]{Hu2}, discussed below. Furthermore, we will investigate the non-stable part of the homology, e.g. prove the ECH analogue of the computation of the reduced Heegaard Floer homology of nontrivial circle bundles in \cite[Thm. 5.6]{knotHF}.

We are further interested in the $U$-map at the chain level as it is crucial to the definition of the ``ECH spectrum." This is a list where the $k^\text{th}$ term is the minimum action of a homologically essential ECH cycle in $ECH_*(Y,\xi,[\emptyset])$ which is homologous to the preimage under $U^k$ of the class of the empty Reeb current. The ECH spectrum of a contact three-manifold governs the embedding properties of its symplectic fillings. It is expected that $U$ is an isomorphism in sufficiently large degree, as is known to be true for Seiberg-Witten cohomology. Additionally, we expect that one could recover the computation in \cite[Cor. 1.0.6]{moy} of the irreducible Seiberg-Witten Floer homology of Brieskorn manifolds.

One should be able to compute the $U$-map for general prequantization bundles following the computation in the $g=0, e=-1$ case (i.e., $Y=S^3$) in \cite[\S4.1]{Hu2}, which interprets $U$ as a count of index two gradient flow lines for a Morse function on the base $S^2$, together with a count of meromorphic multisections of the complex line bundle associated to the Hopf fibration. As discussed in \cite[\S1.2]{farris}, in the case of a general base, a $J$-holomorphic curve in $\R\times Y$ which intersects each $\R\times\{\text{fiber}\}$ the same number of times can be interpreted as a meromorphic multisection of the line bundle associated to $Y\to\Sigma_g$ with its asymptotic ends on Reeb orbits interpreted as either zeroes or poles.  Computations ECH of the disk cotangent bundles $D^*S^2$ and $D^*\mathbb{RP}^2$ based on our work have appeared in \cite[Prop. 3.7, 4.2]{fr}; additional Morse-Bott methods are assumed to compute the associated $U$-maps and ECH capacities \cite[Lem. 3.10, 4.3]{fr}.   


\begin{remark}
Recently Chen \cite[Thm.~1 \& 2]{chen}, has given an alternate means of computing the $U$-maps and ECH capacities for prequantization bundles over $S^2$ and $T^2$, where the filling is given by the unit disk subbundle of the associated prequantization line bundle.  He establishes certain ECH cobordism maps \cite[Thm.~3]{chen}, for all prequantization bundles.  To compute the ECH capacities for arbitrary prequantization bundles, one needs to establish additional properties.  When the base is $S^2$ or $T^2$,  Chen uses  Corollary \ref{stability} in tandem with some additional study of the $U$-map between the two generators in a particular grading to establish the necessary properties needed to compute the ECH capacities.  (Chen also gives a bundle theoretic means of recovering our computation of the ECH index \cite[Lem.~3.1, Rem.~3.1]{chen}.)
\end{remark}

Since this paper first appeared, we have since extended our methods to compute the knot filtered ECH of the $T(p,q)$ Seifert fibrations of $S^3$, by considering them as prequantization bundles over orbifolds in \cite{knot, calabi}.  (This turned out to be more subtle than first expected.) Knot filtered ECH was introduced in \cite{HuMAC} to study the mean action and Calabi invariant of disk symplectomorphisms, and extended in \cite{weiler, weiler2} to study annular symplectomorphisms via similar methods. Since periodic surface symplectomorphisms can be realized as the return map of open book decompositions on Seifert fibered contact manifolds \cite{CH}, our computations allow us to study the dynamics of such symplectomorphisms.  

\bigskip

\noindent \textbf{Acknowledgements.} We thank Michael Hutchings for numerous elucidating conversations and comments on an earlier draft of this paper.  Moreover, we thank the referee for catching a number of typos and pointing out a couple of imprecise statements.

\section{Basics of ECH}\label{basics}
Let $Y$ be a closed contact 3-dimensional manifold equipped with a nondegenerate contact form $\lambda$; let $\xi=\ker(\lambda)$ denote the associated contact structure.  We say that a closed Reeb orbit $\gamma$ is \emph{elliptic} if the eigenvalues of the linearized return map $P_{\gamma}$ are on the unit circle, \emph{positive hyperbolic} if the eigenvalues of $P_{\gamma}$ are positive real numbers, and \emph{negative hyperbolic} if the eigenvalues of $P_{\gamma}$ are negative real numbers.

A \emph{Reeb current}\footnote{In previous literature, the terminology \emph{orbit set} was used in place of Reeb current.  An embedded Reeb orbit is a Reeb current of multiplicity one.} is a finite set of pairs $\alpha=\{(\alpha_i,m_i)\}$, where the $\alpha_i$ are distinct embedded Reeb orbits and the $m_i$ are positive integers.  We call $m_i$ the \emph{multiplicity} of $\alpha_i$ in $\alpha$.  The homology class of the Reeb current $\alpha$ is defined by
\[
[\alpha]=\sum_i m_i [\alpha_i] \in H_1(Y).
\]
The Reeb current $\alpha$ is \emph{admissible} if $m_i=1$ whenever $\alpha_i$ is positive or negative hyperbolic.

The ECH chain complex is generated by admissible Reeb currents.  The differential counts  ECH index one $J$-holomorphic currents in $\R \times Y$.  
The definition of the ECH index depends on three components: the relative first Chern class $c_\tau$, which detects the contact topology of the curves; the relative intersection pairing $Q_\tau$, which detects the algebraic topology of the curves; and the Conley-Zehnder terms, which detect the contact geometry of the Reeb orbits.    Further details are provided in the subsequent sections.

\subsection{Pseudoholomorphic curves and currents}

Given a punctured compact Riemann surface $(\ds, j)$, with a partition of its punctures $\mathcal{S}$ into a positive subset $\mathcal{S}_+$  and negative subset $\mathcal{S}_-$  we consider \emph{asymptotically cylindrical $J$-holomorphic} maps of the form 
\[
C:(\ds, j) \to (\R \times Y, J), \ \ \ dC \circ j = J \circ dC,
\]
subject to the following asymptotic condition.  If  $\gamma$ is a (possibly multiply covered Reeb orbit), a \emph{positive end} of $C$ at $\gamma$ is a puncture near which $C$ is asymptotic to $\R \times \gamma$ as $s \to \infty$ and a \emph{negative end} of $C$ at $\gamma$ is a puncture near which $C$ is asymptotic to $\R \times \gamma$ as $s \to -\infty$.  This means there exist holomorphic cylindrical coordinates identifying a punctured neighborhood of $z$ with a respective positive half-cylinder $Z_+=[0, \infty) \times S^1$ or negative half-cylinder $Z_-=(-\infty, 0] \times S^1$ and a trivial cylinder $\conn_{\gamma_z}: \R \times S^1 \to \R \times Y$ such that
\begin{equation}\label{ass1}
C(s,t) = \exp_{\conn_{\gamma_z}(s,t)}h_z(s,t) \mbox{ for $|s|$ sufficiently large},  
\end{equation}
where $h_z(s,t)$ is a vector field along $\conn_{\gamma_z}$ satisfying $|h_z(s,\cdot)| \to 0$ uniformly as $s \to \pm \infty$.  Both the norm and the exponential map are assumed to be defined with respect to a translation-invariant choice of Riemannian metric on $\R \times Y$.  To obtain a moduli space of $J$-holomorphic curves from the above asymptotically cylindrical maps, we mod out by the usual equivalence relation $(\ds, j, \mathcal{S}, C) \sim (\ds', j', \mathcal{S}', C')$, which exists whenever there exists a biholomorphism $\phi: (\ds, j) \to (\ds', j')$ taking $\mathcal{S}$ to $\mathcal{S}'$ with the ordering preserved, i.e. $\phi(\mathcal{S}_+) =\mathcal{S}'_+$ and $\phi(\mathcal{S}_-) =\mathcal{S}'_-$, such that $C = C' \circ \phi.$

\begin{definition}\label{covered}
An asymptotically cylindrical pseudoholomorphic curve 
\[
C: (\ds, j) \to (\R \times Y, J)
\]
is said to be \emph{multiply covered} whenever there exists a pseudoholomorphic curve
\[
\overline{C}: (\ds' , j') \to (\R \times Y, J),
\]
 and a holomorphic branched covering $\varphi: (\ds, j) \to (\ds', j') $ with $\mathcal{S}'_+ = \varphi(\mathcal{S}_+)$ and  $\mathcal{S}'_- = \varphi(\mathcal{S}_-)$ such that
\[
C = \overline{C} \circ \varphi, \ \ \ \mbox{deg}(\varphi)>1,
\]
allowing for $\varphi$ to not have any branch points.  The \emph{multiplicity} of $C$ is given by $\mbox{deg}(\varphi).$ 
\end{definition}
 
An asymptotically cylindrical pseudoholomorphic curve $C$ is called \emph{embedded}\footnote{Frequently in symplectic field theory literature, the terminology simple is used in place of embedded.} whenever it is not multiply covered.  In \cite[\S 3.2]{jo1} we gave a proof of the folk theorem that that every simple asymptotically cylindrical curve  is \emph{somewhere injective}, meaning for some  $z \in \ds$, which is not a puncture, 
\[
dC(z) \neq 0 \ \  \ C^{-1}(C(z))=\{z\}.
\]
A point $z \in \ds$ with this property is called an \emph{injective point} of $C$.

Let $\alpha=\{(\alpha_i,m_i)\}$ and $\beta=\{(\beta_j,n_j)\}$ be Reeb currents in the same homology class $\sum_i m_i [\alpha_i]=\sum_j n_j [\beta_j]=\Gamma\in H_1(Y).$  We define a \emph{$J$-holomorphic current} from $\alpha$ to $\beta$ to be a finite set of pairs $\mathcal{C} = \{ (C_k,d_k) \}$, where the $C_k$ are distinct irreducible somewhere injective $J$-holomorphic curves in $\R \times Y$, the $d_k$ are positive integers, $\mathcal{C}$ is asymptotic to $\alpha$ as a current as the $\R$-coordinate goes to $+\infty$, and $\mathcal{C}$ is asymptotic to $\beta$ as a current as the $\R$-coordinate goes to $-\infty$.  This last condition means that the positive ends of the $C_k$ are at covers of the Reeb orbits $\alpha_i$, the sum over $k$ of $d_k$ times the total covering multiplicity of all ends of $C_k$ at covers of $\alpha_i$ is $m_i$, and analogously for the negative ends.

Let $\mathcal{M}^J(\alpha,\beta)$ denote the set of $J$-holomorphic currents from $\alpha$ to $\beta$.  A holomorphic current $\mathcal{C} $ is said to be \emph{somewhere injective} if $d_k=1$ for each $k$.  A somewhere injective current is said to be \emph{embedded} whenever each $C_k$ is embedded and the $C_k$ are pairwise disjoint.

Let $H_2(Y,\alpha,\beta)$ denote the set of relative homology classes of 2-chains $Z$ in $Y$ such that 
\[
\partial Z =\sum_i m_i \alpha_i - \sum_j n_j \beta_j,
\] 
modulo boundaries of 3-chains.  The set $H_2(Y,\alpha,\beta)$ is an affine space over $H_2(Y)$, and every $J$-holomorphic current $\mathcal{C}\in \M^J(\alpha, \beta)$ defines a class $[\CC] \in H_2(Y,\alpha,\beta)$.


\subsection{The Fredholm index}\label{sec:freddie}
Let $C$ be an asymptotically cylindrical pseudoholomorphic curve with $k$ positive ends at (possibly multiply covered) Reeb orbits $\gamma_1^+,...,\gamma_k^+$ and $\ell$ negative ends at  $\gamma_1^-,...,\gamma_\ell^-$.  We denote the set of equivalence classes by 
\[
{\mathcal{M}}^J (\gamma^+,\gamma^-): ={\mathcal{M}}^J(\gamma_1^+,...,\gamma_k^+;\gamma_1^-,...,\gamma_\ell^-).
\]
Note that $\R$ acts on ${\mathcal{M}}^J(\gamma^+,\gamma^-)$ by translation of the $\R$ factor in $\R \times Y$.

The \emph{Fredholm index} of $C$ is defined by
\begin{equation}\label{w-index}
{\ind}(C) = -\chi(C)   + 2c_\tau(C) + \displaystyle \sum_{i=1}^k{CZ}_\tau(\gamma_i^+)  -  \sum_{j=1}^\ell {CZ}_\tau(\gamma_j^-).
\end{equation}
The {\em Euler characteristic} of the domain $\ds$ of $C$ is denoted by 
\[
\chi(C) =(2-2g(\Sigma) - k - \ell).
\]
The remaining terms are a bit more involved to define as they depend on the choice of a trivialization $\tau$ of $\xi$ over the Reeb orbits $\gamma_i^{+}$ and $\gamma_j^-$, which is compatible with $d\lambda$, and are defined in the following subsections.

The significance of the Fredholm index is that the moduli space of somewhere injective curves ${\mathcal{M}}^J(\gamma^+,\gamma^-)$ is naturally a manifold near $C$ of dimension $\mbox{ind}(C)$. In particular, we have the following results. 

\begin{theorem}{\em \cite[Theorem 8.1]{wendl-sft}} \label{si-thm1}
Fix a nondegenerate contact form $\lambda$ on a closed manifold $Y$.  Let $\sj(Y,\lambda)$ be the set of all $\lambda$-compatible almost complex structures on $\R \times Y$. Then there exists a comeager subset $\sj_{reg}(Y,\lambda) \subset \sj(Y,\lambda), $ such that for every $J \in \sj_{reg}(Y,\lambda) $, every curve $C \in {\mathcal{M}}^J(\gamma^+,\gamma^-)$ with a representative $C \colon ({\ds},j) \to (\R \times Y, J)$ that has an injective point $z\in \mbox{\emph{int}}({\ds})$ satisfying $ \pi_\xi \circ dC(z)  \neq 0 $ is Fredholm regular.
\end{theorem}

The above result also holds for the set of somewhere injective curves in completed exact symplectic cobordisms $(W, J)$; see \cite[Theorem 7.2]{wendl-sft}.  Moreover, we have that Fredholm regularity implies that a neighborhood of a curve admits the structure of a smooth orbifold.

\begin{theorem}[Theorem 0, \cite{Wtrans}]\label{folk0}
Assume that $C\colon ({\ds},j) \to (W, J)$ is a non-constant curve in $\mathcal{M}^J (\gamma^+,\gamma^-).$   If $C$ is Fredholm regular, then a neighborhood of $C$ in $\mathcal{M}^J (\gamma^+,\gamma^-)$   naturally admits the structure of a smooth orbifold of dimension 
\[
{\mbox{\em ind}}(C) = -\chi(C)   + 2c_\tau(C) + \displaystyle \sum_{i=1}^k{CZ}_\tau(\gamma_i^+)  -  \sum_{j=1}^\ell {CZ}_\tau(\gamma_j^-),
\] 
whose isotropy group at $C$ is given by
\[
\mbox{\em Aut}:=\{ \varphi \in \mbox{\em Aut}({\ds}, j) \ | \ C = C \circ \varphi \}.
\] 
Moreover, there is a natural isomorphism
\[
T_C\mathcal{M}^J (\gamma^+,\gamma^-) = \ker D\deebar(j,C)/\mathfrak{aut}({\ds},j).
\]
\end{theorem}

\subsubsection{The relative first Chern number}

The trivialization $\tau$ determines a trivialization of $\xi|_C$ over the ends of $C$ up to homotopy.   We denote the set of homotopy classes of symplectic trivializations of the 2-plane bundle $\gamma^*\xi$ over $S^1$ by $\mathcal{T}(\gamma)$; this is an affine space over $\mathbb{Z}$.   After fixing trivializations $\tau_i^+ \in \mathcal{T}(\gamma_i^+)$ for each $i$ and $\tau_j^- \in \mathcal{T}(\gamma^-_j)$, we denote this set of trivialization choices by $\tau \in \mathcal{T}(\alpha,\beta)$. 

We are now ready to define the \emph{relative first Chern number} of the complex line bundle $\xi|_C$ with respect to the trivialization $\tau$, which we denote by
\[
c_\tau(C)  = c_1(\xi|_C,\tau).
\]
Let $\pi_Y:\R\times Y\to Y$ denote projection onto $Y$. We define  $c_1(\xi|_C,\tau)$ to be the algebraic count of zeros of a generic section $\psi$ of $\xi|_{[\pi_YC]}$ which on each end is nonvanishing and constant with respect to the trivialization on the ends.  In particular, given a class $Z\in H_2(Y,\alpha,\beta)$ we represent $Z$ by a smooth map $f: S \to Y$ where $S$ is a compact oriented surface with boundary.  Choose a section $\psi$ of $f^*\xi$ over $S$ such that $\psi$ is transverse to the zero section and $\psi$ is nonvanishing over each boundary component of $S$ with winding number zero with respect to the trivialization $\tau$.  We define 
\[
c_\tau(Z) : = \# \psi^{-1}(0),
\]
where `\#' denotes the signed count.

In addition to being well-defined, the relative first Chern class depends only on $\alpha, \beta, \tau,$ and $Z$.  If $Z' \in H_2(Y,\alpha,\beta)$ is another relative homology class then 
\begin{equation}
c_\tau(Z) - c_\tau(Z') = \langle c_1(\xi), Z - Z' \rangle.
\end{equation}

\begin{remark}[Change of trivializations]
We briefly clarify our sign convention for and definition of a change in trivialization up to homotopy of $\gamma^*\xi$. Given 
any two symplectic trivializations
\[
\tau_1, \ \tau_2: \gamma^*\xi \stackrel{\simeq}{\longrightarrow}  S^1 \times \R^2,
\]
we ``define"
\begin{equation}\label{degtrivs}
\tau_1 - \tau_2 := \mbox{deg} \left(\tau_2 \circ \tau_1^{-1} : S^1 \to \mbox{Sp}(2,\R) \approx S^1\right),
\end{equation}
where $\op{deg}  \left(\tau_2 \circ \tau_1^{-1} \right) $ is defined after homotoping $\tau_1$ and $\tau_2$ to unitary trivializations as follows.  Namely, given two unitary trivializations $\tilde{\tau}_j:E \to S^1 \times \R^{2n}$ for $j=1,2$  of a real rank $2n$ Hermitian vector bundle $(E,J,\omega)$, we denote   
\[
\op{deg}  \left(\tilde{\tau}_2 \circ \tilde{\tau}_1^{-1} \right) 
\]
to be the winding number of $\det g: S^1 \to \op{U}(n) \subset \C \setminus \{ 0 \}$, where $g: S^1 \to \op{U}(n)$ is the transition map appearing in the formula $\left(\tilde{\tau}_2 \circ \tilde{\tau}_1^{-1} \right) (t,v) = (t, g(t)v)$. (Note that in the setting under consideration in this paper, $n=1$.)

\end{remark}

Thus, given another collection of trivialization choices up to homotopy $\tau' = \left( \{ {\tau'}_i^+ \}, \{ {\tau'}_j^-\} \right)$ over the  Reeb currents and the convention \eqref{degtrivs}, we have
\begin{equation}\label{cherntriv}
c_\tau(Z) - c_{\tau'}(Z) = \sum_i m_i\left(\tau_i^+-{\tau}_i^{+'}\right) - \sum_j n_j \left(\tau_j^--{\tau}_j^{-'}\right).
\end{equation}

\subsubsection{The Conley-Zehnder index in dimension 3}\label{cz-sec}

Next we define the \emph{Conley-Zehnder index} of an embedded nondegenerate Reeb orbit $\gamma$ with respect to the trivialization $\tau$ up to homotopy. Pick a parametrization 
\[
\gamma: \R/T\Z \to Y.
\]
The choice of trivialization $\tau$ of $\xi$ over $\gamma$ is an isomorphism of symplectic vector bundles
\[
\tau: \gamma^*\xi \stackrel{\simeq}{\longrightarrow} (\R/T\Z) \times \R^2.
\]
Let $\{\varphi_t\}_{t\in \R}$ denote the one-parameter group of diffeomorphisms of $Y$ given by the flow of the Reeb vector field $R$.  With respect to $\tau$, the linearized flow $(d\varphi_t)_{t\in[0,T]}$ induces an arc of symplectic matrices $P:[0,T]\to \mbox{Sp}(2)$ defined by
\[
\label{symparc}
P_{t} = \tau(t) \circ d\varphi_t \circ \tau(0)^{-1}.
\]
To each arc of symplectic matrices $\{P_t\}_{t\in[0,T]}$ with $P_0=1$ and $P_T$ nondegenerate, there is an associated Conley-Zehnder index $CZ(\{P_t\}_{t\in[0,T]})\in\Z$.  We define the \emph{Conley-Zehnder index} of $\gamma$ with respect to $\tau$ by
\[
CZ_\tau(\gamma) = CZ\left(\{P_t\}_{t\in[0,T]}\right).
\]

 \begin{list}{\labelitemi}{\leftmargin=2em }
\item[\textbf{Elliptic case:}]
In the elliptic case, each trivialization is homotopic to one whose linearized flow $\{\varphi_t\}$ can be realized as a path of rotations.  If we take $\tau$ to be one of these trivializations so that each $\varphi_t$ is rotation by the angle $2 \pi \vartheta_t \in \R$ then $\vartheta_t$ is a continuous function of $t \in [0, T]$ satisfying $\vartheta_0=0$ and $\vartheta:=\vartheta_T \in \R / \Z $.  The number $\vartheta \in \R / \Z $ is called the \emph{rotation angle} of $\gamma$ with respect to the trivialization and
\[
{CZ}_\tau(\ga^k) = 2 \lfloor k \vartheta \rfloor + 1. 
\]
 \item[ \textbf{Hyperbolic case:}]
 Let $v \in \R^2$ be an eigenvector of $\phi_T$. Then for any trivialization used, the family of vectors $\{ \varphi_t(v) \}_{t \in [0,T]}$, rotates through angle $\pi r$ for some integer $r$.  The integer $r$ is dependent on the choice of trivialization $\tau$, but is always even in the positive hyperbolic case and odd in the negative hyperbolic case.  We obtain
 \[
CZ_\tau(\ga^k) = k r.
 \]
 \end{list}

The Conley-Zehnder index depends only on the Reeb orbit $\gamma$ and homotopy class of $\tau$ in the set of homotopy classes of symplectic trivializations of the 2-plane bundle $\gamma^*\xi$ over $S^1 = \R / T\Z$.  Given two trivializations $\tau_1$ and $\tau_2$ we have that
\begin{equation}\label{cztrivs}
CZ_{\tau_1}(\gamma^k) - CZ_{\tau_2}(\gamma^k) =2 k (\tau_2 - \tau_1 ),
\end{equation}
maintaining our sign convention \eqref{degtrivs}.

In our later expression of the ECH index, we will use the following shorthand for the total Conley-Zehnder index of a Reeb current $\alpha=\{ (\alpha_i, m_i)\}$:
\begin{equation}\label{CZ^I}
CZ^I_\tau(\alpha) = \sum_i \sum_{k=1}^{m_i}CZ_\tau(\alpha_i^k).
\end{equation}
Another set of trivialization choices $\tau'$ for $\alpha$ yields
\begin{equation}
CZ^I_\tau(\alpha) - \CZ^I_{\tau'}(\alpha) = \sum_i (m_i^2 + m_i)(\tau_i' - \tau_i).
\end{equation}

We will also abbreviate the following Conley-Zehnder contributions to the Fredholm index and ECH index for a curve $C \in \M(\alpha,\beta;J)$ by:
\begin{equation}\label{CZshorthand}
\begin{array}{rcl}
CZ^I_\tau(C) &=&CZ^I_\tau(\alpha) - CZ^I_\tau(\beta) \\
CZ^{ind}_\tau(C) &=&  \displaystyle \sum_{i=1}^k{CZ}_\tau(\alpha_i^{m_i})  -  \sum_{j=1}^\ell {CZ}_\tau(\beta_j^{n_j}).\\
\end{array}
\end{equation}


\subsection{The relative intersection pairing and relative adjunction formula}
The ECH index depends on the relative intersection pairing, which is related to the asymptotic writhe and linking number.  We review these notions now and summarize the relative adjunction formula.
\subsubsection{Asymptotic writhe and linking number}\label{s:writhe}
Given a somewhere injective curve $C \in \M^J(\gamma^+,\gamma^-)$, we consider the slice $u \cap ( \{s\} \times Y)$.  If $s \gg 0$, then the slice $u \cap ( \{s\} \times Y)$ is an embedded curve which is the union, over $i$, of a braid $\zeta^+_i$ around the Reeb orbit $\gamma^+_i$ with $m_i$ strands.  The trivialization $\tau$ can be used to identify the braid $\zeta_i^+$ with a link in $S^1 \times D^2$.   We identify $S^1 \times D^2$ with an annulus cross an interval, projecting $\zeta^+_i$ to the annulus, and require that the normal derivative along $\gamma^+_i$ agree with the trivialization $\tau$.

We define the \emph{writhe} of this link, which we denote by $w_\tau(\zeta_i^+) \in \Z$, by counting the crossings of the projection to $\R^2 \times \{ 0 \} $ with (nonstandard) signs.  Namely, we use the sign convention in which counterclockwise rotations in the $D^2$ direction as one travels counterclockwise around $S^1$ contribute positively.  

Analogously the slice $u \cap ( \{s\} \times Y)$ for $s \ll 0$ is the union over $j$ of a a braid $\zeta^-_j$ around the Reeb orbit $\beta_j$ with $n_j$ strands and we denote this braid's writhe by $w_\tau(\zeta^-_i) \in \Z$.

The writhe depends only on the isotopy class of the braid and the homotopy class of the trivialization $\tau$.  If $\zeta$ is an $m$-stranded braid around an embedded nondegenerate Reeb orbit $\gamma$ and  $\tau' \in \mathcal{T}(\gamma)$ is another trivialization then
\[
w_\tau(\zeta) - w_{\tau'}(\zeta) = m(m-1)(\tau' - \tau)
\]
because shifting the trivialization by one adds a full clockwise twist to the braid.

If $\zeta_1$ and $\zeta_2$ are two disjoint braids around an embedded Reeb orbit $\gamma$ we can define their \emph{linking number} $\ell_\tau(\zeta_1,\zeta_2) \in \Z$
to be the linking number of their oriented images in $\R^3$.  The latter is by definition one half of the signed count of crossings of the strand associated to $\zeta_1$ with the strand associated to $\zeta_2$ in the projection to $\R^2 \times \{0\}$.  If the braid $\zeta_k$ has $m_k$ strands then a change in trivialization results in the following formula
\[
\ell_\tau(\zeta_1,\zeta_2) - \ell_{\tau'}(\zeta_1,\zeta_2) = m_1m_2 (\tau'-\tau).
\]
The writhe of the union of two braids can be expressed in terms of the writhe of the individual components and the linking number:
\begin{equation}\label{eqn:writhelinking}
w_\tau(\zeta_1 \cup \zeta_2) = w_\tau(\zeta_1) + w_\tau(\zeta_2) + 2\ell_\tau(\zeta_1,\zeta_2).
\end{equation}
If $\zeta$ is a braid around an embedded Reeb orbit $\gamma$ which is disjoint from $\gamma$ we define the \emph{winding number} to be the linking number of $\zeta$ with $\gamma$:
\[
\eta_\tau(\zeta) :=\ell_\tau(\zeta,\gamma) \in \Z.
\]

We have the following bound on writhe \cite[\S 5.1]{Hu2}.

\begin{proposition}\label{prop:writheCZbound}
If $C \in \M(\alpha,\beta;J)$ is somewhere injective then,
\[
w_\tau(C) \leq  CZ^I_\tau(C) - CZ^{ind}_\tau(C),
\]
where the Conley-Zehnder shorthand is given by \eqref{CZshorthand}.
\end{proposition}
\subsubsection{Admissible representatives}

In order to speak more ``globally" of writhe and winding numbers associated to a curve, we need the following notion of an admissible representative for a class $Z \in H_2(Y,\alpha,\beta)$, as in \cite[Def. 2.11]{Hrevisit}.  Given $Z\in H_2(Y,\alpha,\beta)$ we define an \emph{admissible representative} of $Z$ to be a smooth map $f: S \to [-1,1] \times Y$, where $S$ is an oriented compact surface such that
\begin{enumerate}
\item The restriction of $f$ to the boundary $\partial S$ consists of positively oriented  covers of $\{ 1\} \times \alpha_i$ whose total multiplicity is $m_i$ and negatively oriented covers of $\{ -1\} \times \beta_j$ whose total multiplicity is $n_j$.
\item The projection $\pi_Y: [-1,1] \times Y \to Y$ yields $[\pi(f(S))]=Z$. 
\item The restriction of $f$ to $\mbox{int}(S)$ is an embedding and $f$ is transverse to $\{-1,1\} \times Y$.
\end{enumerate}
If $S$ is an admissible representative for any $Z\in H_2(Y,\alpha,\beta)$ then we say $S$ is an \emph{admissible surface}.

The utility of the notion of an admissible representative $S$ for $Z$ can be seen in the following. For $\varepsilon >0$ sufficiently small, $S \cap ( \{ 1 - \varepsilon \} \times Y)$ consists of braids $\zeta_i^+$ with $m_i$ strands in disjoint tubular neighborhoods of the Reeb orbits $\alpha_i$, which are well defined up to isotopy.  Similarly, $S \cap ( \{ -1+ \varepsilon \} \times Y)$ consists of braids $\zeta_j^-$ with $n_j$ strands in disjoint tubular neighborhoods of the Reeb orbits $\alpha_i$, which are well defined up to isotopy.  

Thus an admissible representative of $Z\in H_2(Y;\alpha,\beta)$ permits us to define the \emph{total writhe} of a curve interpolating between the  Reeb currents $\alpha$ and $\beta$ by 
\[
w_\tau(S) = \sum_i w_{\tau_i^+}(\zeta_i^+) -  \sum_j w_{\tau_j^-}(\zeta_j^-).
\]
Here $\zeta_i^+$ are the braids with $m_i$ strands in a neighborhood of each of the $\alpha_i$ obtained by taking the intersection of $S$ with $\{ s \} \times Y$ for $s$ close to 1.  Similarly, the $\zeta_j^-$ are the braids with $n_j$ strands in a neighborhood of each of the $\beta_j$ obtained by taking the intersection of $S$ with $\{ s \} \times Y$ for $s$ close to $-1$.  Bounds on the writhe in terms of the Conley-Zehnder index are given in \cite[\S 3.1]{dc}, which relates the asymptotic behavior of pseudoholomorphic curves, extensively explored by Hutchings, cf. \cite[\S 5.1]{Hu2}.  We review and make use of these writhe bounds as well as their refinements obtained by \cite{CGHZ} in \S \ref{sec:connectors} and \ref{handleslides}. 

Taking a similar viewpoint with regard to the linking number results in the following formula.  If $S'$ is an admissible representative of $Z' \in H_2(Y,\alpha',\beta')$ such that the interior of $S'$ does not intersect the interior of $S$ near the boundary, with braids $\zeta_i^{+ '}$ and $\zeta_j^{- '}$ we can define the \emph{linking number of $S$ and $S'$} to be
\[
\ell_\tau(S,S') := \sum_i \ell_\tau(\zeta_i^+,\zeta_i^{+ '}) - \sum_j \ell_\tau(\zeta_j^-,\zeta_j^{- '}).
\]
Above the Reeb currents $\alpha$ and $\alpha'$ are both indexed by $i$, so sometimes $m_i$ or $m_{i}'$ is 0, similarly both $\beta$ and $\beta'$ are indexed by $j$ and sometimes $n_j$ or $n_{j}'$ is 0.  The trivialization $\tau$ is a trivialization of $\xi$ over all Reeb orbits in the sets $\alpha, \alpha', \beta,$ and $\beta'$.

\subsubsection{The relative intersection pairing}

The relative intersection pairing can be defined using an admissible representative, which is more general than the notion of a $\tau$-representative \cite[Def. 2.3]{Hindex}, as the latter uses the trivialization to control the behavior at the boundary.  Consequently, \emph{we see an additional linking number term appear in the expression of the relative intersection pairing when we use an admissible representative. }   
Let  $S$ and $S'$ be two surfaces which are admissible representatives of $Z \in H_2(Y,\alpha,\beta)$ and $Z'\in H_2(Y,\alpha',\beta')$  whose interiors $\dot{S}$ and $\dot{S}'$ are transverse and do not intersect near the boundary.  We define 
the \emph{relative intersection pairing} by the following signed count
\begin{equation}
Q_\tau(Z,Z'):= \# \left( \dot{S} \cap \dot{S}'\right) - \ell_\tau(S,S').
\end{equation}
Moreover, $Q_\tau(Z,Z')$ is an integer which depends only on $\alpha, \beta, Z, Z'$ and $\tau$.  If $Z=Z'$ then we write $Q_\tau(Z):=Q_\tau(Z,Z)$.

For another collection of trivialization choices $\tau'$, 
\[
Q_\tau(Z,Z') - Q_{\tau'}(Z,Z') = \sum_i m_i m_i' (\tau_i^+ - \tau_i^{+'}) - \sum_i n_j n_j' (\tau_i^- - \tau_i^{-'}).
\]
We recall how \cite[\S 3.5]{Hrevisit} permits us to compute the relative intersection pairing using embedded surfaces in $Y$.  An admissible representative $S$ of $Z \in H_2(Y,\alpha,\beta)$ is said to be \emph{nice} whenever the projection of $S$ to $Y$ is an immersion and the projection of the interior $\dot{S}$ to $Y$ is an embedding which does not intersect the $\alpha_i$'s or $\beta_j$'s. Lemma 3.9 from  \cite{Hrevisit} establishes that if none of the $\alpha_i$ equal the $\beta_j$ then every class $Z \in H_2(Y,\alpha,\beta)$ admits a nice admissible representative.

If $S$ is a nice admissible representative of $Z$ with associated braids $\zeta_i^+$ and $\zeta_j^-$ then we can define the \emph{winding number} by
\[
\eta_\tau(S) : = \sum_i \eta_{\tau_i^+}\left(\zeta_i^+\right) -\sum_j \eta_{\tau_j^-}\left(\zeta_j^-\right).
\]

\begin{lemma}[Lemma 3.9 \cite{Hrevisit}]
Suppose that $S$ is a nice admissible representative of $Z$.  Then
\[
Q_\tau(Z) = -w_\tau(S) - \eta_\tau(S).
\]
\end{lemma}

\subsubsection{The relative adjunction formula}

In this section we review the relative adjunction formulas of interest, which are used later in \S \ref{sec:connectors}.  This is taken from \cite[\S 3]{Hindex} and is stated for pseudoholomorphic curves interpolating between Reeb currents $\alpha$ and $\beta$ in symplectizations.   As explained in \cite[\S 4.4]{Hrevisit} the proof carries over in a straightforward manner to exact symplectic cobordisms.  
\begin{lemma}
\label{lem:adjunction}
Let $u\in \mathcal{M}^J(\alpha,\beta)$ be somewhere injective, $S$ be a representative of $Z\in H_2(Y,\alpha,\beta)$, and $\tau\in \mathcal{T}(\alpha,\beta)$.  Let $N_S$ be the normal bundle to $S$.  
\begin{enumerate}[{\em (i)}]
\item If $u$ is further assumed to be embedded everywhere then
\begin{equation}\label{relcalc}
c_1^{\tau}(Z) = \chi(S) + c_1^\tau(N_S).
\end{equation}
\item For general embedded representatives $S$, e.g. ones not necessarily coming from pseudoholomorphic curves, \eqref{relcalc} holds mod 2 and
\begin{equation}\label{chernwQ}
c_1^\tau(N_S) = w_\tau(S) + Q_\tau(Z,Z).
\end{equation}
\item If $u$ is embedded except at possibly finitely many singularities then
\begin{equation}\label{relsingeq}
c_1^\tau(Z) = \chi(u) + Q_\tau(Z) + w_\tau(u) - 2\delta(u), 
\end{equation}
where $\delta(u)$ is a sum of positive integer contributions from each singularity.
\end{enumerate}
\end{lemma}

\begin{remark}
If $u$ is a closed pseudoholomorphic curve, then there is no writhe term or trivialization choice, and \eqref{relsingeq} reduces to the usual adjunction formula
\[
\langle c_1(TW),[u]\rangle = \chi(u) + [u] \cdot [u] - 2\delta(u).
\]
\end{remark}


\subsection{The ECH index}

We are now ready to give the definition of the ECH index.

\begin{definition}[ECH index]
Let $\alpha=\{(\alpha_i,m_i)\}$ and $\beta=\{(\beta_j,n_j)\}$ be Reeb currents in the same homology class, $\sum_i m_i [\alpha_i]=\sum_j n_j [\beta_j]=\Gamma\in H_1(Y).$   Given $Z \in H_2(Y,\alpha,\beta)$. We define the \emph{ECH index} to be
\[
I(\alpha,\beta,Z) = c_1^\tau(Z) + Q_\tau(Z) +  CZ^I_\tau(\alpha) -CZ^I_\tau(\beta).
\]
where $CZ^I$ is the shorthand defined in \eqref{CZ^I}.   When $\alpha$ and $\beta$ are clear from context, we use the notation $I(Z)$.
\end{definition}

The Chern class term is linear in the multiplicities of the Reeb currents and the relative intersection term is quadratic. The ``total Conley-Zehnder" index term $CZ_\tau^I$ typically behaves in a complicated way with respect to the multiplicities.   We also have the following general properties of the ECH index, as proven in \cite[\S 3.3]{Hindex}.

\begin{theorem}\label{thm:Iproperties}The ECH index has the following basic properties: \hfill

\begin{enumerate}[{\em (i)}]
\item {\em(Well Defined)} The ECH index $I(Z)$ is independent of the choice of trivialization.
\item\label{property:indexamb} {\em (Index Ambiguity Formula) }If $Z' \in H_2(Y,\alpha,\beta)$ is another relative homology class, then
\[
I(Z)-I(Z')= \langle Z-Z', c_1(\xi) + 2 \mbox{\em PD}(\Gamma) \rangle.
\]
\item {\em (Additivity) } If $\delta$ is another Reeb current in the homology class $\Gamma$, and if $W \in H_2(Y,\beta,\delta)$, then $Z+W \in  H_2(Y,\alpha,\delta)$ is defined and 
\[
I(Z+W)=I(Z)+I(W).
\]
\item {\em (Index Parity) }If $\alpha$ and $\beta$ are generators of the ECH chain complex, then 
\[
(-1)^{I(Z)} = \varepsilon(\alpha)\varepsilon(\beta),
\]
where $\varepsilon(\alpha)$ denotes $-1$ to the number of positive hyperbolic orbits in $\alpha$.
\end{enumerate}
\end{theorem}
To learn more about the wonders of the ECH index see \cite[\S 2]{Hrevisit}.

\begin{remark}
We will also use the notation $I(\alpha,\beta,\mathcal{C})$ and $I(\mathcal{C})$ for $\mathcal{C}=\{ (C_k, d_k) \}\in\M^J(\alpha,\beta)$ to denote $I(\alpha,\beta,[\mathcal{C}])$, where $[\cdot]$ denotes equivalence in $H_2(Y,\alpha,\beta)$, as well as $c_\tau(\mathcal{C})$ and $Q_\tau(\mathcal{C})$, following \cite[Notation 4.7]{Hrevisit}. In addition, we will occasionally use the notation $Q_\tau(S)$ for $S$ an admissible surface in $[-1,1]\times Y$ to denote $Q_\tau([\pi_YS])$ (recall that we are using the notation $\pi_Y$ to denote both the projections from $\R\times Y$ and from $[-1,1]\times Y$ to $Y$).
  \end{remark}

The ECH index inequality (cf. \S \ref{ECH-ineq}, \cite[Theorem 4.15]{Hrevisit}) permits the following results regarding low ECH index curves.  A {\em trivial cylinder} is a cylinder $\R \times \gamma \subset \R \times Y$ where $\gamma$ is an embedded Reeb orbit.

\begin{proposition}[Prop. 3.7 \cite{Hu2}]\label{lowiprop}
Suppose $J$ is generic.  Let $\alpha$ and $\beta$ be Reeb currents and let $\mathcal{C} \in \M^J(\alpha,\beta)$ be any $J$-holomorphic current in $\R \times Y$, not necessarily somewhere injective.  Then:
\begin{enumerate}[{\em (i)}]
\item We have $I(\mathcal{C})\geq 0$ with equality if and only if $\mathcal{C}$ is a union of trivial cylinders with multiplicities.
\item If $I(\mathcal{C}) =1$ then $\mathcal{C}= \mathcal{C}_0 \sqcup C_1$,  where $I(\mathcal{C}_0)=0$, and $C_1$ is embedded and has $\ind(C_1)=I(C_1)=1$.
\item If $I(\mathcal{C}) =2$, and if $\alpha$ and $\beta$ are generators of the chain complex $ECC_*(Y,\lambda,\Gamma,J)$, then $\mathcal{C}= \mathcal{C}_0 \sqcup C_2$, where $I(\mathcal{C}_0)=0$, and $C_2$ is embedded and has $\ind(C_2)=I(C_2)=2$.
\end{enumerate}
\end{proposition}


\subsection{The ECH partition conditions and index inequality}\label{ECH-ineq}

The ECH partition conditions are a topological type of data associated to the pseudoholomorphic curves (and currents) which can be obtained indirectly from certain ECH index relations.  In particular, the covering multiplicities of the Reeb orbits at the ends of the nontrivial components of the pseudoholomorphic curves (and currents) are uniquely determined by the trivial cylinder component information.  The genus can be determined by the current's relative homology class.

\begin{definition}\label{def:part} \cite[\S 3.9]{Hu2} Let $\gamma$ be an embedded Reeb orbit and $m$ a positive integer.  We define two partitions of $m$, the \emph{positive partition} $P^+_\gamma(m)$ and  the \emph{negative partition} $P^-_\gamma(m)$\footnote{Previously the papers \cite{Hindex, Hrevisit} used the terminology incoming and outgoing partitions.}  as follows.
\begin{itemize}
\item If $\gamma$ is positive hyperbolic, then
\[
P_\gamma^+(m): = P_\gamma^-(m): = (1,...,1).
\]
\item If $\gamma$ is negative hyperbolic, then
\[
P_\gamma^+(m): = P_\gamma^-(m): = \left\{ \begin{array}{ll}
(2,...,2) & m \mbox{ even,} \\
(2,...,2,1) & m \mbox{ odd. } \\
\end{array}
\right .
\]
\item If $\gamma$ is elliptic then the partitions are defined in terms of the quantity $\vartheta \in \R / \Z $ for which $\czm^\tau(\gamma^k) = 2\lfloor k \vartheta \rfloor + 1$.  We write 
\[
P_\gamma^\pm(m): = P_\vartheta^\pm(m),
\]
with the right hand side defined as follows. 

 Let $\Lambda^+_\vartheta(m)$ denote the highest concave  polygonal path in the plane that starts at $(0,0)$, ends at  $(m,\lfloor m \vartheta \rfloor)$, stays below the line $y = \vartheta x$, and has corners at lattice points.  Then the integers $P^+_\vartheta(m)$ are the horizontal displacements of the segments of the path $\Lambda^+_\vartheta(m)$ between the lattice points.

Likewise, let  $\Lambda^-_\vartheta(m)$ denote the  lowest convex polygonal path in the plane that starts at $(0,0)$, ends at $(m,\lceil m \vartheta \rceil)$, stays above the line $y = \vartheta x$, and has corners at lattice points.  Then the integers $P^-_\vartheta(m)$ are the horizontal displacements of the segments of the path $\Lambda^-_\vartheta(m)$ between the lattice points.

Both $P_\vartheta^\pm(m)$ depend only on the class of $\vartheta$ in $\R / \Z$.  Moreover, $P_\vartheta^+(m) = P_{-\vartheta}^-(m)$.
\end{itemize}

\end{definition}

\begin{example}\label{ex:partitions}
If the rotation angle for an elliptic orbit $\gamma$ satisfies $\vartheta \in (0,1/m)$ then
\[
\begin{array}{lcl}
P_\vartheta^+(m) &=& (1,...,1) \\
P_\vartheta^-(m) &=& (m). \\
\end{array}
\]
The partitions are quite complex for other $\vartheta$ values, see \cite[Fig. 1]{Hu2}.
\end{example}

We end this section with the ECH index inequality \cite[Theorem 4.15]{Hrevisit} in symplectic cobordisms.  As before we take $\alpha=\{(\alpha_i,m_i)\}$ and $\beta=\{(\beta_j,n_j)\}$ to be Reeb currents in the same homology class.  Let $C \in \mathcal{M}^J(\alpha,\beta)$.  For each $i$ let $a_i^+$ denote the number of positive ends of $C$ at $\alpha_i$ and let $\{ q_{i,k}^+\}_{k=1}^{a_i^+}$ denote their multiplicities.  Thus $\sum_{k=1}^{a_i^+} q_{i,k}^+=m_i$.   Likewise, for each $j$ let $b_i^-$ denote the number of negative ends of $C$ at $\beta_j$ and let $\{ q_{j,k}^-\}_{k=1}^{b_j^-}$ denote their multiplicities; we have $\sum_{k=1}^{b_j^-} q_{j,k}^-=n_j$.

\begin{theorem}[ECH index inequality]\label{thm:indexineq} Suppose $C \in \mathcal{M}^J(\alpha,\beta)$ is somewhere injective.  Then \[ \mbox{\em ind}(C) \leq I(C) - 2 \delta(C). \] Equality holds only if $\{  q_{i,k}^+ \} = P_{\alpha_i}^+(m_i)$ for each $i$ and $\{  q_{j,k}^- \} = P_{\beta_j}^-(n_j)$ for each $j$. \end{theorem}

\subsection{The ECH differential and grading}

If $\alpha$ and $\beta$ are Reeb currents and $k$ is an integer, define
\[ 
\M_k^J(\alpha,\beta) = \{ \CC \in \M^J(\alpha, \beta) \ | \ I(\CC) =k\}.
\]
If $\alpha$ is a generator of the ECH chain complex we define the differential $\partial$ on $ECC_*(Y,\lambda,\Gamma, J)$ by
\[
\partial \alpha = \sum_{\beta} \# (\M_1^J(\alpha,\beta)/\R ) \beta.
\]
The sum is over chain complex generators $\beta$, and `$\#$' denotes the mod 2 count.  Here $\R$ acts on $\M_1^J(\alpha,\beta)$ by translation of the $\R$-coordinate on $\R \times Y$.  By Proposition \ref{lowiprop} the quotient  $\M_1^J(\alpha,\beta)/\R$ is a discrete set.  By the arguments in \cite[\S 5.3]{Hu2} we can conclude that  $\M_1^J(\alpha,\beta)/\R$ is finite.  Finally, because the differential decreases the symplectic action and since any nondegenerate contact form has only finitely many Reeb orbits with bounded symplectic action, we have that the ECH differential is well-defined. 

Proving that $\partial^2=0$ is a substantial undertaking, see \cite{obg1} and \cite{obg2}. Moreover, Taubes' proof in \cite{taubesechswf1} that the homology of $(ECC_*(Y,\lambda,\Gamma,J),\partial)$ is independent of $J$ and of $\lambda$ up to $c_1(\xi)$ requires Seiberg-Witten theory. In light of this invariance we denote this homology by $ECH_*(Y,\xi,\Gamma)$, and call it the \emph{embedded contact homology}, or ECH, of $(Y,\xi,\Gamma)$.

\begin{remark}\label{rem:grading}
Since the ECH index depends on a choice of relative second homology class $Z$, for general $(Y,\xi,\Gamma)$ we can only expect a relative $\Z_d$ grading, where $d$ is the divisibility of the class $c_1(\xi) + 2 \mbox{\em PD}(\Gamma)$ in $H^2(Y;\Z)$ mod torsion. This is because of the index ambiguity property of the ECH index, Theorem \ref{thm:Iproperties} (\ref{property:indexamb}), whereby we set
\[
I(\alpha,\beta):=[I(\alpha,\beta,Z)]\in\Z_d.
\]

In the setting of prequantization bundles, we will have $d=0$, because $c_1(\xi) + 2 \mbox{\em PD}(\Gamma)$ is torsion, as we now explain. If $g=0$, then $H^2(Y)\cong H_1(Y)$ is torsion, as shown in \S\ref{subsec:homology}. If $g>0$, then we will demonstrate in Lemma \ref{lem:indepZ} that the divisibility of $c_1(\xi)$ is zero.  In \S\ref{subsec:PQBgens} we explain why we only need to consider $\Gamma$ which are a multiple of the fiber class. The fiber class is torsion as we will show in \S\ref{subsec:homology}. Thus we can work with a relative $\Z$ grading.

We often refine our relative $\Z$ grading to an absolute $\Z$ grading by setting a chosen generator to have grading zero. In particular, if $\Gamma=[\emptyset]$, we will choose $\emptyset$ to have grading zero.
\end{remark}

\subsection{Filtered ECH}

There is an action filtration on ECH which enables us to compute it via successive approximations, as explained in \S\ref{subsec:PQBgens} and \S\ref{subsec:directlimit}. The \emph{symplectic action} or \emph{length} of an Reeb current $\alpha=\{(\alpha_i,m_i)\}$ is
\[
\mathcal{A}(\alpha):=\sum_im_i\int_{\alpha_i}\lambda.
\]
If $J$ is $\lambda$-compatible and there is a $J$-holomorphic current from $\alpha$ to $\beta$, then $\mathcal{A}(\alpha)\geq\mathcal{A}(\beta)$ by Stokes' theorem, since $d\lambda$ is an area form on such $J$-holomorphic curves. Since $\partial$ counts $J$-holomorphic currents, it decreases symplectic action, i.e.,
\[
\langle\partial\alpha,\beta\rangle\neq0\Rightarrow\mathcal{A}(\alpha)\geq\mathcal{A}(\beta).
\]

Let $ECC_*^L(Y,\lambda,\Gamma;J)$ denote the subgroup of $ECC_*(Y,\lambda,\Gamma;J)$ generated by Reeb currents of symplectic action less than $L$. Because $\partial$ decreases action, it is a subcomplex. It is shown in \cite[Theorem 1.3]{cc2} that the homology of $ECC_*(Y,\lambda,\Gamma;J)$ is independent of $J$, therefore we denote its homology by $ECH_*^L(Y,\lambda,\Gamma)$, which we call \emph{filtered ECH}.

Given $L<L'$, there is a homomorphism
\[
\iota^{L,L'}:ECH_*^L(Y,\lambda,\Gamma)\to ECH_*^{L'}(Y,\lambda,\Gamma),
\]
induced by the inclusion $ECC_*^L(Y,\lambda,\Gamma;J)\hookrightarrow ECC_*^{L'}(Y,\lambda,\Gamma;J)$ and independent of $J$, as shown in \cite[Theorem 1.3]{cc2}. The $\iota^{L,L'}$ fit together into a direct system $(\{ECC_*^L(Y,\lambda,\Gamma)\}_{L\in\R},\iota^{L,L'})$. Because taking direct limits commutes with taking homology, we have
\[
ECH_*(Y,\lambda,\Gamma)=H_*\left(\lim_{L\to\infty}ECC^L_*(Y,\lambda,\Gamma;J)\right)=\lim_{L\to\infty}ECH^L_*(Y,\lambda,\Gamma).
\]

The inclusion maps are compatible with certain cobordism maps as follows. An \textit{exact symplectic cobordism} from $(Y_+,\lambda_+)$ to $(Y_-,\lambda_-)$ is a pair $(X,\lambda)$ where $X$ is a four-manifold with $\partial X=Y_+-Y_-$ and $\lambda$ a one-form on $X$ with $d\lambda$ symplectic and $\lambda|_{Y_\pm}=\lambda_\pm$.

Define
\begin{equation}\label{eqn:ECHdecompH1}
ECH_*(Y,\lambda):=\bigoplus_{\Gamma\in H_1(Y)}ECH_*(Y,\lambda,\Gamma),
\end{equation}
and define $ECH^L_*(Y,\lambda)$ similarly.

Exact symplectic cobordisms induce maps on filtered ECH. The properties of these cobordism maps will allow us to compute the ECH of prequantization bundles via a nondegenerate perturbation of the contact form (up to large action), as discussed in \S\ref{sec:ECHI} and \S\ref{sec:finalcomp}. For now we state the results from \cite{cc2} on these cobordism maps which we will need in order to understand \S\ref{sec:ECHI}-\S\ref{handleslides}; a more in-depth discussion and detour through Seiberg-Witten theory is postponed to \S\ref{sec:finalcomp}. We do not need the following two results in their full strength, and so we paraphrase below. In particular, we do not need the full notion of ``composition" of exact symplectic cobordisms, so we cite the composition property in the following theorem only in the case which we will need. The fact that if $\varepsilon'<\varepsilon$ then $([\varepsilon',\varepsilon]\times Y,(1+s\fp^*H)\lambda)$ (where $s$ is the coordinate on $[\varepsilon',\varepsilon]$) is an exact symplectic cobordism from $(Y,\lambda_\varepsilon)$ to $(Y,\lambda_{\varepsilon'})$ is addressed in the proof of Proposition \ref{prop:directlimitcomputesfiberhomology} in \S\ref{subsec:PQBgens}.

\begin{theorem}[{\cite[Theorem 1.9, Remark 1.10]{cc2}}]\label{thm:cc2cobmaps} Let $\lambda_\pm$ be contact forms on closed, oriented, connected three-manifolds $Y_\pm$, with $(X,\lambda)$ an exact symplectic cobordism from $(Y_+,\lambda_+)$ to $(Y_-,\lambda_-)$. There are maps of ungraded $\Z_2$-modules:
\[
\Phi^L(X,\lambda):ECH_*^L(Y_+,\lambda_+)\to ECH^L_*(Y_-,\lambda_-),
\]
for each real number $L$ such that:

(Inclusion) If $L<L'$ then the following diagram commutes:
\begin{equation}\label{eqn:ECHcd}
\xymatrixcolsep{3pc}\xymatrix{
ECH^L(Y,\lambda_+,\Gamma) \ar[r]^{\Phi^L(X,\lambda)} \ar[d]_{\iota^{L,L'}} & ECH^L(Y,\lambda_-,\Gamma) \ar[d]^{\iota^{L,L'}}
\\ECH^{L'}(Y,\lambda_+,\Gamma) \ar[r]_{\Phi^{L'}(X,\lambda)} & ECH^{L'}(Y,\lambda_-,\Gamma)
}
\end{equation}

(Composition) Given $\varepsilon''<\varepsilon'<\varepsilon$,
\[
\Phi^L([\varepsilon'',\varepsilon]\times Y,(1+s\fp^*H)\lambda)=\Phi^L([\varepsilon'',\varepsilon']\times Y,(1+s\fp^*H)\lambda)\circ\Phi^L([\varepsilon',\varepsilon]\times Y,(1+s\fp^*H)\lambda).
\]

Furthermore, the maps $\Phi^L(X,\lambda)$ respect the decomposition (\ref{eqn:ECHdecompH1}) in the following sense: the image of $ECH_*(Y_+,\lambda_+,\Gamma_+)$ has a nonzero component in $ECH_*(Y_-,\lambda_-,\Gamma_-)$ only if $\Gamma_\pm\in H_1(Y_\pm)$ map to the same class in $H_1(X)$.
\end{theorem}

In order to understand the impact of these cobordism maps on the chain level in certain well-behaved cobordisms, we will also use the simplification expressed in the following lemma, which encapsulates two technical lemmas and a definition from \cite{cc2}.

\begin{lemma}[{\cite[Lemma 3.4 (d), Lemma 5.6, and Definition 5.9]{cc2}}]\label{lem:nicecobmap} Given a real number $L$, let $\lambda_t$ and $J_t$ be smooth 1-parameter families of contact forms on $Y$ and $\lambda_t$-compatible almost complex structures such that
\begin{itemize}
\item The contact forms $\lambda_t$ are of the form $f_t\lambda_0$, where $f:[0,1]\times Y\to\R_{>0}$ satisfies $\frac{\partial f}{\partial_t}<0$ everywhere.
\item All Reeb orbits of each $\lambda_t$ of length less than $L$ are nondegenerate, and there are no Reeb currents of $\lambda_t$ of action exactly $L$ (This condition is referred to in \cite{cc2} as $\lambda_t$ being ``$L$-nondegenerate.")
\item Near each Reeb orbit of length less than $L$ the pair $(\lambda_t,J_t)$ satisfies the conditions of \cite[(4-1)]{taubesechswf1}. (This condition is referred to in \cite{cc2} as $(\lambda_t,J_t)$ being ``$L$-flat.")
\item For Reeb currents of action less than $L$, the ECH differential $\partial$ is well-defined on admissible Reeb currents of action less than $L$ and satisfies $\partial^2=0$. (This is a condition on the genericity of $J_t$ described in \cite{obg2}, and referred to in \cite{cc2} as $J_t$ being ``$ECH^L$-generic.")
\end{itemize}
Then $([-1,0]\times Y,\lambda_{-t})$ is an exact symplectic cobordism from $(Y,\lambda_0)$ to $(Y,\lambda_1)$, and for all $\Gamma\in H_1(Y)$, the cobordism map $\Phi^L([-1,0]\times Y,\lambda_{-t})$ agrees with the map
\[
ECC^L_*(Y,\lambda_0,\Gamma;J_0)\to ECC^L_*(Y,\lambda_1,\Gamma;J_1),
\]
determined by the canonical bijection on generators.

\end{lemma}

\section{The ECH index for prequantization bundles}\label{sec:ECHI}

In this section we compute the ECH index for certain Reeb currents in perturbations of prequantization bundles. The canonical contact form associated to a prequantization bundle is degenerate, so it is not possible to compute its ECH directly. Instead we introduce the perturbation
\begin{equation}\label{eqn:lambdapert}
\lambda_\varepsilon=(1+\varepsilon\fp^*H)\lambda,
\end{equation}
where $H$ is a perfect Morse function on the base $\Sigma_g$ which is $C^2$ close to 1. As explained in \S\ref{subsec:pertReeb}, we have:
\begin{lemma}\label{lem:efromL}
Fix a Morse function $H$ such that $H$ is ${C^2}$ close to 1.
\begin{enumerate}
\item[(i)] For each $L>0$, there exists $\varepsilon(L)>0$ such that for all $\varepsilon<\varepsilon(L)$, all Reeb orbits with $\mathcal{A}(\gamma) < L$ are nondegenerate and project to critical points of $H$, where $\mathcal{A}$ is computed using $\lambda_\varepsilon$.
\item[(ii)] The action of a Reeb orbit $\gamma_p^k$ of $R_\varepsilon$ over a critical point $p$ of $H$ is proportional to the length of the fiber, namely
\[
\mathcal{A}(\gamma_p^k) = \int_{\gamma_p^k} \lambda_\varepsilon = 2k \pi (1+\varepsilon \fp^*H).
\]
\end{enumerate}
\end{lemma}

By Lemma \ref{lem:efromL} (i), it is possible to choose $\varepsilon(L)$ based on $L$ so that the embedded orbits contributing to the generators of $ECC^L_*(Y,\lambda_{\varepsilon(L)},\Gamma;J)$ consist only of fibers above critical points of $H$. In the proof of Lemma \ref{lem:efromL} in \S\ref{subsec:pertReeb} we show that $\varepsilon(L)\sim\frac{1}{L}$. Therefore, by Lemma \ref{lem:efromL} (ii), we can choose $L$ large enough so that the generators of $ECC^L_*(Y,\lambda_{\varepsilon(L)}),\Gamma;J)$ include any given Reeb current whose embedded Reeb orbits are fibers above critical points.

To capture all these filtered complexes, we prove the following result in \S\ref{subsec:PQBgens}:

\begin{proposition}\label{prop:directlimitcomputesfiberhomology} With $Y,\lambda$, and $\varepsilon(L)$ as discussed above, for any $\Gamma$, there is a direct system consisting of all the $ECH^L_*(Y,\lambda_{\varepsilon(L)},\Gamma)$. The direct limit $\lim_{L\to\infty}ECH^L_*(Y,\lambda_{\varepsilon(L)},\Gamma)$ is the homology of the chain complex generated by Reeb currents $\{(\alpha_i,m_i)\}$ where the $\alpha_i$ are fibers above critical points of $H$ and $\sum_im_i[\alpha_i]=\Gamma$.
\end{proposition}

Proposition \ref{prop:directlimitcomputesfiberhomology} will allow us to prove Theorem \ref{thm:mainthm} in \S\ref{sec:finalcomp}, where we will relate the direct limit to the ECH of the original prequantization bundle.  From the Conley-Zehnder index computations in in \S\ref{subsec:CZ} we have:


\begin{lemma}\label{lem:orbitseh} The fibers above the index zero and two critical points of $H$ are embedded elliptic orbits, while the fibers above the index one critical points of $H$ are embedded positive hyperbolic orbits.
\end{lemma}

\begin{remark}[Notation] {For technical reasons}, we need to assume $H$ is perfect. We denote the corresponding embedded Reeb orbits by $e_-, e_+$, and $h_i$, respectively, and throughout the rest of the paper consider generators of the form $e_-^{m_-}h_1^{m_1}\cdots h_{2g}^{m_{2g}}e_+^{m_+}$ where $m_\pm,m_i\in\Z_\geq0$, denoting the orbit set
\[
\{(e_-,m_-),(h_1,m_1),\dots,(h_{2g},m_{2g}),(e_+,m_+)\}.
\]
When specifying a particular orbit set with multiplicative notation, we will follow the convention that $m_\pm, m_i>0$, omitting the term $\gamma^m$ if $m=0$. When using multiplicative notation to denote an unspecified or general orbit set, however, we will allow $m_\pm,m_i=0$, and it will correspond to the orbit set in the usual notation with the pair $(e_\pm,m_\pm)$ or $(h_i,m_i)$ removed.
\end{remark}

Given Reeb currents $\alpha=e_-^{m_-}h_1^{m_1}\cdots h_{2g}^{m_{2g}}e_+^{m_+}$ and $\beta=e_-^{n_-}h_1^{n_1}\cdots h_{2g}^{n_{2g}}e_+^{n_+}$, let
\[
d=\frac{M-N}{|e|},
\]
where $M:=m_-+m_1+\cdots+m_{2g}+m_+$ and similarly $N:=n_-+n_1+\cdots+n_{2g}+n_+$. (Note that $d$ corresponds to the \textit{degree} of any curves counted in $\langle\partial\alpha,\beta\rangle$, as proved in \S\ref{degree}.) We will prove

\begin{proposition}\label{prop:indexcalc} Let $(Y,\lambda)$ be a prequantization bundle over a surface $\Sigma_g$ with Euler class $e\in\Z_{<0}$. The ECH index in $ECH_*^L(Y,\lambda_{\varepsilon(L)},\Gamma)$ satisfies the following formula for any $\Gamma, L$, and Reeb current $\alpha$ and $\beta$.
\begin{equation}\label{eqn:indexformula}
I(\alpha,\beta)=\chi(\Sigma_g)d-d^2e+2dN+m_+-m_--n_++n_-.
\end{equation}
\end{proposition}

In particular, note that the ECH index $I(\alpha,\beta)$ depends only on the generators $\alpha$ and $\beta$ and not on a relative homology class $Z\in H_2(Y,\alpha,\beta)$. This is proved in Lemma \ref{lem:indepZ}.

This section is organized as follows. In \S\ref{subsec:CZ} we compute the Conley-Zehnder index term, in \S\ref{subsec:relfirstchern} we compute the relative first Chern class term, and in \S\ref{subsec:relselfint} we compute the relative intersection pairing term. We combine these results to prove Proposition \ref{prop:indexcalc} in \S\ref{subsec:ECHindexcomp}.

\subsection{Perturbed Reeb dynamics}\label{subsec:pertReeb}

Let $(Y,\lambda)$ be the prequantization bundle over $(\Sigma_g,\omega)$ with Euler class $e=-\frac{1}{2\pi}[\omega]<0$ and contact structure $\xi=\ker(\lambda)$. Recall that the Reeb orbits of $\lambda$ consist of the $S^1$ fibers of $\fp:Y\to\Sigma_g$, all of which have action $2\pi$. We take:
\begin{equation}
\label{perturbedform1}
\lambda_\varepsilon=(1+\varepsilon \fp^*H)\lambda.
\end{equation}

A standard computation, cf. \cite[Prop. 4.10]{jo2}, yields the following.
\begin{lemma}
The Reeb vector field of $\lambda_\varepsilon$ is given by
\begin{equation}
\label{perturbedreeb1}
R_{\varepsilon}=\frac{R}{1+\varepsilon \fp^*H} + \frac{\varepsilon \widetilde{X}_H}{{(1+\varepsilon \fp^*H)}^{2}},
\end{equation}
where $X_{H}$ is the Hamiltonian vector field\footnote{We use the convention $\omega(X_H, \cdot) = dH.$} on $\Sigma$ and $\widetilde{X}_{ H}$ is its horizontal lift.
\end{lemma}

We now prove Lemma \ref{lem:efromL}.
\begin{proof}[Proof of Lemma \ref{lem:efromL}]
To prove (i), note that the horizontal lift $\widetilde{X}_{H}$  is determined  by
\[
  dh(q)\widetilde{X}_H(q) = X_{\varepsilon  H}(h(q)) \ \ \mbox{ and } \ \ \lambda(\widetilde{X}_H)=0.
\]
Thus those orbits which do not project to $p\in \mbox{Crit}(H)$ must project to $X_H$.  We have 
\[
\frac{\varepsilon}{(1+\varepsilon)^2} <  \frac{\varepsilon}{{(1+\varepsilon \fp^*H)}^{2}} < \frac{\varepsilon}{(1-\varepsilon)^2}.
\]
A Taylor series expansion shows that the $k$-periodic orbits of $X_H$ give rise to orbits of $ \frac{\varepsilon \widetilde{X}_H}{{(1+\varepsilon \fp^*H)}^{2}}$ which are $\frac{C}{\varepsilon}$-periodic for some $C$.  We note that $C$ and $k$ must be bounded away from 0 since $X_H$ is time autonomous.  Nondegeneracy of Reeb orbits $\gamma$ such that $\mathcal{A}(\gamma) <L$ follows from the proof of Theorem 13 in Appendix A of \cite{ABW}.

We obtain (ii) because the period of an orbit of $R_\varepsilon$ over a critical point $p$ of $H$ must be $1+\varepsilon H(p)$ times the period of the orbit of $R$ over $p$ by (\ref{perturbedreeb1}). \end{proof}

\subsection{Homology of prequantization bundles}\label{subsec:homology}

In this subsection we review preliminaries on the homology of prequantization bundles. We prove Lemma \ref{lem:HYreps} using the Leray-Serre spectral sequence to identify representatives of $H_*(Y;\Z)$ classes. We note that one could alternatively use the Gysin sequence, which is specialized to sphere bundles and avoids the use of spectral sequences; cf. \cite[\S 7]{spanier}.

Let $Y$ be a prequantization bundle over a two-dimensional surface $\Sigma_g$ of negative Euler class $e$. The second page of its Leray-Serre spectral sequence has terms
\[
E^2_{p,q}=H_p(\Sigma_g;\{H_q(Y_x)\})=H_p(\Sigma_g;\Z)
\]
for $q=0,1$. Since $\partial_2:E^2_{p,q}\to E^2_{p-2,q+1}$, the only differential on the second page which neither starts nor ends at a trivial group is from $E^2_{2,0}=H_2(\Sigma_g;\Z)$ to $E^2_{0,1}=H_0(\Sigma_g;\Z)$; this differential sends the element of $E^1_{2,0}$ corresponding to a closed 2-cell in $\Sigma$ to the obstruction to finding a section over $\Sigma_g$, and so the image of $\partial_2$ in $E^2_{0,1}$ is $e\Z$. The other groups are unchanged.

Since all higher differentials will either start or end at a trivial group, we obtain
\begin{equation}\label{eqn:HofPQB}
H_*(Y;\Z)=\begin{cases}\Z&*=3,
\\\Z^{2g}&*=2,
\\\Z^{2g}\oplus\Z_{-e}&*=1,
\\\Z&*=0.
\end{cases}
\end{equation}

In our computations of the ECH index we will need to understand representatives of the degree one and two homology classes.

\begin{lemma}\label{lem:HYreps}Let $Y$ be a prequantization bundle over a two-dimensional surface $\Sigma_g$ of negative Euler class $e$.
\begin{enumerate}
\item[\em (i)] Let $f_p$ denote the fiber over $p\in\Sigma_g$. Its $k$-fold cover represents the class $k \text{ mod}(-e)$ in the $\Z_{-e}$ summand of $H_1(Y)$.
\item[\em (ii)] Each $H_2(Y)$ class is represented by the union of fibers over a representative of an $H_1(\Sigma_g)$ class.
\end{enumerate}
\end{lemma}
\begin{proof} Because $\Sigma_g$ is a CW complex, the Leray-Serre spectral sequence can be constructed using the filtration on $C_*(Y)$ where $F_p(C_*(Y))$ is the subcomplex consisting of singular chains supported in the preimage under $\fp$ of the $p$-skeleton of $\Sigma_g$.

To show (i), we first show that $E^2_{0,1}$ is generated by the fiber. Because $E^1_{0,1}$ is generated by the fiber and $E^1_{1,1}$ is generated by 2-chains of $Y$ over the 1-skeleton of $\Sigma_g$, the image of $E^1_{1,1}$ under $\partial_1$ is zero. Therefore $E^2_{0,1}$ is also generated by the fiber.

Secondly, we show that $E^2_{2,0}$ is generated by a section of $Y$ over $\Sigma_g-\{pt\}$. It follows from the definitions that $E^1_{2,0}$ is generated by a section of $Y$ over $\Sigma_g-\{pt\}$, and every such section is in the kernel of $\partial_1$ because its boundary is a 1-chain in $\pi^{-1}(\{pt\})$. Therefore $E^2_{2,0}=E^1_{2,0}$.

The differential $\partial_2$ takes the generator of $E^2_{2,0}$ to $f_p^e$, so (i) is proved.

To show (ii), note that $E^2_{1,1}$ consists of 2-chains over the 1-skeleton of $\Sigma_g$ whose boundaries do not wrap around fibers and which are not the boundary of a 3-chain in the preimage under $\fp$ of the 1-skeleton of $\Sigma_g$. Therefore elements of $E^2_{1,1}$ can be represented by preimages under $\fp$ of representatives of $H_1(\Sigma_g)$ classes.
\end{proof}

In an abuse of notation, we will often refer to elements of the subgroup $\{0\}\times\Z_{-e}$ of $H_1(Y)$ simply as elements of $\Z_{-e}$.

\subsection{Trivialization and Conley Zehnder index}\label{subsec:CZ}
We will use the constant trivialization as considered in \cite[\S 3.1, 4.2]{ggm1} to compute the Conley-Zehnder indices.   For any point  $q \in \fp^{-1}(p)$, a fixed trivialization of $T_p\Sigma_g$ allows us to trivialize $\xi_q$ because $\xi_q \cong T_p\Sigma$.   This trivialization is invariant under the linearized Reeb flow and can be thought of as a  \emph{constant trivialization} over the orbit $\gamma_p$ because the linearized Reeb flow, with respect to this trivialization, is the identity map.

Using this constant trivialization, we have the following result regarding the Robbin-Salamon index, see \cite[Lem. 3.3]{vknotes}, \cite[Lem. 4.8]{jo2}.

\begin{lemma}\label{consttrivlem}
Let $(Y,\lambda) \overset{\fp}{\rightarrow}(\Sigma_g, \omega)$ be a prequantization bundle of negative Euler number $e$.  Then for the constant trivialization $\tau$ along the circle fiber $\gamma = \fp^{-1}(p) $, we obtain $RS_\tau(\gamma)=0$ and $RS_\tau(\gamma^k) =0$, where $RS$ denotes the Robbin-Salamon index of the $k$-fold iterate of the fiber.\end{lemma}

We also have the following formula for the Conley-Zehnder indices of iterates of orbits which project to critical points $p$ of $H$.  We denote the $k$-fold iterate of an orbit which projects to $p \in \mbox{Crit}(H)$ by $\gamma_p^k$. 

\begin{lemma}\label{lempre}
Fix $L>0$ and $H$ a Morse-Smale function on $\Sigma$ which is $C^2$ close to 1.  Then there exists $\varepsilon >0$ such that all periodic orbits $\gamma$ of $R_\varepsilon$ with action $\mathcal{A}(\gamma) <L$ are nondegenerate and project to critical points of $H$.  The Conley-Zehnder index such a Reeb orbit over $p \in \mbox{\em Crit}(H)$ is given by
\[
\begin{array}{lcl}
CZ_\tau(\gamma_p^k) &=& {RS}_\tau(\gamma^k) -1 + \mbox{\em index}_pH,\\
&=& \mbox{\em index}_pH -1.\\
\end{array}
\]
\end{lemma}

Detailed definitions and computations of the Conley-Zehnder and Robbin-Salamon index as well as the proofs of the preceding standard computations can be found in \cite[\S 4]{jo2}, \cite{vknotes}.

 Lemma \ref{lem:orbitseh}, which classifies the orbits we consider, follows from the above computation:

\begin{proof}[Proof of Lemma \ref{lem:orbitseh}] By Lemma \ref{lempre}, we have
\begin{equation}\label{eqn:CZH}
CZ_\tau(\gamma_p^k)=\mbox{index}_pH-1,
\end{equation}
Since $\lambda_\varepsilon$ is a small perturbation of $\lambda$, all linearized return maps of Reeb orbits must be close to the identity. Therefore there can be no negative hyperbolic orbits.  (Alternatively, one could conclude that there are no negative hyperbolic orbits by way of the iteration properties of the Conley-Zehnder index, as summarized in \S \ref{cz-sec}.)

Positive hyperbolic orbits have even Conley Zehnder indices, so the $h_i$, which all have Conley Zehnder index zero by (\ref{eqn:CZH}), must be positive hyperbolic.  Elliptic orbits have odd Conley Zehnder indices, so the $e_\pm$, with Conley Zehnder indices $\pm1$ by (\ref{eqn:CZH}), must be elliptic.

\end{proof}

\subsection{ECH generators}\label{subsec:PQBgens}

We explain why generators of the form $e_-^{m_-}h_1^{m_1}\cdots h_{2g}^{m_{2g}}e_+^{m_+}$ are all we need to consider until \S\ref{sec:finalcomp}. Our focus through \S\ref{handleslides} will be to build the foundations necessary to understand the direct limit $\lim_{L\to\infty}ECH^L_*(Y,\lambda_{\varepsilon(L)},\Gamma;J)$. In Theorem \ref{thm:dlisECH} we will show that
\[
\lim_{L\to\infty}ECH^L_*(Y,\lambda_{\varepsilon(L)},\Gamma)=ECH_*(Y,\xi,\Gamma).
\]
We will then prove the main theorem in \S\ref{sec:finalcomp} by relating the direct limit $\lim_{L\to\infty}ECH^L_*(Y,\lambda_{\varepsilon(L)},\Gamma)$, as we understand it via Proposition \ref{prop:directlimitcomputesfiberhomology}, to the Morse homology of the base, which will require the analysis of \S\ref{sec:modspcs}-\S\ref{handleslides}. This section will be devoted to understanding the ECH index of the generators which Proposition \ref{prop:directlimitcomputesfiberhomology} tells us are relevant, i.e., those whose embedded orbits are fibers above critical points of $H$.
 
\begin{remark}
We require that $H$ be perfect, so that $H$ has exactly as many critical points of index $i$ as its $i^\text{th}$ Betti number. Let $e_+$ denote the orbit whose image is the fiber above the index two critical point of $H$. Similarly, let $e_-$ denote the orbit above the index zero critical point and let $h_1,\dots,h_{2g}$ denote the orbits above the index one critical points.
\end{remark}

The notation is derived from the fact that the orbits $e_\pm$ are elliptic, with slightly positive/negative rotation numbers, respectively, in the constant trivialization discussed in \S\ref{subsec:CZ}, and the $h_i$ are positive hyperbolic.  Heuristically, this follows from the fact that the linearized return map of an orbit projecting to a critical point $p$ of $H$ approximately agrees with a lift of the linearized flow of $\varepsilon X_H$ on $T_p\Sigma_g$. However, in \S \ref{subsec:CZ} we classified the Reeb orbits by appealing to properties of the Conley-Zehnder index.

We next prove Proposition \ref{prop:directlimitcomputesfiberhomology}.
\begin{proof}[Proof of Proposition \ref{prop:directlimitcomputesfiberhomology}]

Given $\varepsilon>\varepsilon'$ there is an exact symplectic cobordism $(X_{\varepsilon,\varepsilon'},\lambda_{\varepsilon,\varepsilon'}):=([\varepsilon',\varepsilon]\times Y,(1+s\fp^*H)\lambda)$ from $(Y,\lambda_\varepsilon)$ to $(Y,\lambda_{\varepsilon'})$. (It is symplectic because $d\lambda_{\varepsilon,\varepsilon'}^2$ is a positive multiple of $ds\wedge\lambda\wedge d\lambda$.)

Thus we have cobordism maps $\Phi^L(X_{\varepsilon,\varepsilon'},\lambda_{\varepsilon,\varepsilon'})$ as in Theorem \ref{thm:cc2cobmaps}, inclusion maps $\iota^{L,L'}$ as in \cite[Thm. 1.3]{cc2}, and a commutative diagram
\begin{equation}\label{eqn:trivialcobcd}
\xymatrixcolsep{4pc}\xymatrix{
ECH^L_*(Y,\lambda_{\varepsilon},\Gamma) \ar[r]^{\Phi^L(X_{\varepsilon,\varepsilon'},\lambda_{\varepsilon,\varepsilon'})} \ar[d]_{\iota^{L,L'}} & ECH^L_*(Y,\lambda_{\varepsilon'},\Gamma) \ar[d]^{\iota^{L,L'}}
\\ECH^{L'}_*(Y,\lambda_{\varepsilon},\Gamma) \ar[r]_{\Phi^{L'}(X_{\varepsilon,\varepsilon'},\lambda_{\varepsilon,\varepsilon'})} & ECH^{L'}_*(Y,\lambda_{\varepsilon'},\Gamma)
}
\end{equation}
by adapting (\ref{eqn:ECHcd}) from the Inclusion property of Theorem \ref{thm:cc2cobmaps}. (Because $X$ is a product of $Y$ with an interval, the cobordism maps respect the splitting.)

Because if $L<L'$ then $\varepsilon(L)>\varepsilon(L')$, from either path on the commutative diagram (\ref{eqn:trivialcobcd}) we get a well-defined map
\begin{equation}\label{eqn:directsystemoverLmaps}
ECH^L_*(Y,\lambda_{\varepsilon(L)},\Gamma)\to ECH^{L'}_*(Y,\lambda_{\varepsilon(L')},\Gamma).
\end{equation}
For the $ECH^L_*(Y,\lambda_{\varepsilon(L)},\Gamma)$ to form a direct system, it remains to show that the maps (\ref{eqn:directsystemoverLmaps}) compose. In the following denote by $\Phi^L(\varepsilon,\varepsilon')$ the cobordism map $\Phi^L(X_{\varepsilon,\varepsilon'},\lambda_{\varepsilon,\varepsilon'})$. It is enough to show that for $L''>L'>L$ and $\varepsilon''<\varepsilon'<\varepsilon$, the composition
\[
ECH^L_*(Y,\lambda_\varepsilon,\Gamma)\overset{\iota^{L,L'}}{\to}ECH^{L'}(Y,\lambda_\varepsilon,\Gamma)\overset{\Phi^{L'}(\varepsilon',\varepsilon'')\circ\Phi^{L'}(\varepsilon,\varepsilon')}{\longrightarrow}ECH^{L'}(Y,\lambda_{\varepsilon''},\Gamma)\overset{\iota^{L',L''}}{\to}ECH^{L''}(Y,\lambda_{\varepsilon''},\Gamma)
\]
equals either $\Phi^{L''}(\varepsilon,\varepsilon'')\circ\iota^{L,L''}$ or $\iota^{L,L''}\circ\Phi^{L}(\varepsilon,\varepsilon'')$. This follows from the four-fold commutative diagram consisting of the versions of (\ref{eqn:trivialcobcd}) for $(L,\varepsilon)$ to $(L',\varepsilon')$, $(L,\varepsilon')$ to $(L',\varepsilon'')$, $(L',\varepsilon)$ to $(L'',\varepsilon')$, and $(L',\varepsilon')$ to $(L'',\varepsilon'')$ in concert. In this four-fold commutative diagram, the path across the top and down the right side equals $\Phi^{L''}(\varepsilon,\varepsilon'')\circ\iota^{L,L''}$, by the Composition property of Theorem \ref{thm:cc2cobmaps}, and similarly the path down the left side and across the bottom equals $\iota^{L,L''}\circ\Phi^{L}(\varepsilon,\varepsilon'')$.

It remains to show that the direct limit is the homology of the chain complex generated by Reeb currents whose embedded orbits are fibers above the critical points of $H$ and whose multiplicities can be any element of $\Z_{>0}$.

The embedded orbits contributing to the generators of any $ECC^L_*(Y,\lambda_{\varepsilon(L)},\Gamma;J)$ must be orbits over critical points of $H$ by Lemma \ref{lem:efromL} (i). And by Lemma \ref{lem:efromL} (ii), for any pair $(\gamma,m_\gamma)$ where $\gamma$ is an orbit above a critical point of $H$ and $m_\gamma\in\Z_{>0}$, there is some (possibly very large) $L$ for which $m_\gamma\mathcal{A}(\gamma)<L$ when $\mathcal{A}$ is computed using $\lambda_{\varepsilon(L)}$.

To complete the proof we need to know that the maps (\ref{eqn:directsystemoverLmaps}) are induced by the obvious inclusion of chain complexes
\[
ECC^L_*(Y,\lambda_{\varepsilon(L)},\Gamma;J)\to ECC^{L'}_*(Y,\lambda_{\varepsilon(L')},\Gamma;J).
\]
Because $\iota^{L,L'}$ is induced by inclusion, it suffices to show that the map $\Phi^L(X,\lambda_{\varepsilon,\varepsilon'})$ is also induced by inclusion. (There is no need to check the cobordism map $\Phi^{L'}(X,\lambda_{\varepsilon,\varepsilon'})$ because the diagram commutes.)

That $\Phi^L(X,\lambda_{\varepsilon,\varepsilon'})$ is induced by inclusion follows if there is a smooth 1-parameter family $\lambda_t$ where $\lambda_t=\lambda_{\varepsilon-(\varepsilon-\varepsilon')t}$ which, for $\lambda_t$-compatible almost complex structures $J_t$, the pairs $(\lambda_t,J_t)$ satisfy the hypotheses of Lemma \ref{lem:nicecobmap} for $L$. The first and second bullet points follow immediately from the construction. The fourth bullet point generically holds. The third bullet point can then be accomplished by a deformation as in \cite[Prop. B.1]{taubesechswf1} (see also \cite[Lem. 3.6]{cc2}).
\end{proof}

\subsection{Computation of the ECH index}

In this subsection we prove Proposition \ref{prop:indexcalc}, namely:
\[
I(\alpha,\beta)=\chi(\Sigma_g)d-d^2e+2dN+m_+-m_--n_++n_-.
\]
 We note that the structure of the proof follows \cite[\S3]{farris}. The differences are the following:
 \begin{itemize}
 \item We have incorporated the notion of degree earlier in the computation.
 \item We have clarified the surfaces used for the computation. This is quite delicate, particularly when computing the relative intersection pairing in \S\ref{subsec:relselfint}, thus we devote quite a bit of time to their setup in \S\ref{subsec:surfaces}.
 \end{itemize}
 
 Let $Y$ be a prequantization bundle over a surface $\Sigma_g$ of negative Euler class $e$, and let $L\in\R$ be large. Let $\Gamma$ be a torsion element in the $\{0\}\times\Z_{-e}$ subgroup of $H_1(Y)$. In this subsection we prove Proposition \ref{prop:indexcalc}. We first introduce some notation which we will use throughout the rest of this section in our computation of the ECH index.

Given generators $\alpha=e_-^{m_-}h_1^{m_1}\cdots h_{2g}^{m_{2g}}e_+^{m_+}$ and $\beta=e_-^{n_-}h_1^{n_1}\cdots h_{2g}^{n_{2g}}e_+^{n_+}$ of $ECC^L_*(Y,\lambda_{\varepsilon(L)},\Gamma)$, let
\[
M:=m_-+m_1+\cdots+m_{2g}+m_+\text{ and }N:=n_-+n_1+\cdots+n_{2g}+n_+.
\]
Because $[\alpha]=[\beta]=\Gamma$, there is some $d\in\Z$ so that
\begin{equation}\label{eqn:diffmult-e}
M=N+(-e)d.
\end{equation}
Throughout the proof of the index formula, which occupies the rest of this section, we will assume $d\geq0$; the $d\leq0$ case is handled similarly, and signs will change appropriately.

\begin{lemma}\label{lem:indepZ} Given $\alpha$ and $\beta$ as above and $Z\in H_2(Y,\alpha,\beta)$, the ECH index $I(\alpha,\beta,Z)$ does not depend on $Z$.
\end{lemma}
\begin{proof} Let $A\in H_2(Y)$ and $Z\in H_2(Y,\alpha,\beta)$. From the index ambiguity formula, Theorem \ref{thm:Iproperties} (\ref{property:indexamb}), we have
\[
I(\alpha,\beta,Z+A)-I(\alpha,\beta,Z)=\langle c_1(\xi)+2PD(\Gamma),A\rangle=c_1(\xi)(A)+2\Gamma\cdot A.
\]

Assume $g>0$. Recall from Lemma \ref{lem:HYreps} (ii) that $H_2(Y)=\Z^{2g}$, and if $a_1,b_1,\dots,a_g,b_g$ generate $H_1(\Sigma)$, then the unions of fibers over representatives of $a$ and $b$ will generate $H_2(Y)$. The class $\Gamma$ can be represented by a single fiber, so a representative of $\Gamma$ can be isotoped not to intersect a representative of $A$. Thus $\Gamma\cdot A=0$.  Moreover, we have $c_1(\xi)(A)=0$, via
\[
c_1(\xi)(A)=c_1(T\Sigma_g)(\fp_*(A))=0,
\]
because $\fp_*$ will send a representative of $A$ to a representative of a 1-cycle in $\Sigma$.

If $g=0$, then we have $H_2(Y)=0$, and because $H_2(Y,\alpha,\beta)$ is affine over $H_2(Y)$, there is no possibility for index ambiguity.
\end{proof}

Therefore $I(\alpha,\beta,Z)$ is independent of $Z$, and will from now on be denoted $I(\alpha,\beta)$. Similarly, we will use $c_\tau(\alpha,\beta)$ and $Q_\tau(\alpha,\beta)$ to denote $c_\tau(Z)$ and $Q_\tau(Z)$.

We will now compute the relative first Chern class and relative intersection pairing terms in the ECH index. Lemmas \ref{lempre} and \ref{lem:orbitseh} allow us to compute the Conley Zehnder index term in the final proof of Proposition \ref{prop:indexcalc}. Throughout the computation we use the constant trivialization $\tau$ from \S\ref{subsec:CZ}.

\subsubsection{Surfaces in $Y$}\label{subsec:surfaces}

Let $\alpha$ and $\beta$ be homologous Reeb currents, thus satisfying (\ref{eqn:diffmult-e}). Before computing $c_\tau$ and $Q_\tau$, we will define surfaces in $Y$ representing $[\alpha]=[\beta]$ to be used in both the computation of the relative first Chern class in \S\ref{subsec:relfirstchern} and the computation of relative intersection pairing in \S\ref{subsec:relselfint}. In this section we use the notation $\alpha=\{(\alpha_k,m_k)\}$ and $\beta=\{(\beta_l,n_l)\}$; in particular, the $m_k$ and $n_l$ are not necessarily equal to the multiplicities of hyperbolic orbits.

Let $\fp(\alpha)$ and $\fp(\beta)$ denote the sets of points $\{\fp(\alpha_k)\}$ and $\{\fp(\beta_l)\}$, respectively, where the point $\fp(\alpha_k)$ appears with multiplicity $m_k$ and $\fp(\beta_l)$ appears with multiplicity $n_l$. Choose any subset of $\fp(\alpha)$ of total multiplicity $N$ and denote it $\fp(\alpha)_\beta$; such a subset exists because we are assuming $d\geq0$ in (\ref{eqn:diffmult-e}). Note that the multiplicity of $\fp(\alpha_k)$ in $\fp(\alpha)_\beta$ does not have to equal $m_k$, though it is at most $m_k$. Denote the set of points in $\Sigma_g$ underlying $\fp(\alpha)_\beta$ by $\{\fp(\alpha)_\beta\}$.

First we explain how to obtain a surface in $Y$ connecting a set of orbits from $\alpha$ of total multiplicity $N$ with $\beta$. Choose a graph $G_N$ embedded in $\Sigma_g$ with vertices $\{\fp(\alpha)_\beta\}\cup\{\fp(\beta)\}$, where the degree of each vertex equals its multiplicity as part of $\fp(\alpha)_\beta$ or $\fp(\beta)$. Furthermore, we require that the edges of $G_N$ partition $\{\fp(\alpha)_\beta\}\cup\{\fp(\beta)\}$ into $\fp(\alpha)_\beta$ and $\fp(\beta)$ in the sense that each edge of $G_N$ can be labeled with a pair in $\fp(\alpha)_\beta\times\fp(\beta)$ where the edge of $G_N$ connects the underlying pair in $\Sigma_g$, and all points in $\fp(\alpha)_\beta\cup\fp(\beta)$ are connected in this way. Finally, we require that the edges of $G_N$ intersect only transversely, including at their endpoints (meaning that if $x$ is an endpoint with degree at least two, the one-sided limits of the tangent vectors to those edges form a basis for $T_x\Sigma_g$). In particular, if $x\in\fp(\alpha)_\beta\cap\fp(\beta)$ then the graph can include transversely intersecting loops from $x$ to $x$. Let $\check{S}_N$ denote the union of the fibers above $G_N$.

Now we explain how to obtain a surface in $Y$ with boundary homologous to the remaining $(-e)d$ orbits in $\alpha$ (counted with multiplicity). Denote $\fp(\alpha)-\fp(\alpha)_\beta$ by $\fp(\alpha)_\alpha$. Divide the points in $\fp(\alpha)_\alpha$ into $d$ subsets, each of total multiplicity $-e$. Denote each subset by $\fp(\alpha)_\alpha^k$, for $k=1,\dots,d$, and denote the underlying set of points by $\{\fp(\alpha)_\alpha^k\}$. Let $\text{mult}_k(x)$ denote the multiplicity of $x$ as an element of $\fp(\alpha)_\alpha^k$. For each $k$ choose a section $\check{S}_k$ of $Y$ over $\Sigma_g-\{\fp(\alpha)_\alpha^k\}$ whose boundary forms a $\text{mult}_k(x)$-fold cover of the fiber over $x$, for each $x\in\{\fp(\alpha)_\alpha^k\}$.

For $z\not\in\{\fp(\alpha)_\alpha^k\}$, denote the point of $\check{S}_k$ above $z$ by $\check{S}_k(z)$.

\subsubsection{Relative first Chern class}\label{subsec:relfirstchern}

We compute the relative first Chern class $c_\tau(\alpha,\beta)$ of Reeb currents $\alpha$ and $\beta$.

\begin{lemma}\label{lem:relfirstCherncalc}
Given Reeb currents $\alpha$ and $\beta$ satisfying (\ref{eqn:diffmult-e}), their relative first Chern class $c_\tau(\alpha,\beta)$ satisfies the following formula:
\begin{equation}\label{eqn:relfirstChern}
c_\tau(\alpha,\beta)=\chi(\Sigma_g)d.
\end{equation}
\end{lemma}
\begin{proof}
We will use the surfaces $\check{S}_N$ and $\check{S}_k$ of \S\ref{subsec:surfaces}. It is immediate from the definition of $c_\tau$ that if $S$ and $S'$ are two admissible surfaces, then
\[
c_\tau(S\cup S')=c_\tau(S)+c_\tau(S').
\]
Therefore, we have
\begin{equation}\label{eqn:ctausum}
c_\tau(\alpha,\beta)=c_\tau(\check{S}_N\cup\check{S}_1\cdots\cup\check{S}_d)=c_\tau(\check{S}_N)+\sum_{k=1}^dc_\tau(\check{S}_k).
\end{equation}

Since $\xi=\fp^*T\Sigma_g$, the first Chern class of $\xi$ is $\fp^*c_1(T\Sigma_g)$. Since $\fp(\check{S}_N)=G_N$ represents zero in $H_2(\Sigma_g;\Z)$,
\begin{equation}\label{eqn:ctauSN}
c_\tau(\check{S}_N)=0.
\end{equation}

Since $[\fp(\check{S}_k)]=[\Sigma_g]$ in $H_2(\Sigma_g;\Z)$,
\begin{equation}\label{eqn:ctauSk}
c_\tau(\check{S}_k)=\chi(\Sigma_g).
\end{equation}

Combining equations (\ref{eqn:ctausum}), (\ref{eqn:ctauSN}), and (\ref{eqn:ctauSk}) yields the desired result.
\end{proof}

\subsubsection{Relative intersection pairing}\label{subsec:relselfint}

We compute the relative intersection pairing $Q_\tau(\alpha,\beta)$ of Reeb currents $\alpha$ and $\beta$. 

\begin{lemma}\label{lem:relselfint} 
Given Reeb currents $\alpha$ and $\beta$ satisfying (\ref{eqn:diffmult-e}), their relative intersection pairing $Q_\tau(\alpha,\beta)$ satisfies the following formula:
\begin{equation}\label{eqn:relselfint}
Q_\tau(\alpha,\beta)=-ed^2+2dN.
\end{equation}
\end{lemma}

\begin{proof}
We will first lift the surfaces $\check{S}_N$ and $\check{S}_k$ in $Y$ from \S\ref{subsec:surfaces} to admissible surfaces in $[-1,1]\times Y$ to use in our computation. To lift $\check{S}_N$ to an admissible surface $S_N\subset[-1,1]\times Y$, parameterize the edges of $G_N$ by $[-1,1]$ from $\fp(\beta)$ to $\fp(\alpha)_\beta$ so that they do not intersect as parameterized curves. Denote these parameterizations by $g_i$. The non-intersecting requirement means that if $g_1,g_2$ parameterize two edges of $G_N$ which intersect in $\Sigma_g$, then we have $g_1(t_1)=g_2(t_2)$ only if $t_1\neq t_2$. Let $S_N$ be the surface
\[
\union_{i=1}^N(t,\fp^{-1}(g_i(t))).
\]

To construct an admissible surface with boundary on the remaining $(-e)d$ components of $\alpha$, we will define a family of lifts for each $\check{S}_k$ to an admissible surface $S_k\subset[-1,1]\times Y$. The lifts are isotopic, so the relative intersection pairing will not depend on our choice within the family. We will need this flexibility in order to guarantee transverse intersections.

Choose a disc neighborhood $\mathbb{D}^2_x$ for each $x\in\{\fp(\alpha)_\alpha\}$ which do not pairwise intersect, and parameterize each $\mathbb{D}^2_x$ with radial function $0\leq r_x\leq2$. For any choice of functions $\epsilon,l:\{1,\dots,d\}\to[0,2)$ with $0\leq\epsilon(k)<l(k)<2$ and $\delta:\{1,\dots,d\}\to\R_{>0}$, let $f_k:\Sigma_g\to\R$ be a smooth function for which
\[
f_k(z)=\begin{cases}
\delta(k)r_x&0\leq r_x\leq\frac{\epsilon(k)}{\delta(k)}
\\l(k)&l(k)\leq r_x\leq2
\\l(k)&\text{ outside $\bigcup_{x\in\fp(\alpha)_\alpha^k}\mathbb{D}^2_x$}
\end{cases} \text{ and }f'_k>0\text{ for }\frac{\epsilon(k)}{\delta(k)}\leq r_x<l(k).
\]
For each $k=1,\dots,d$, we define $S_k\subset[-1,1]\times Y$ to be the surface
\[
S_k:=(1-f_k(z),\check{S}_k(z)).
\]
Heuristically, $S_k$ lifts to $[-1,1]$ near $1$ by the negative of each radial direction $r_x$ times $\delta(k)$, until the $[-1,1]$ coordinate reaches $\epsilon(k)$. After some smooth interpolation depending within the $\mathbb{D}^2_x$ discs, the rest of $S_k$ simply equals $\{1-l(k)\}\times\left(\check{S}_k-\bigcup_{x\in\fp(\alpha)_\alpha^k}\mathbb{D}^2_x\right)$. 

Expanding \cite[(3.11)]{Hrevisit}, we have
\begin{equation}\label{eqn:Qtausum}
Q_\tau(S_N\cup S_1\cup\cdots\cup S_d)=Q_\tau(S_N)+2\sum_{k=1}^dQ_\tau(S_N,S_k)+Q_\tau(S_1\cup\cdots\cup S_d).
\end{equation}
We will compute each term separately.

First, we have
\begin{equation}\label{eqn:QtauSN}
Q_\tau(S_N)=0,
\end{equation}
because the graph $G_N$ has self-intersection zero as a parameterized graph. That is, any intersections between the edges of $G_N$, including self-intersections, can occur away from $
\fp(\alpha)_\beta$ and $\fp(\beta)$, and the parameterizations can be adjusted so as to avoid intersection in $[-1,1]\times Y$. In particular, the self-linking of the braids $S_N\cap\{1-\epsilon\}\times Y$ is zero because $G_N$ can be isotoped so that its edges do not intersect near $\fp(\alpha)_\beta$ and $\fp(\beta)$, even as non-parameterized curves.\footnote{Alternately, one could show that $S_N$ is a ``$\tau$-representative" of $[\fp_Y(S_N)]$, following \cite{Hindex}, an alternate construction of $Q_\tau$ for which there is no need to consider boundary self-linking.}

Second, we compute $Q_\tau(S_N,S_k)$. We can choose the parameterizations $g(t)$ of the edges of $G_N$ so that when $t=1-l(k)$, the point $g(t)$ is outside all disks $\mathbb{D}^2_x$, the derivative $g'(1-l(k))\neq0$, and if $g(t)$ has an end at $x$, that when $t\geq1-\frac{\epsilon(x)}{\delta(x)}$, the parameterization $g(t)\equiv x$.

The points $(1-l(k),\check{S}_k(z))\in(1-l(k),\fp^{-1}(g(1-l(k))))$ will then be the only points of intersection between $\check{S}_k$ and the edges of $G_N$. Each contributes to the count of intersections with sign $+1$ because, in the oriented local basis $\{\partial_s,R,\partial_1,\partial_2\}$ for $\R_s\times Y$, where $\{\partial_1,\partial_2\}$ is an oriented basis for $\Sigma_g$ and $\partial_1$ equals the tangent vector to the edge in question (i.e., $\partial_1=g'(1-l(k))$), an oriented basis for $TS_N\oplus TS_k$ at their point of intersection is
\[
\{(1,0,1,0),(0,1,0,0),(0,0,1,0),(0,0,0,1)\}.
\]
Therefore
\begin{equation}\label{eqn:QtauSNSk}
Q_\tau(S_N,S_k)=N.
\end{equation}

Finally we consider the self-intersection of the union of $S_k$s. We have
\[
Q_\tau(S_1\cup\cdots\cup S_d)=\sum_{k=1}^dQ_\tau(S_k)+\sum_{k_1\neq k_2}Q_\tau(S_{k_1},S_{k_2}).
\]
We will show that $Q_\tau(S_{k_1},S_{k_2})$ does not depend on the $k_i$ (even if $k_1=k_2$). Therefore, because
\[
d+2{d\choose2}=d+\frac{2d!}{2(d-2)!}=d+d(d-1)=d^2,
\]
we will get
\[
Q_\tau(S_1\cup\cdots\cup S_d)=d^2Q_\tau(S_1).
\]

To compute $Q_\tau(S_{k_1},S_{k_2})$ for any $k_1,k_2$, let $\delta_i,\epsilon_i,l_i$ denote $\delta(k_i)$, etc. Choose $\delta_1>\delta_2$ and $\frac{\epsilon_1}{\delta_1}>l_2$.

Because $\frac{\epsilon_1}{\delta_1}>l_2$, all intersections between $S_{k_1}$ and $S_{k_2}$ must occur at points whose projection to $\Sigma_g$ lies within the disk neighborhoods of $\fp(\alpha)_\alpha^{k_1}$.

Assume $x\in\{\fp(\alpha)_\alpha^{k_1}\}$ but $x\not\in\{\fp(\alpha)_\alpha^{k_2}\}$. In the local product coordinates $\partial_s,R,\partial_r,\partial_\theta$ determined by the section $\check{S}_{k_2}$, the intersection $S_{k_1}\cap(\{1-l_2\}\times Y)$ consists of a $(1,\text{mult}_{k_1}(x))$ torus knot about the fiber over $x$ in the $s=1-l_2$ level of $[-1,1]$ with $r_x=\frac{l_2}{\delta_1}$, and $S_{k_2}$ consists of the zero section of $\fp$, so is parameterized by $(1-l_2,0,r,\theta)$. In particular, by $T_{1,\text{mult}_{k_1}(x)}$ we are referring to $\partial_\theta$ as the meridional coordinate and $R$ as the longitudinal coordinate.

Since in oriented bases for $TS_{k_1}$ and $TS_{k_2}$ near any intersection in the $s=1-l_2$ slice only the first basis vector for $TS_{k_1}$ will have any $\partial_s$ component, and it will be positive, the intersection number in $[-1,1]\times Y$ will agree with the intersection number of the projections to $Y$ in $Y$. These projections will consist of the $T_{1,\text{mult}_{k_1}(x)}$ torus knot in the $r=\frac{l_2}{\delta_1}$ torus and the disk obtained by projecting off the Reeb direction.

Similarly, only the first basis vector for $TS_{k_2}$ will have any $\partial_r$ component, and it will be positive. Therefore the intersection number of the projections to $Y$ will agree with the intersection number in the $r_x=\frac{l_2}{\delta_1}$ torus of the $T_{1,\text{mult}_{k_1}(x)}$ torus knot and the meridian, parameterized by $\theta$. Their intersection number can easily be computed via a matrix:
\begin{equation}\label{eqn:nomult2}
\begin{vmatrix}\text{mult}_{k_1}(x)&0\\1&1\end{vmatrix}=\text{mult}_{k_1}(x).
\end{equation}

Now assume $x\in\{\fp(\alpha)_\alpha^{k_1}\}\cap\{\fp(\alpha)_\alpha^{k_2}\}$. Let $0<\epsilon<\epsilon_2$. Using local coordinates $s$, the Reeb direction coming from $\fp^{-1}(x)$, and polar coordinates $r_x,\theta$ on the base, the intersections $S_{k_i}\cap(\{1-\epsilon\}\times Y)$ consist of $T_{1,\text{mult}_{k_i}(x)}$ torus knots in the tori $r_x=\frac{\epsilon}{\delta_i}$, respectively. Because $\delta_1>\delta_2$, the $T_{1,\text{mult}_{k_1}(x)}$ torus knot lies on a torus nested ``inside" the torus of the $T_{1,\text{mult}_{k_2}(x)}$ torus knot, where ``inside" refers to the component of $Y-T^2$ containing $\fp^{-1}(x)$. From knot diagrams of the image in $\R^3$ under the diffeomorphism of \S\ref{s:writhe} (see also the version in coordinates defined between Definitions 2.7 and 2.8 in \cite{Hrevisit}), it is immediate that
\begin{equation}\label{eqn:mult12}
\ell_\tau(\pi_Y(S_{k_1}\cap(\{1-\epsilon\}\times Y)),\pi_Y(S_{k_2}\cap(\{1-\epsilon\}\times Y)))=\text{mult}_{k_1}(x).
\end{equation}

Equations (\ref{eqn:nomult2}) and (\ref{eqn:mult12}) show that each $x\in\{\fp(\alpha)_\alpha^{k_1}\}$ contributes to $Q_\tau(S_{k_1},S_{k_2})$ according to its multiplicity. Since there are no other intersections or boundary components, we obtain
\begin{equation}\label{eqn:QtauSk}
Q_\tau(S_{k_1},S_{k_2})=-e.
\end{equation}

Combining equations (\ref{eqn:Qtausum}), (\ref{eqn:QtauSN}), (\ref{eqn:QtauSNSk}), and (\ref{eqn:QtauSk}) yields the desired result.
\end{proof}

\subsubsection{Proof of the ECH index formula}\label{subsec:ECHindexcomp}

\begin{proof}[Proof of Proposition \ref{prop:indexcalc}] Combining Lemmas \ref{lem:orbitseh}, \ref{lempre} tells us that
\[
CZ^I_\tau(\alpha)-CZ^I_\tau(\beta)=m_+-m_--(n_+-n_-).
\]
Adding (\ref{eqn:relfirstChern}) and (\ref{eqn:relselfint}) proves the result.
\end{proof}

Checking that our formula satisfies the additivity property of Theorem \ref{thm:Iproperties} (iii) is straightforward. Checking that our formula satisfies the index parity property of Theorem \ref{thm:Iproperties} (iv) requires relating the sums $m_1+\cdots+m_{2g}$ and $n_1+\cdots+n_{2g}$ to $m_\pm$ and $n_\pm$ via the formula $M=N+(-e)d$ defining $d$.

\section{The many flavors of $J$}\label{sec:modspcs}

In this section we work towards proving that $\partial^{L}$ only counts cylinders which are the union of fibers over Morse flow lines in $\Sigma_g$. 
One can count cylinders with a fiberwise $S^1$-invariant almost complex structure  $\mathfrak{J} := \mathfrak{p}^*j_{\Sigma_g}$,   the $S^1$-invariant lift of $j_{\Sigma_g}$, but unfortunately we cannot use $\mathfrak{J}$ for higher genus curves because it cannot be independently perturbed at the intersection points of $\pi_Y C$ with a given $S^1$-orbit by an $S^1$-invariant perturbation; see \S \ref{degree}.   Generically, there will always exist a regular $J \in \sj(Y,\lambda)$ for moduli spaces of nonzero genus curves, but we cannot assume that this $J$ is $S^1$-invariant.

To resolve this issue, we employ Farris' strategy \cite[\S 6]{farris} of using a family of $S^1$-invariant domain dependent almost complex structures, $\bj$, for higher genus curves, which was modeled on Cieliebak and Mohnke's approach for genus zero pseudoholomorphic curves in \cite{CM}.  To an ($S^1$-invariant) domain dependent almost complex structure $\bj \in \sj_{\ds}^{S^1}$ and a map $C: (\ds, j) \to (\R \times Y, \bj)$ where $g(\ds) >0$, we associate the $(0,1)$-form
\[
\overline{\partial}_{j,\bj} : = \tfrac{1}{2}\left( dC + \bj(z,C)\circ dC \circ j\right),
\]
which at the point $z\in \ds$ is given by
\[
\overline{\partial}_{j,\bj}C(z) := \tfrac{1}{2} \left( dC(z) + \bj(z,C(z))\circ dC(z) \circ j ( z) \right).
\]
We say that $C$ is \emph{$\bj$-holomorphic} whenever $\overline{\partial}_{j,\bj}C=0$.  There are two new phenomena to be accounted for in the case of higher genus $\bj$-holomorphic curves.  The first is that higher genus Riemann surfaces have  finite nontrivial symmetry groups, so the moduli space $\mgn$ is an orbifold, and therefore the moduli spaces of $\bj$-holomorphic curves are also orbifolds.  The second is that, even when using domain dependent almost complex structures, a nodal curve with a constant component of positive genus cannot be perturbed away to achieve transversality.  However in dimension 4, we will show how Farris' index considerations obstruct the latter configurations from arising.  

Our scheme for obtaining regularity will be that if $z, z' \in \ds$ map under $C$ to the same $S^1$ orbit in $Y$, then we will perturb $\bj$ independently at $z$ and $z'$ while preserving $\bj $'s $S^1$-invariance.  We will exploit this construction to prove the existence of regular $S^1$-invariant domain dependent almost complex structures in \S \ref{sec:ddacs}-\ref{sec:reg}.  Moreover, we show that for a generic choice of $S^1$-invariant domain dependent almost complex structure $\bj$, the moduli spaces of curves of nonzero genus with ECH index 1 are empty.  In \S \ref{handleslides} we will consider a one parameter family of domain dependent almost complex structures to relate curve counts defined with a domain dependent ($S^1$-invariant) $\bj$ and a generic $\lambda$-compatible almost complex structure $J$, permitting us to conclude that the only contributions to an appropriately filtered ECH differential are from cylinders which project to Morse flow lines.   


\subsection{Degree of a completed projected curve}\label{degree}
We first review the notion of a degree of a completed projection for a pseudoholomorphic curve in the symplectization of a prequantization bundle.
\begin{definition}[Degree of a completed projection]\label{def:degree}
If we compose the $J$-holomorphic curve
\[
C: \ds \to \R \times Y,
\]
with the projection
\[
\fp: Y \to \Sigma_g,
\]
then we obtain a map 
\[
\fp \circ \pi_Y C: \ds \to \Sigma_g,
\]
which has a well-defined non-negative \emph{degree} because $\fp \circ \pi_Y C$ extends to a map of closed surfaces. We define the \textit{degree} of $C$, denoted $\op{deg}(C)$, to be the degree of this map.
\end{definition}

\begin{remark}
One should not confuse degree with \emph{multiplicity}.  Recall that in Definition \ref{covered} that the multiplicity of a pseudoholomorphic curve $C$ is given by degree of the holomorphic branched covering map between the domain of $C$ and the domain of the underlying somewhere injective curve.  The multiplicity of a somewhere injective curve is always 1.  
\end{remark}

  We can relate the degree of the completed projection map $\fp \circ \pi_Y C$ to the number of positive and negative ends via the $d\lambda_\varepsilon$-energy and Stokes' Theorem as follows.  First we note the following.
  
\begin{remark}\label{acrem}
The action of a Reeb orbit $\gamma_p^k$ of $R_\varepsilon$ over a critical point $p$ of $H$ is proportional to the length of the fiber, namely
\[
\mathcal{A}(\gamma_p^k) = \int_{\gamma_p^k} \lambda_\varepsilon = 2k \pi (1+\varepsilon \fp^*H),
\]
because $\fp^*H$ is constant on critical points $p$ of $H$. 
\end{remark}

\begin{proposition}\label{prop-degree} For all $\lambda_\varepsilon$, we have the following relation between the degree $\op{deg}(C)$ of a curve $C\in\mathcal{M}^J(\alpha,\beta)$ and the total multiplicity of the Reeb orbits at the positive and negative ends:

\[
M-N = |e|\op{deg}(C).
\]
\end{proposition}

\begin{proof} Note that equality of the total homology classes of $\alpha$ and $\beta$ forces
\[
M-N = 0 \mbox{ mod } |e|.
\]

Denote by $H_\pm$ the values of $H$ at $\fp(e_\pm)$, respectively, and denote by $H_i$ the value of $H$ at $\fp(h_i)$.

Recall that the $d\lambda_\varepsilon$-energy (equivalently, contact area) $A(C)$ of a $J$-holomorphic curve $C$ is given by
\[
A(C) := \int_{\ds} (\pi_YC)^*(d\lambda_\varepsilon).
\]
Stokes' Theorem yields
\begin{align}
A(C)&= \int_{\ds} (\pi_YC)^*(d\lambda_\varepsilon) \nonumber
\\& = \int_{\partial(\pi_YC_*[\ds])}\lambda_\varepsilon \nonumber
\\&=2\pi\left(M-N+\varepsilon\left((m_--n_-)H_-+(m_+-n_+)H_++\sum_{i=1}^{2g}(m_i-n_i)H_i\right)\right).\label{eqn:ACinY}
\end{align}

On the other hand, we have
\[
d\lambda_\varepsilon=\varepsilon\fp^*dH\wedge\lambda+(1+\varepsilon\fp^*H)d\lambda,
\]
where $\fp^*\omega=d\lambda$ and $\omega[\Sigma_g]=2\pi|e|$. Therefore
\begin{align}
A(C)&=\varepsilon\int_{\ds}(\pi_YC)^*(\fp^*dH\wedge\lambda)+\int_{\ds}(1+\varepsilon H\circ\fp\circ\pi_YC)(\pi_YC)^*(\fp^*\omega) \nonumber
\\&=\varepsilon\int_{\ds}(\pi_YC)^*(\fp^*dH\wedge\lambda)+\omega[(\fp\circ\pi_YC)_*{\ds}]\left(1+\varepsilon\int_{\Sigma_g}H\right) \nonumber
\\&=\varepsilon\int_{\ds}(\pi_YC)^*(\fp^*dH\wedge\lambda)+2\pi|e|\op{deg}(C), \text{ because $\int_{\Sigma_g}H=0$.} \label{eqn:ACinSigma}
\end{align}

We claim that
\begin{equation}\label{eqn:dHlambda}
\int_{\ds}(\pi_YC)^*(\fp^*dH\wedge\lambda)=2\pi\left((m_--n_-)H_-+(m_+-n_+)H_++\sum_{i=1}^{2g}(m_i-n_i)H_i\right).
\end{equation}
Assuming (\ref{eqn:dHlambda}, we obtain the desired conclusion by setting the values for $A(C)$ computed in (\ref{eqn:ACinY}) and (\ref{eqn:ACinSigma}) equal to one another.

Therefore, it remains to show the claim (\ref{eqn:dHlambda}). We again use Stokes' theorem. Because $\int_{\Sigma_g}H=0$, we have
\begin{align*}
\int_{\ds}(\pi_YC)^*(\fp^*dH\wedge\lambda)&=\int_{\ds}(\pi_YC)^*(\fp^*dH\wedge\lambda)+\omega[(\fp\circ\pi_YC)_*{\ds}]\int_{\Sigma_g}H
\\&=\int_{\ds}(\pi_YC)^*(\fp^*dH\wedge\lambda)+\int_{(\pi_YC)_*[\ds]}\fp^*(H\omega)
\\&=\int_{\ds}(\pi_YC)^*(d(H\circ\fp)\wedge\lambda+(H\circ\fp)d\lambda)
\\&=\int_{\ds}(\pi_YC)^*d((H\circ\fp)\lambda)
\\&=\int_{\partial(\pi_YC)_*[\ds])}(H\circ\fp)\lambda
\end{align*}
which equals the right hand side of (\ref{eqn:dHlambda}).
\end{proof}

As a consequence of Proposition \ref{prop-degree}, the degree of any two curves in $\mathcal{M}^J(\alpha,\beta)$ must be equal, and therefore we make the following definition.
\begin{definition}\label{degree-gen} The \textit{degree} of a pair of ECH generators $(\alpha,\beta)$, denoted $\op{deg}(\alpha,\beta)$, is
\[
\op{deg}(\alpha,\beta):=\frac{M-N}{|e|}.
\]
\end{definition}

A curve $C$ which contributes nontrivially to the ECH differential is degree zero if and only if it is a cylinder, as we explain in the following Lemma.

\begin{lemma}\label{lem:d0g0} Let $C\in\mathcal{M}^J(\alpha,\beta)$ be a $J$-holomorphic curve with domain $(\ds,j)$ and with $I(C)=1$. Then $\op{deg}(C)=0$ if and only if $\ds$ a cylinder.
\end{lemma}
\begin{proof}
If $\ds$ is a cylinder and $\op{deg}(C)>0$, then the completion of $\fp\circ\pi_YC$ has $S^2$ as its domain. The Riemann-Hurwitz formula (\ref{eqn:RH}) tells us
\begin{equation}\label{eqn:RHcontra}
2=\op{deg}(C)\chi(\Sigma_g)-\sum_{p\in\ds}(e(p)-1),
\end{equation}
where $(e(p)-1)$ is the ramification index at $p$. Since we are assuming $g\geq1$, the right hand side is at most zero, therefore (\ref{eqn:RHcontra}) is a contradiction. Thus $\deg(C)=0$, so that Riemann-Hurwitz does not apply.

For the opposite implication, assume $\op{deg}(C)=0$. Because $I(C)=1$, Proposition \ref{lowiprop} implies that either $C$ is a trivial cylinder (in which case we are done, so we will assume $C$ is not a trivial cylinder), or $C$ is embedded and has $\op{ind}(C)=1$. By the ECH index inequality, Theorem \ref{thm:indexineq}, the positive and negative partitions of the ends of $C$ must equal the positive and negative partitions defined in Definition \ref{def:part}. Since we do not want to assume that $J$ is generic, in this setting the argument carries through as explained in \cite[Rmk. 5.10]{calabi}. One should argue as in Step 2 of the proof of \cite[Lem. 5.23]{calabi}.

By Example \ref{ex:partitions} and the analogous result for $\vartheta$ slightly smaller than zero, i.e.
\[
\begin{array}{lcl}
P_\vartheta^+(m) &=& (m) \\
P_\vartheta^-(m) &=& (1,\dots,1) \\
\end{array}
\]
and the fact that in admissible Reeb currents hyperbolic orbits have multiplicity at most one, we find that $\ds$ has exactly $1+m_1+\cdots+m_{2g}+m_+$ positive ends and exactly $n_-+n_1+\cdots+n_{2g}+1$ negative ends. Moreover,
\[
CZ_\tau^{ind}(C)=(m_+-1)-(1-n_-).
\]
Therefore,
\begin{align}
1&=\op{ind}(C)\nonumber
\\&=-2+2g(\ds)+M-m_-+1+N-n_++1+0+m_+-1-1+n_-\text{ by (\ref{eqn:relfirstChern})}\nonumber
\\&=-2+2g(\ds)+2M-m_--n_++0+m_++n_-\text{ because $M=N$}.\label{eqn:1ind}
\end{align}

Note that $I(C)=1$ and $\frac{M-N}{|e|}=\op{deg}(C)=0$, along with the index formula (\ref{eqn:indexformula}), imply
\begin{equation}\label{eqn:d0I1}
1=I(C)=m_+-m_--n_++n_-.
\end{equation}

Combining (\ref{eqn:1ind}) and (\ref{eqn:d0I1}) gives
\[
0=-2+2g(\ds)+2M \Leftrightarrow 1=g(\ds)+M.
\]

Since $M>0$, we must have $M=N=1$ and $g(\ds)=0$, therefore $\ds$ must consist of a union of cylinders.
\end{proof}

In particular, fiberwise $S^1$-invariant $\fj$-holomorphic cylinders have degree 0.  We have the following correspondence between $\fj$-holomorphic cylinders asymptotic to Reeb orbits which project to critical points of $H$ and downward gradient flow lines of $H$; this will allow us to relate the filtered ECH differential to the Morse differential on the base.

\begin{proposition}\label{cylinder-to-morse}
For suitable orientation choices, if $p$ and $q$ are critical points of $H$, then there is an orientation-preserving bijection between the moduli space of $\fj$-holomorphic cylinders $\M^{\fj}(\gamma_p^k,\gamma_q^k)$ and the moduli space $\M^{\mbox{\tiny Morse}}(p,q)$ of downward gradient flow lines of $H$ from $p$ to $q$, modulo reparametrization. Furthermore, each of the holomorphic cylinders is a $k$-fold cover which is cut out transversely. 
\end{proposition}
Complete details on this correspondence are found in \cite[\S 5.1]{jo2} and \cite[\S 3.5.1, \S 6.1]{moreno}; that the multiple covers are cut out transversely requires automatic transversality, cf. \cite[\S 4.1-4.2]{dc}, \cite{Wtrans}.   In ECH, we will only have somewhere injective curves, because the ECH chain complex does not admit hyperbolic orbits with multiplicity greater than one as generators.

\begin{remark}
If a $J$-holomorphic curve $C: \ds \to \R \times Y$ has {degree} $d=0,1$ then can prove directly or appeal to automatic transversality that a regular $S^1$-invariant $J$ exists.   If $d \geq 1$, the projection $\pi_Y C$ has intersection number $d$ with a given $S^1$-orbit, and hence has at least $d$ intersections, which are counted with multiplicity, since $C$ could be a nontrivial branched covering of its image.  The complex structure $\fj$ cannot be independently perturbed at these $d \geq 2$ points by an $S^1$-invariant perturbation.  By Lemma \ref{lem:d0g0}, a curve $C$ which contributes nontrivially to the ECH differential is degree zero if and only if it the union of cylinders.
 \end{remark}


\subsection{Domain dependent almost complex structures}\label{sec:ddacs}
Let $\sj(Y,\lambda)$ denote the set of $\lambda$-compatible almost complex structures and let $\sj^{S^1}(Y,\lambda) \subset \sj(Y,\lambda)$ denote the subset of $S^1$-invariant $\lambda$-compatible almost complex structures.  Since $\fj \in \sj^{S^1}(Y,\lambda)$ is always obtained from $\fp^*j_{\Sigma_g}$, there is a correspondence between the $S^1$-invariant complex structures on $\xi$ and the complex structures on $T\Sigma_g$.  Fix a ``generic" $\fj_0 \in \sj^{S^1}(Y,\lambda) $,  let $\{ N_\gamma\}$ be a disjoint union of tubular neighborhoods associated to the set of Reeb orbits $\{ \gamma \}$, and set $N = \sqcup_\gamma N_\gamma$.  We define 
\begin{equation}\label{JNs}
\sj_N : = \{ J \in \sj(Y,\lambda) \ | \ J_q = (\fj_0)_q, \mbox{ for } q\in N \},
\end{equation}
to be the subset of $\lambda$-compatible almost complex structures which agree with $\fj_0$ on $N$, and let $\sj^{S^1}_N  \subset \sj_N $ consist of the $S^1$ invariant elements of $\sj_N$. The elements of $\sj_N $ are in correspondence with complex structures on $T\Sigma_g$ which agree with a fixed $\fj_0 = \fp^*j_0$ on $N$.   We first note that we have regularity for ECH index 1 curves for generic $\lambda$-compatible $J$ satisfying the constraint $J|_N = \fj_0$.

\begin{proposition}\label{rem:Js}
Let $\alpha$ and $\beta$ be nondegenerate admissible Reeb currents in the same homology class with $I(\alpha,\beta)=1$.  If $J \in \sj_N$ is generic then $\M^{J}(\alpha, \beta)$ is cut out transversely.
 \end{proposition}

\begin{proof}
Regularity follows from the subclaim in the proof of \cite[Lemma 9.12]{Hindex}.  We have that in the absence of trivial cylinders, there is a nonempty set $U \subset C$ away from a neighborhood of of the periodic orbits with  action $\leq \mathcal{A}(\alpha)$, such that for each $x \in U$, $\pi_Y^{-1}(\pi_Y(x))= \{x \}$, $C$ is nonsingular, and that a certain projection of the derivative of $\overline{\partial}_{j,J}(C)$ with respect to $J$ is surjective on $U$.  The proof actually shows that $U$ is an open dense subset of $C$, and because the intersection with $C^{-1}(\R \times (Y \setminus N))$ contains a nonempty open set, the result holds.
\end{proof}

When $\ds \neq \R \times S^1$, it is not possible to achieve regularity via a generic choice of $S^1$-invariant $\fj$.  Instead, we will use an $S^1$-invariant domain dependent family $\bj$ of $\lambda$-compatible almost complex structures.  To define such a $\bj$, we consider a certain class of functions on the domain $\ds$ that are independent of reparameterization, meaning that these functions are to be defined on isomorphism classes of punctured Riemann surfaces, e.g. elements of $\mgn$, where we view the punctures $\mathcal{S}$ of $\ds$ as the $n$ marked points and $g = g(\ds)$.  

\subsubsection{Preliminaries relating to $\mgn$}

To ensure that we have well-defined, nontrivial functions on $\ds$ (from which we will construct domain dependent almost complex structures), we need $\ds$ to be \emph{stable}, meaning that $\chi(\ds) < 0$, e.g. $2g+n \geq 3$, where $n$ is the number of punctures of $\ds$.   If $\ds$ is stable then $\mgn$ is an orbifold with
\[
\mbox{dim}(\mgn) = 6g - g + 2n,
\]
and the associated automorphism group  
\[
\Aut(\ds, j):=\{ \varphi \in \mbox{Diff}_+(\Sigma, \mathcal{S}) \ | \ \varphi^* j =j \}
\]
is finite for any $j \in \sj(\Sigma)$.  Here $\sj(\Sigma)$ is the set of smooth complex structures on $\Sigma$ that induce the given orientation and $\mbox{Diff}_+(\Sigma, \mathcal{S})$  is the group of orientation preserving diffeomorphisms on $\Sigma$ that fix the set of punctures $\mathcal{S}$.  Our class of domain dependent almost complex structures on $\ds$ must respect the orbifold structure of $\mgn$, meaning they must be invariant with respect to the finite symmetry groups of the orbifold points of $\mgn$.  While the derivative of such an invariant function (giving rise to $\bj$) will have nontrivial kernel at an orbifold point of $\mgn$, the set of orbifold points has positive complex codimension, meaning that there is sufficient flexibility in the normal direction.

\begin{remark}\label{unstable-cases}
Following \cite[\S 3.1]{Wtrans}, and because our pseudoholomorphic curves must all have at least one puncture, we conclude that the nonstable cases are $\ds = \R \times S^1$ and $\ds = \C$.  We previously showed that we can use a domain independent $S^1$-invariant $\fj$ to count  pseudoholomorphic cylinders, so it remains to consider when $\ds = \C$.  Since any pseudoholomorphic map $C: \C \to \R \times Y$ must be asymptotic to a Reeb orbit $\gamma$, we can consider its \emph{projected completion} to the base $(\Sigma_g, \omega)$ of the prequantization bundle 
\[
\overline{C}: S^2 \to \Sigma_g, 
\]
which is null homotopic for $g>0$.  Since $\overline{C}$ is null homotopic when $g>0$,  if we consider a sufficiently small perturbation, $\overline{C}$ must be close to constant, which means that $C$ is ``concentrated" near its limiting Reeb orbit, and thus cannot bound $\overline{C}$, because otherwise we would obtain a contradiction to the fact that far away fibers of $\fp: (Y,\lambda) \to (\Sigma_g, \omega)$ are linked. 
For this reason, in \S \ref{subsec:pfmainthm} we split the proof of Theorem \ref{thm:mainthm} into two cases, depending on whether the  genus of the base of the prequantization bundle $\Sigma_g$ is zero or positive.    When $g=0$, we show that the differential vanishes for grading reasons, which permits us to use a generic $\lambda$-compatible $J$ without appealing to the results of \S \ref{sec:modspcs}-\ref{handleslides}.
\end{remark}

In order to vary a domain dependent $\bj$ and take limits of sequences of $\bj$-holomorphic curves in the sense of \cite{BEHWZ}, we must actually work with the Deligne-Mumford compactification $\overline{\mgn}$, a compact and metrizable topological space containing $\mgn$ as an open subset,  which consists of connected stable nodal Riemann surfaces with $n$ marked points, (presupposing that $2g+n \geq 3$).

\begin{definition}
Recall that an element of $\mgn$ is \emph{stable} whenever $2g + n \geq 3$.  A \emph{stable nodal Riemann surface} is an element $(\ds, j) \in \overline{\mgn}$, which is itself a disjoint union of elements $(\ds_i,j_i) \in \sm_{g_i, n_i+m_i}$, where $\ds_i$ is a stable curve whose $n_i+m_i$ marked points consist of a subset $n_i$ of the marked points of $\ds$ (hence $\sum n_i = n$), with the induced ordering, and $m_i$ nodes.  Every node $z \in \ds_i$ is paired with another node $z' \in \ds_i'$, with the stipulation that $i'\neq i$ for at least one of the nodes of each $C_i$.  We thus obtain a connected singular surface by gluing $z$ to $z'$ for every pair $\{ z, z' \}$ of nodes.
\end{definition}

Any sequence of curves $\{\ds(k)\} \in \mgn \subset \overline{\mgn}$ has a subsequence whose limit is a nodal curve $\ds \in \overline{\mgn}$.  Furthermore, if $z_i(k) \in \ds(k)$ is a marked point, passing to a subsequence means $z_i(k)$ converges to some marked point $z\in\ds$, hence $z \in \ds_i$, where $\ds_i \in \sm_{g_i, n_i+m_i}$ for some $\ds_i \subset \ds$.  We recall the following result regarding the topology of the nodal limit, noting that further details can be found in \cite[\S 9.3.3]{wendl-sft} and \cite{ss92}.

\begin{lemma}[Lem.\ 6.1.1 \cite{farris}]\label{lem:mgn}
If $\ds \in \overline{\mgn}$ then for each component $\ds_i \in \sm_{g_i,n_i+m_i}$ of $\ds$ we have that $g_i \leq g$ and if $g_i=g$ then $n_i+m_i < n$.  \end{lemma}

\begin{proof}
The nodal limit $\ds$ is obtained topologically from any smooth sequence $\{ \ds(k) \}$ via the following types of degenerations.  The first degeneration is that $\ell$ marked points in some component $\ds_j$ can collide and form a bubble attached to $\ds_j$.  The genus of $\ds_j$ does not change, but it loses $\ell$ marked points and gains a node where the bubble arises.  Thus, the total number of marked and nodal points on $\ds_j$ decreases by $\ell-1$.  The bubble itself is a genus 0 component with $\ell$ marked points and one node.  If the original smooth curve $\ds(k_0)$ had genus 0, then, then it must have had more than $\ell$ marked points.  Thus each new component resulting from iterative bubbling has either genus 0 or $g$.  

The second kind of degeneration comes from letting the complex structure on the curve degenerate.  Topologically, this results in a simple closed curve on some component, i.e. the vanishing cycle, being crushed to a point.  If the vanishing cycle is a non-separating curve, it reduces the genus of a component by 1 without creating new components.  If the vanishing cycle is a separating curve, it breaks a component into two pieces, whose genera sum to the genus of the original component.  The case where one component has genus 0 and the other has full genus is topologically identical to bubbling, cf. \cite{jhol}.
\end{proof}

Lemma \ref{lem:mgn} induces an ordering on pairs $(g,n)$ wherein $(g',n') < (g,n)$ whenever $g' < g$ or  $g' =g$ and $n' < n$.  The boundary $\partial \overline{\mgn}$ is a stratified space, and each stratum containing a nodal curve $\ds$ is the product over the $\sm_{g',n'}$ for each component $\ds_i$ of $\ds$ with $\ds_i \in \sm_{g',n'}$. Moreover, we have that if $\sm_{g',n'}$ is a factor in a stratum of $\mgn$, then $(g',n') < (g,n)$ and we can distinguish one of the components of a given $\ds \in \overline{\mgn}$ whenever it contains the $n^{th}$ marked point, which we will always denote as $z_0$.  This prescribes an inductive means of coherently defining functions on all strata of $\overline{\mgn}$ simultaneously, though before we get into this we need to briefly review \cite[\S 3]{CM}.

\subsubsection{Coherent maps}
In the following discussion, we consider the Deligne-Mumford space $\overline{\sm_{g,n+1}}$ for $2g+n \geq 3$, with $n+1$-marked points $z_0,...,z_{n}$.  We will see momentarily that the point $z_0$ plays a special role, as it serves as the variable for holomorphic maps and is key in the proof of Theorem~\ref{thm:ddacs}. 

First consider the case when $g=0$.  We call a decomposition $\mathbf{I}=(I_0,...,I_\ell)$ of $\{0,...,n \}$ \emph{stable} whenever $I_0 = \{ 0\}$ and $|\mathbf{I}| := \ell+1 \geq 3$.  We will always order the $I_j$ such that the integers $i_j:=\min \{ i \ | \ i \in I_j \}$ satisfy
\[
0 = i_0 < i_1 < ... < i_\ell.
\]
Denote by $\sm_\mathbf{I} \subset \overline{\sm_{0,n+1}} $ the union over stable trees that give rise to the stable decomposition  $\mathbf{I}$. The $\sm_\bfi$ are submanifolds of $\overline{\sm_{0,n+1}} $ with 
\[
\overline{\sm_{0,n+1}} = \cup_{\bfi}\sm_\bfi 
\]
and the closure of $\sm_\bfj$ is a union of certain strata $\sm_\bfi$ with $|\bfi| \leq |\bfj|$.  The above ordering of the $I_j$ determines a projection
\[
p_\bfi: \sm_\bfi \to \sm_{|\bfi|},
\]
sending a stable curve to the special points on the component $S_{\alpha0}$.
\begin{definition}\cite[Definition 1.3]{CM}
Let $Z$ be a Banach space and $n \geq3$.  We call a continuous map $J_{0,n+1}: \overline{\sm_{g,n+1}} \to Z$ \emph{coherent} if it satisfies the following two conditions:
\begin{enumerate}[(a)]
\item $J_{0,n+1} \equiv 0$ in a neighborhood of those $\sm_{\mathbf{I}}$ with $|\mathbf{I}|=3$;
\item For every stable decomposition $\mathbf{I}$ with $|\mathbf{I}|\geq4$, there exists a smooth map $J_\bfi: \sm_{0,|\bfi|} \to Z$ such that
\[
J|_{\sm_\bfi} = J_\bfi \circ p_\bfi: \sm_\bfi \to Z.
\]
More generally, let $Z^* \subset Z$ be an open neighborhood of 0 and let $\mathcal{I}$ be a collection of stable decompositions.  Then we call a continuous map $J: \cup_{\bfi \in \mathcal{I}}\sm_\bfi \to Z^*$ coherent if it satisfies (a) and (b) and in addition:
\item The image of $J$ is contractible in $Z^*$.
\end{enumerate}
\end{definition}
The space of coherent maps from $\overline{\sm_{0,n+1}}$ to $Z$ is equipped with the $C^0$-topology on $\overline{\sm_{0,n+1}}$ and the $C^\infty$-topology on each $\sm_\bfi$ via the projection $p_\bfi$.

For $g>0$, we deploy Lemma \ref{lem:mgn}, obtaining an inductive means of coherently defining functions on all strata of $\overline{\mgn}$ simultaneously.  In particular, assume that for all $(g',n') < (g,n)$ we have defined continuous maps
\[
J_{g',n'}: \sm_{g',n'} \to Z.
\]
  Each element of $\partial \overline{\mgn}$ is a nodal curve $\ds$ with $n$ marked points.  If $p_n$ lies on $\ds' \in \sm_{g',n'}$, then by Lemma \ref{lem:mgn}, $(g',n') < (g,n)$, and by hypothesis there is a function $J_{g',n'}$ defined on $\sm_{g',n'}$.  We can define 
\[
J_{g,n}(\ds) : = J_{g',n'}(\ds').
\] 
The collection $\{J_{g',n'}\}_{(g',n') < (g,n)}$ thus determines $\{ J_{g',n'}\}_{\overline{\mgn}}$ and we can extend $J_{g,n}$ to the interior $\mgn$.  We can continue this procedure, defining $J_{g,N}$ on $\sm_{g,N}$ for all $N>n$, then $f_{g+1}, n$ for all $n$, and so on.  This inductive procedure provides the definition for our continuous maps with $g>0$ to be \emph{coherent}.  

{Before we can use this class of to define domain dependent almost complex structures, we need to review a few details regarding Banach manifolds of (parametrized) almost complex structures.}

\subsubsection{Banach manifolds of almost complex structures}
For a symplectic vector space $(V, \omega)$, denote by $\sj(V,\omega)$ the space of $\omega$-tamed almost complex structures.  The space $\sj(V,\omega)$ is a manifold with tangent space
\[
T_{J_0}\sj(V,\omega) := \{ \mathcal{Y} \in \op{End}(V) \ | \ J_0 \cy J_0 = \cy \}.
\]
If $V$ is equipped with a Euclidean metric $g$ then $\op{trace}(\cy^t\cy)$ defines a Riemannian metric on $\sj(V,\omega)$.  The exponential map of this metric defines embeddings from the open ball of radius $\rho(g,J)>0$ which continuously depend on $g$ and $J$:
\[
\exp_J:T_J \sj(V,\omega) \supset B(0, \rho(g,J)) \hookrightarrow \sj(V,\omega).
\]
We review the construction of the \emph{Floer $C^\varepsilon$-space } \cite{floer}, which circumvents the issue that naturally arising spaces of smooth functions are not Banach spaces; see also \cite[Appendix B]{wendl-sft}.  For a vector bundle $E \to X$ over a closed manifold $X$, we denote the space of Floer's $C^\varepsilon$-sections in $E$ by
\[
C^\varepsilon(X,E):= \left \{ s \in \Omega(X,E) \ \bigg \vert \ \sum_{i=1}^\infty \varepsilon_i || s||_{C^i} < \infty \right \}.
\]
Here $\Omega(X,E)$ is the space of smooth sections of $E$, $\varepsilon = (\varepsilon_i)_{i \in \N}$ is a fixed sequence of positive numbers and $||\cdot||_{C^i}$ is the $C^i$-norm with respect to some connection on $E$.  By \cite[Lemma 5.1]{floer}, if the $\varepsilon_i$ converge sufficiently fast to zero, then $C^{\varepsilon}(X,E)$ is a Banach space consisting of smooth sections with support in arbitrarily small sets in $X$.  

Next let $(X,\omega)$ be the symplectization of a closed contact manifold $Y$ (or an exact symplectic cobordism).  Fix a $\lambda$-compatible almost complex structure $J_0$ on $(X,\omega)$, e.g. a smooth section in the bundle $\sj(TX,\omega) \to X$ with fibers $\sj(T_xX,\omega_x)$.  Let $g$ be the canonical Riemannian metric on $X$ defined via $\omega$ and $J$.  Let $T_{J_0}\sj(TX,\omega) \to X$ be the vector bundle with fibers $T_{J_0(x)}\sj(T_xX,\omega_x)$ and set
\[
\begin{array}{rcl}
T_{J_0}\sj^\varepsilon & : =& C^\varepsilon(X, T_{J_0}\sj(TX,\omega) ) \\
\sj^\varepsilon &:=& \sj^\varepsilon(X, \omega) \ := \ \exp_{J_0}(B), \\
\end{array}
\]
where $B:= \{ \cy \in T_{J_0}\sj^\varepsilon \ | \ \cy(x) \in B(0, \rho(g(x),J_0(x)))  \}.$  Thus $\sj^\varepsilon$ is the space of $\lambda$-compatible almost complex structures of $(X,\omega)$ that are of class $C^\varepsilon$ over $J_0$ via $\exp_{J_0}$.  We can regard $\sj^\varepsilon$ as a Banach manifold with a single chart $\exp_{J_0}$.

Next we consider spaces of parametrized complex structures.  A \emph{complex structure on $(X, \omega)$ parametrized by a manifold $P$} is a smooth section in the pullback bundle $\sj(TX, \omega) \to P \times X$.  Fix $J_0$ as above and let $T_{J_0}\sj_P(TX, \omega) \to P \times X $ be the vector bundle with fibers $T_{J_0(p,q)}\sj(T_xX, \omega_x)$ and set 
\[
\begin{array}{rcl}
T_{J_0}\sj_P^\varepsilon & := & C^\varepsilon(P \times X, T_{J_0}\sj_P(TX,\omega) ) \\
\sj^\varepsilon_P &:=& \sj_P^\varepsilon(X, \omega) \ := \ \exp_{J_0}(B_P), \\
\end{array}
\]
where $B_P:= \{ \cy \in T_{J_0}\sj^\varepsilon_P \ | \ \cy(p, x) \in B(0, \rho(g(x),J_0(x)))  \}$.  We may think of $J \in \sj^\varepsilon_P$ as a map $P \to \sj^\varepsilon$.  For an open subset $U\subset P$, we denote by $T_J\sj^\varepsilon_U \subset T_J\sj^\varepsilon_P$ the subspace of those section having compact support in $U$.  

We will be interested in the spaces of domain dependent almost complex structures
\[
\sj^\varepsilon_{\ds}:=\sj^\varepsilon_{\overline{\mgn}} \mbox{ and } \sj_{\ds}:=\sj_{\overline{\mgn}}
\]
parametrized by the Deligne-Mumford space $\overline{\mgn}$. 

\begin{definition}
A \emph{domain dependent almost complex structure} $\bj$ is a coherent collection of $C^\ell$, $\ell >0$ maps
\[
\bj = \{ J_{g,n}: \overline{\mgn} \to \sj_N \},
\]
(recall $\sj_N$ was defined in \eqref{JNs}) that additionally satisfy the following condition. If given a sequence $\{\ds(k)\} \in \overline{\mgn}$ converging to $\ds_\infty \in \partial \overline{\mgn}$ and $\ds^n_\infty \in \sm_{g',n'}$ is the component of $\ds_\infty$ containing the $n^{th}$ marked point, then 
\[
\lim_{k \to \infty}J_{g,n}(\ds(k)) = J_{g,n}(\ds_\infty) =J_{g,n}(\ds^n_\infty).
\]
{We denote the set of all such $C^\ell$ domain dependent almost complex structures by $\sj_\ds^\ell$.}    We call a domain dependent almost complex structure $\bj$ \emph{generic} if for every $(g,n)$, the extension of $J_{g,n}$ from the boundary, where the values are determined by $J_{g',n'}$, to the interior of $\mgn$ is a generic $C^\ell$ map.  
\end{definition}
The extension to nodal maps follows from \cite[\S 5]{CM}.    The previous discussion guarantees that $\sj_\ds^\varepsilon$ is a Banach manifold.   When it is understood that we should be using the Floer $C^\varepsilon$-space we will drop $\varepsilon$ from the notation, and use $\sj_\ds$.  

\begin{remark}\label{ddacsintpos}
Since the target of our collection of functions on $\ds$ is $\sj_N$, we have that if $\bj \in \sj_\ds$ and if
\[
C: (\ds, j) \to (\R \times Y, \bj)
\]
is a $(j,\bj)$-holomorphic curve, then $C \vert_{C^{-1}(\R \times N)}$ is $(j,\fj_0)$-holomorphic.  Since $\fj_0$ is domain independent, the subset $C(\ds) \cap \R \times N \subset C(\ds)$ satisfies intersection positivity, which will be important in the proof of Proposition \ref{notopbottomconn}.
\end{remark}

Before we can conclude that this algorithm for constructing domain dependent almost complex structures is well-defined, it remains to discuss two technicalities, the first being that $\mgn$ is an orbifold, while the other concerns the special role of the ``last" marked point $z_0$.  We elucidate these points in the following remarks.

\begin{remark}[Orbifold structure of $\mgn$]
A neighborhood of a point in an $k$-dimensional orbifold is modeled on the quotient of $\R^k$ by the linear action of some finite group $G$, and a $C^\ell$ function on an orbifold in a neighborhood modeled on $\R^k/G$ is a $C^\ell$ function on $\R^k$ which is invariant under the group action $G$.    For $g>1$, the locus of points on $\mgn$ without automorphisms (e.g. the action of $G$ on $\R^k$ is nontrivial) has real codimension at least two.  Thus, a generic curve of genus $g >1$ has no nontrivial automorphisms.  For $g=0$, every stable curve has a trivial automorphism group.  For $g=1$ and $n=1$, $\mbox{dim}_\R\sm_{g,1} =2$, and a generic elliptic curve has an involution and isolated points in $\sm_{g,1}$ have additional automorphisms, hence functions on $\sm_{g,1}$ have no constraints at generic points and respect the additional symmetries at the isolated points admitting the extra automorphisms. 

{The derivative of a $G$-invariant function always has nontrivial kernel, but on any tangent space $T_z\sm_{g, n+1}$, there is a subspace of real dimension at least two on which the derivative has no constraints.  Hence there exists a map from a neighborhood of any point in $\sm_{g, n+1}$ to a neighborhood of any point $x$ in a manifold $X$ with $\dim_\R X \geq 2,$ sending a two dimensional subspace of the unconstrained subspace to any two dimensional subspace of $T_xX$.  If we fix $j$ and the  marked points $z_1,...,z_n$ on $\sm_{g, n+1}$, we may view this as a map $T_{z_{0}}\ds \to T_xX$.}
\end{remark}

\begin{remark}[The role of the special marked point $z_0$]
Recall that there is a forgetful map
\[
\pi: \overline{\sm_{g,n+1}} \to \overline{\mgn}
\]
which forgets the special marked point $z_0$ and collapses any resulting unstable components, which are necessarily of genus zero.  The fiber over $\ds \in \overline{\mgn}$ is itself isomorphic to $\ds$.  To see why this holds over the marked and nodal points, note the following.  The fiber above the $k^{th}$ marked point $z_k$ is a single nodal curve which has a genus zero component containing the marked points $z_k$ and $z_0$ as well as a node, which is glued to $z_k\in \ds$.  This component collapses when the marked point $z_0$ is removed.  The fiber above a node resulting from gluing $z' \in \ds'$ to $z''\in \ds''$ is a single curve which has a genus zero component containing two nodes and the marked point $z_0$, attached by the first node to $\ds'$ at $z'$ and to $\ds''$ at $z''$ by the second node.  This genus zero component similarly collapses when the marked point is removed.  Thus, a point of $\overline{\sm_{g,n+1}}$ is equivalent to a pair $(\ds,z)$, where $\ds\in \overline{\mgn}$ and $z\in \ds$.  

If $\ds \in\overline{\mgn}$, we can delete the first $n$ marked points to obtain an $n$-times punctured Riemann surface with one marked point $z_0$.  Fix the $n$ marked points corresponding to the $n$ punctures and let $\bj:=\{ J_{g',n'}\}$ be a domain dependent almost complex structure.  By restricting $J_{g,n+1}$ to $\ds \cong \pi^{-1}(\ds) \subset \sm_{g,n+1}$, we obtain a family of almost complex structures on $\xi$ parametrized by $\ds$, which we denote by $\bj_\ds$.  Rather than writing $\bj_\ds$, we will work under the assumption that in the Cauchy-Riemann equation below, the domain of $\bj$ is restricted to $\pi^{-1}(\ds)$, where $(\ds, j) \in \mgn$.  Returning to the perspective of $\ds$ as a $n$-times punctured Riemann surface, a map $C: (\ds, j) \to (\R \times Y, \bj)$ can be associated with the $(0,1)$-form
\[
\overline{\partial}_{j,\bj} : = \tfrac{1}{2}\left( dC + \bj(z,C)\circ dC \circ j\right),
\]
which at the point $z\in \op{int}(\ds)$ is given by
\[
\overline{\partial}_{j,\bj}C(z) := \tfrac{1}{2} \left( dC(z) + \bj(z,C(z))\circ dC(z) \circ j ( z) \right).
\]
We say that $C$ is \emph{$\bj$-holomorphic} whenever $\overline{\partial}_{j,\bj}C=0$.  
\end{remark}


\subsection{Regularity for generic $S^1$-invariant domain dependent $\bj$}\label{sec:reg}
In this section, we prove that a generic $S^1$-invariant domain dependent almost complex structure $\bj$ is regular.  We note that the weaker statement that a generic $\bj \in \sj_\ds$ is regular follows similarly.

\begin{theorem}[Thm.\ 6.2.1 \cite{farris}]\label{thm:ddacs}
Let $\alpha$ and $\beta$ be nondegenerate Reeb currents with $\op{deg}(\alpha,\beta) >0$.  If $\bj \in \sj_\ds^{S^1}$ is generic and $\ds$ does not include $\C$ or a union of cylinders, then any nonconstant holomorphic curve $C \in \M^{\bj}(\alpha, \beta)$ is regular, meaning that the linearization $D\overline{\partial}_{\bj}(C)$ is surjective and a neighborhood of $C \in  \M^{\bj}(\alpha, \beta)$ naturally admits the structure of a smooth orbifold of dimension given by the Fredholm index $\ind(C)$, whose isotropy group at $C$ is 
\[
\op{Aut}(C) := \{ \varphi \in \op{Aut}(\ds, j) \ | \ C = C \circ \varphi  \}
\]
and there is a natural isomorphism
\[
T_C \M^{\bj}(\alpha, \beta) = \ker D\overline{\partial}_{\bj}(j, C) / \mathfrak{aut}(\ds, j).
\] 
\end{theorem}

\begin{remark}
At an orbifold point $C \in \M^{\bj}(\alpha, \beta)$,  $C$ has a nontrivial automorphism group with respect to which $C$ is invariant, so $C$ factors through the branched covering $\ds \to \frac{\ds}{\op{Aut}(\ds, j)}$.  Additional multiple covers may arise which do not come from automorphisms of the domain, but the use of domain-dependent almost complex structures permits us to perturb away the multiple covers of the latter type by choosing different perturbations at the different points in $C^{-1}(C(q))$.    However, multiple covers coming from automorphisms of the domain remain because the functions from which we defined domain dependent almost complex structures are invariant with respect to the orbifold symmetry groups of $\mgn$.  Since the subset of orbifold points of $\overline{\mgn}$ has real codimension at least 2 in $\overline{\mgn}$, we can conclude that the subset of $\bj$-holomorphic curves in the moduli space whose domains are orbifolds also has real codimension at least 2.  Thus a generic $\bj$-holomorphic curve is not an orbifold point in its moduli space, and a generic path of $\bj$-holomorphic curves avoids the locus of orbifold points.
\end{remark}

Before giving the proof, we provide the corollary, which demonstrates that positive degree curves do not contribute to $ECH$ index 1 moduli spaces.

\begin{corollary}[Cor.\ 6.2.3 \cite{farris}]\label{nogenus}
Let $\alpha$ and $\beta$ be nondegenerate admissible Reeb currents and $\bj \in  \sj_\ds^{S^1}$ be generic.  If $\op{deg}(\alpha,\beta) >0$ and $I(\alpha,\beta)=1$ then $\M^{\bj}(\alpha,\beta) = \emptyset$.
\end{corollary}
\begin{proof} Given a generic $\bj \in \sj_\ds^{S^1}$ consider a $\bj$-holomorphic curve $C: \ds \to \R \times Y$ with $\op{deg}(\alpha,\beta) >0$.  Take $\bj' \in \sj_\ds$ to be generic, then $\ind(C_\bj) = \ind(C_{\bj'})$.  Moreover, if we take $J \in \sj_{reg}(Y,\lambda)$ to be sufficiently close to $\bj'$ then $\ind(C_{\bj'}) = \ind(C_{J})$.  

By the definition of degree, the domain cannot be a union of cylinders, so $S^1$ acts locally freely on $\M^{\bj}(\alpha,\beta) $.  Since $\R$ acts freely on $\M^{\bj}(\alpha,\beta)$ and, because these actions commute, we have that $\dim \M^{\bj}(\alpha, \beta) \geq 2$ whenever $\M^{\bj}(\alpha, \beta) \neq \emptyset$.   This is because Theorem \ref{thm:ddacs} guarantees that $\M^{\bj}(\alpha, \beta)$ is cut out transversely and has dimension equal to the Fredholm index.  But by the ECH index inequality property, Theorem \ref{thm:indexineq}, we have $\ind(C_{\bj}) = \ind(C_J) \leq I(C_J) = I(\alpha,\beta) =1$, a contradiction because if $I(\alpha,\beta)=1$ then $\op{ind}(C_J) \leq 1$, .   
\end{proof}

Prior to proving our main regularity result, we provide some definitions and construct the universal moduli space, mostly following \cite[\S 4-5]{CM} and \cite[\S 7.2]{wendl-sft}.  Assuming $kp>2$, let 
\[
\B^{k,p,\delta} :=\W(\ds,\R \times Y; \alpha,\beta) \subset C^0(\ds,\R \times Y)
\]
be the usual smooth, separable, and metrizable Banach manifold of {exponentially weighted Sobolev spaces of maps which are asymptotically cylindrical curves to the Reeb currents $\alpha$ and $\beta$ at the ends}.  The tangent space to $\B^{k,p,\delta}$ at $C\in \B^{k,p,\delta}$ can be written as 
\[
T_C\B^{k,p,\delta} = \W(C^*(\R \times Y) ) \oplus V_\cs,
\]
where $V_\cs \subset \Gamma(C^*(\R \times Y))$ is a non-canonical choice of a $2|\cs|$-dimensional vector space of smooth sections asymptotic at the punctures to constant linear combinations of the vector fields spanning the canonical trivialization of the first factor in $T(\R \times Y) = \epsilon \oplus \xi$.  The space $V_\cs$ appears due to the fact that two distinct elements of $\B^{k,p,\delta}$ are generally asymptotic to collections of trivial cylinders that differ from each other by $|\cs|$ pairs of constant shifts $(a,b) \in \R \times S^1$.

Fix $\bj \in \sj_\ds^{\varepsilon}$.  The nonlinear Cauchy-Riemann operator is then defined as a smooth section 
\[
\overline{\partial}_{j,\bj}: \B^{k,p,\delta} \to \E^{k-1,p,\delta}; C \mapsto TC + \bj \circ TC \circ j
\]
of a Banach space bundle
\[
\E^{k-1,p,\delta} \to \B^{k,p,\delta} 
\]
with fibers
\[
\E^{k-1,p,\delta}_C = W^{k-1,p,\delta}(\overline{\op{Hom}}_\C(T\ds, C^*(\R \times Y))).
\]
The zero set of $\overline{\partial}_{j,\bj}$ is the set of all maps $C \in \B^{k,p,\delta} $ that are $\bj$-holomorphic.  

More generally, the \emph{universal Cauchy-Riemann operator} is the section
\[
\overline{\partial}: \B^{k,p,\delta}  \times \sj_\ds^{\varepsilon} \to \E^{k-1,p,\delta}; (C, \bj) \mapsto \overline{\partial}_{j,\bj}C
\]
of a Banach space bundle 
\[
\E^{k-1,p,\delta}   \to   \B^{k,p,\delta} \times  \sj_\ds^{\varepsilon} 
\]
with fibers
\[
\E^{k-1,p,\delta}_C = W^{k-1,p,\delta}(\overline{\op{Hom}}_\C(T\ds, C^*(\R \times Y))).
\]
The zero section gives rise to the \emph{universal moduli space}:
\[
\M(\sj_{\ds}^{\varepsilon}):= \{ (C,\bj) \ | \ \bj \in \sj_{\ds}^{\varepsilon}, \ \overline{\partial}_{j,\bj}C=0  \}.
\]
Arguments similar to \cite[Lemma 7.15]{wendl-sft} demonstrate that the universal moduli space $\M(\sj_{\ds}^{\varepsilon})$ is a smooth separable Banach manifold, and the projection $\M(\sj_{\ds}^{\varepsilon}) \to \sj_{\ds}^{\varepsilon}; (C,\bj) \mapsto \bj$ is smooth.  

For any $C \in \overline{\partial}_{j,\bj}^{-1}(0)$, the linearization
\[
D\overline{\partial}_{j,\bj}: T_\bj \sj_\ds^{\varepsilon}  \times T_C\B^{k,p,\delta}  \to \E^{k-1,p,\delta}_C
\]
defines a bounded linear operator
\[
D: W^{k,p,\delta}(C^*T(\R \times Y)) \oplus T_\bj \sj_\ds^{\varepsilon} \oplus V_\cs \to W^{k-1,p,\delta}(\overline{\op{Hom}}_\C(T\ds, C^*(\R \times Y)))
\]
Since $V_\cs$ is finite dimensional,  $D$ will be Fredholm if and only if its restriction to the first two factors is Fredholm; denote this restriction by
\[
\D:= \D_C + \D_\bj :  W^{k,p,\delta}(C^*T(\R \times Y)) \oplus T_\bj \sj_\ds^{\varepsilon} \to W^{k-1,p,\delta}(\overline{\op{Hom}}_\C(T\ds, C^*(\R \times Y)))
\]

We will view the $n$ punctures of the domain $\ds$ of $C$ as fixed, with $j$ varying on $\op{int}(\ds)$, so that the tangent space to $\sm_{g,n+1}$ at a point $(\Sigma, j, z_0,z_1,...,z_n)$ is $T_j\sj(\ds) \oplus T_{z_0}\ds$.  If $V = (a, A) \in T_\bj(\sj_\ds^{S^1, \varepsilon})$, where $A: T_j(\ds) \to T_\bj \sj_\ds^{S^1, \varepsilon}$ and $a \in \overline{\op{End}}_j(T\ds)$, then 
\[
\D_\bj(V)=A \circ du \circ j_\ds + \bj \circ du \circ a.
\]

\begin{proof}[Proof of Theorem~\ref{thm:ddacs}.]
We begin by recalling a few observations in the proof of \cite[Theorem 6.2.1]{farris}.  {Since $\deg(\alpha,\beta)>0$, the domain $\ds$ cannot solely consist of a union of cylinders.}  The ECH index is additive and positive for pseudoholomorphic curves which are not themselves cylinders, so there is a unique noncylindrical component $\ds'$ of $\ds$. Trivial cylinders are always cut out transversely \cite[Proposition 8.2]{wendl-sft}, as are somewhere injective cylinders \cite[\S 7-8]{wendl-sft}.  In light of Remark \ref{unstable-cases} and without loss of generality, we may prove the theorem in the case when $\ds = \ds'$ and $\ds$ is stable.  Since $C$ is not a trivial cylinder, $C^{-1}(\R \times (Y \setminus N))$ contains a nonempty open set of $\ds$.  

Next we show that $C$ cannot be a nodal curve with a constant component of positive genus, which crucially relies on $\op{dim}(\R \times Y) = 4$, noting this is in part why \cite{CM} restricts to genus 0 curves.  Suppose to the contrary that $C$ is the union of a nodal curve $C_1$ with a constant component of positive genus $C_2$ and that $\ind(C)=1$, then
\[
\ind(C_1) + \ind(C)=1.
\]
Since $C|_{\ds_2}$ is constant, the restricted pullback $C^*(T(\R \times Y)|_{\ds_2}$ is trivial, thus
\[
c_1(C^*(T(\R \times Y))= c_1(C^*(T(\R \times Y))|_{\ds_1} + c_1(C^*(T(\R \times Y))|_{\ds_2} = c_1(C^*(T(\R \times Y))|_{\ds_1}.
\]
Denote
\[
c_1|_{\ds_1}:=c_1(C^*(T(\R \times Y))|_{\ds_1}.
\]
Because $C$ maps $\ds_2$ to a constant, all the punctures must lie on $\ds_1$.  Thus the Fredholm index contribution of the Conley-Zehnder indices of the orbits asymptotic to the ends of $\ds$ and $\ds_1$ must agree.  Denote this contribution by $CZ_\tau^{ind}(C_1)$.  By hypothesis, we have that $g(\ds_2) > 0$, $g(\ds_1) < g(\ds)$, and $\chi(\ds_1) > \chi(\ds)$.  Thus
\[
\begin{array}{r c l c l}
1&= &\op{ind}(C) &= & -\chi(C) + c_1|_{\ds_1} + CZ_\tau^{ind}(C_1), \\
&&\op{ind}(C_1)& = & -\chi(C_1) + c_1|_{\ds_1} + CZ_\tau^{ind}(C_1), \\
\end{array}
\]
hence $\op{ind}(C_1)<\op{ind}(C) =1$. Therefore $\op{ind}(C_2) = 1-\op{ind}(C_1) > 0$.  Since we assumed that $\bj$ is a generic $S^1$-invariant domain dependent almost complex structure, we have that all of its restrictions to $\partial \overline{\mgn}$ are generic, which determine the almost complex structure on $\ds_1$ and $\ds_2$.  {However, for generic almost complex structures, positive index curves of positive genus do not exist.} Thus $C$ is not constant on a component of positive genus.  Since constant components of genus 0 can be eliminated by reparametrization, we can assume without loss of generality that $C$ is not constant on any component of $\ds$, hence the zeros of $dC$ are isolated.  Note that the above argument also holds if $\ind(C)>1$.  
The remainder of the argument is similar to that of \cite[Lemmas 4.1, 5.4, 5.6]{CM}, \cite[Proposition 3.4.2]{jhol}, \cite[\S 8]{wendl-sft}, so we only sketch the argument.

Let $C: \ds \to \R \times Y$ be a $\bj$-holomorphic map.  The set of regular points $z$ of $\ds$ such that $\pi_Y C(z)$ is a regular value of $\pi_Y \circ C$ form an open dense subset of $\ds$.  If we intersect the set of regular points with the set of points $z \in \ds$ where $\op{im}(dC_z) = \xi_{\pi_YC(z)}$, it remains open and dense because the projection $\pi_Y$ is already open and dense by the nonintegrability of $\xi$.  Denote the further intersection of these sets with $C^{-1}(\R \times (Y\setminus N))$ by $U$.  Note that $U$ contains a nonempty open set.

After fixing $(C,\bj) \in \M(\sj_{\ds}^{S^1})$,\footnote{We are now dropping the $\varepsilon$ from the notation, as it is understood we should be working with the Floer $C^\varepsilon$-space.} we want to show that the linearization $\D= \D_C + \D_\bj$ is surjective.  Since $\D_C$ is Fredholm, $\D$ has closed range and hence surjectivity is equivalent to the triviality of the annihilator of $\op{Im}(\D)$.  
We prove the result for $k=1$, noting that $k>1$ follows by elliptic regularity, cf. \cite[Theorem C.2.3]{jhol}.  When $k=1$, we have that the dual space of any space of sections of class $L^{p,\delta}$ can be identified with sections of class $L^{q,-\delta}$ for $\frac{1}{p} + \frac{1}{q} =1$ \cite[Remark 7.7]{wendl-sft}.  Using the nondegenerate 2-form $d(e^r\lambda)$ on $\R \times Y$ we can use it to define a nondegenerate $L^2$-pairing
\[
\langle \cdot, \cdot \rangle = L^{p,\delta} \times L^{q,-\delta}.
\]
Moreover $\left( L^{p,\delta }\right)^* \cong L^{q,\delta}$, so we can consider the formal adjoint $\D_C^*$ of $\D_C$.  Let $\eta \in \op{coker}(\D)$, then the splitting and dualization yield that orthogonality to $\op{Im}(\D)$ amounts to the equations
\begin{equation}\label{eqs}
\begin{array}{cc}
\langle \D_C(\zeta),\eta\rangle =0 \\
\langle \D_\bj(V),\eta \rangle =0 \\
\end{array}
\end{equation}
for all $\zeta \in T_C \B^{k,p,\delta}$ and $V \in T_\bj(\sj_{\ds}^{S^1})$.  By the first equation, $\eta \in \op{ker}(\D_C^*)$.  By elliptic regularity, $\eta$ is smooth.  The second equation implies that if $\eta$ vanishes on an open set, then $\eta \equiv 0$ by unique continuation; cf. \cite[Lemma 3.4.7]{jhol}.  

The remainder of the argument is similar to the original proof by Farris.  Assume that $\eta_z \neq 0$ for some $z \in U \subset \ds$.  This implies that 
\[
\eta_z \in \overline{\op{Hom}}_{\bj(z,C(z))}(T_z\ds, T_{C(z)}T (\R \times Y)) \mbox{ and } dC_z \circ j_z \in \overline{\op{Hom}}_{\bj(z,C(z))}(T_z\ds, T_{C(z)}T (\R \times Y)) 
\]
are injective maps.  Thus given any $0\neq v \in T_z\ds$ we have that 
\[
\eta_z(v)\neq 0,\ \ \ dC_z \circ j_z(v) \neq 0. 
\]
Next we find some $A_z \in \overline{\op{End}}_{\bj(z,C(z))}(T (\R \times Y), d\lambda_{C(z)})$ such that
\[
A_z ( dC_z \circ j_z(v)) = \eta_z.
\]
On the set $U$, we have that $\xi_{C(z)}$ and $\op{im}(dC_z)$ are distinct complex subspaces which span $T_{C(z)}(\R \times Y)$.  Hence the codomain of $\D$, admits the following splitting:
\[
\overline{\op{Hom}}_{\bj(z,C(z))}(T_z\ds, T_{C(z)}T (\R \times Y)) = \overline{\op{Hom}}_{\bj(z,C(z))}(T_z\ds, \xi_{C(z)}) \oplus \overline{\op{End}}_{\bj(z,C(z))}(T_z\ds).
\]
We split $\eta_z$ accordingly:
\[
\eta_z =  \eta_\ds + \eta_\xi,
\]
where
\[
\begin{array}{lcl}
\eta_\xi & =& \eta_{\xi_{C(z)}}, \\
\eta_\ds & = &  \eta_{T_z\ds}. \\
\end{array}
\]
Since $(\bj \circ dC)_z$ is injective, for any given $\nu_z \in T_z\ds$, we can choose $a_z \in  \overline{\op{End}}_{j(z)}(T_z\ds)$ so that
\[
(\bj \circ du \circ a)_z(\nu_z) = \eta_\ds.
\]
Next, we consider the $\xi$-component. Since  $\overline{\op{End}}_{\bj(z,C(z))}(\xi_{C(z)}) = {T_{\bj(z,C(z))} \sj_\ds^{S^1}}$ is complex one dimensional, and for any given $v_z,w_z \in \xi_z$, there is an element $B_z \in \overline{\op{End}}_{\bj(z,C(z))}(\xi_{C(z)})$ sending $v_z$ to $w_z$.  Hence we take $B_z:T_z \ds \to T_{\bj(z,C(z))}\sj_\ds^{S^1}$ sending $(du \circ j)_z(v_z)$ to $\eta_\xi$.  Thus
\[
A_z := (a_z, B_z): (\nu_z,v_z) \mapsto (\eta_\ds, \eta_\xi),
\]
as desired.

We now need to suitably extend $A_z$ to an $A \in T_{\bj(z,C(z))} \sj_\ds^{S^1}$. When $J_{g,n}$ is restricted to $\ds \in \sm_{g,n+1}$, $V \in T_\bj(\sj_{\ds}^{S^1})$ depends on the special marked point $z_0 \in \ds$ and $\fp(\pi_YC(z)) \in \Sigma_g$.  In order to extend $A$ to all of $T_{\bj} \sj_\ds^{S^1}$, we must let it vary with the complex structure $j$ on $\ds$.  The domain of $V(p,q)$ is $\sm_{g,n+1} \times \Sigma_g$.  

Define a smooth cutoff function $\kappa: \Sigma_g \to \R$ which is nonnegative, takes the value one at $\fp(\pi_YC(z))$, and the value zero outside some open neighborhood of $\fp(\pi_YC(z_0))$ which does not contain any critical points of the perfect Morse function used to define $\lambda_\varepsilon$.  Let $\nu: \sm_{g,n+1} \to \R$ be a smooth nonnegative function which is one at $(\ds, j, z_0,z_1,...,z_n)$ and zero outside an appropriately small open neighborhood of this point.  The neighborhoods of the $z_i$ should not intersect each other, and the neighborhood $U'$ of the special marked point $z_0$ should not contain any preimages of $\fp(\pi_YC(z_0))$ besides $z_0$ itself.  Note that these preimages are finite in number, as otherwise they would accumulate, forcing $C$ to be globally constant.  

Choose an arbitrary smooth extension $A'$ of $A_{z_0}$, shrinking neighborhoods as necessary to ensure that 
\[
\langle A(j, z_0, z_1, ... , z_n, q) \circ dC_{z_0} \circ j_{z_0}, \eta_{z_0} \rangle >0 
\]
whenever $q \in \op{supp}(\kappa)$ and $(j,z_0,...,z_n) \in \op{supp}(\nu)$.  We define
\[
A(j,z_0,...,z_n,q) : = \kappa(q) \nu(j, z_0,...,z_n) A'(j, z_0,...,z_n,q). 
\]
{Since $\langle \D_\bj(A)_z, \eta_z \rangle >0$, we obtain a contradiction to the assumption that $\eta \in \op{coker}(\D)$.  Thus $\eta \equiv 0$ and $\D$ is surjective as claimed.} It follows from surjectivity of $D\overline{\partial}_{\bj}(C)$ that $\M^{\bj}(\alpha, \beta)$ naturally admits the structure of a smooth orbifold of dimension given by the Fredholm index by way of a virtual repeat of \cite[Theorem 0]{Wtrans}.

 \end{proof}

\section{Classification of ECH connectors}\label{sec:connectors}
In this section we carry out some index calculations which allow us to classify \emph{connectors} $\conn$, which are defined to be branched and unbranched covers of a union of trivial cylinders \\ $\bigsqcup_i (\gamma_i \times \R)$.   We will also use intersection theory to show, under sufficient genericity assumptions, that certain sequences of holomorphic curves cannot converge to a building which has  a connector at the top most or bottom most level.   Subsequently in \S \ref{handleslides}, we use these classification results to invoke the obstruction bundle gluing theorems \cite{obg1, obg2} and prove that the appearance of ECH handleslides does not impact the homology.

\subsection{Buildings and connectors}

As in \S \ref{sec:freddie}, if $C$ is a $J$-holomorphic curve with positive ends at Reeb orbits $\alpha_1,\ldots,\alpha_k$ and negative ends at Reeb orbits $\beta_1,\ldots,\beta_l$, then the Fredholm index of $C$ is given by the formula
\begin{equation}
\label{eqn:ind}
\op{ind}(C)=-\chi(C) + 2c_\tau(C) + \sum_{i=1}^k\CZ_\tau(\alpha_i) - \sum_{j=1}^l\CZ_\tau(\beta_j).
\end{equation}
Here $\chi(C)$ denotes the Euler characteristic of the domain of $C$, so if $C$ is irreducible of genus $g$ then
\begin{equation}
\label{eqn:chi}
\chi(C) = 2-2g-k-l.
\end{equation}

Next we recall the definition of a pseudoholomorphic building from \cite{BEHWZ}; see also \cite[\S 9.4]{wendl-sft}.  In our setting all the curves and their limits are non-nodal and unmarked.

\begin{definition}\label{building} 
For our purposes, a \emph{holomorphic building} is an $m$-tuple $(u_1,\ldots,u_m)$, for some positive integer $m$, of (possibly disconnected) $J$-holomorphic curves $u_i$ in $\R\times Y$, called {\em levels}. Although our notation does not indicate this, the building also includes, for each $i\in\{1,\ldots,m-1\}$, a bijection between the negative ends of $u_i$ and the positive ends of $u_{i+1}$, such that paired ends are at the same Reeb orbit\footnote{One might also want a holomorphic building to include appropriate gluing data when Reeb orbits are multiply covered, but we will not need this.}. If $m>1$ then we assume that for each $i$, at least one component of $u_i$ is not a trivial cylinder\footnote{ A {\em trivial cylinder\/} is a $J$-holomorphic cylinder $\R\times\gamma$ in $\R\times Y$ where $\gamma$ is a Reeb orbit, which is not required to be embedded.
}. A {\em positive end} of the building $(u_1,\ldots,u_m)$ is a positive end of $u_1$, and a {\em negative end} of $(u_1,\ldots,u_m)$ is a negative end of $u_m$. The {\em genus} of the building $(u_1,\ldots,u_m)$ is the genus of the Riemann surface obtained by gluing together negative ends of the domain of $u_i$ and positive ends of the domain of $u_{i+1}$ by the given bijections (when this glued Riemann surface is connected).
\end{definition}

 We define the \emph{Fredholm index of a holomorphic building} by 
 \[
 \op{ind}(u_1,\ldots,u_m) : = \sum_{i=1}^m\op{ind}(u_i).
\]

We recall the Riemann-Hurwitz theorem, in part to fix notation.

\begin{theorem}[Hartshorne, Corollary IV.2.4]\label{RH}
Let $\varphi:\widetilde{\dot{\Sigma}} \to \ds$ be a compact $k$-fold cover of the punctured Riemann surface $\ds$.  Then
\begin{equation}\label{eqn:RH}
\chi(\widetilde{\dot{\Sigma}}) =  k\chi(\ds) - \sum_{p \in\widetilde{\dot{\Sigma}}} (e(p)-1),
\end{equation}
where $e(p)-1$ is the ramification index of $\varphi$ at $p$.
\end{theorem}
At unbranched points $p$ we have $e(p)-1=0$, thus for any $q \in \ds$, 
\[
\sum_{p \in \varphi^{-1}(q)}e(p)=k.
\]

The Riemann-Hurwitz theorem provides us with the number of punctures of the cover. The multiplicities of the Reeb orbits at the punctures  are determined by the monodromy of the local behavior of a curve near its punctures \cite{MiWh, s1}, which are in turn governed by the monodromy of the covering. 

\subsection{Low index connectors}
  In this section we investigate the relation between low ECH and Fredholm index connectors $\conn$ and the configurations of Reeb orbits at the ends of the components of each $\conn$.    Recall that a \emph{connector} $\conn$ is a branched cover of a union of trivial cylinders, and all or some of the components may be unbranched.

First, we recall that in a symplectization of a contact 3-manifold, all covers of trivial cylinders have non-negative Fredholm index.

\begin{lemma}[Lem.\ 1.7 \cite{obg1}]\label{lem:ht}
Let $C \in \M^J(\alpha,\beta)$ be a branched or unbranched cover of a trivial cylinder $\R \times \gamma$, where $\gamma$ is an embedded Reeb orbit. Then $\op{ind}(C)\ge 0$, with equality only if
\begin{enumerate}[\em (a)]
\item Each component of the domain $\ds$ of $C$ has genus 0.
\item If $\gamma$ is hyperbolic, then the covering $C: \ds \to \R \times \gamma$ has no branch points.
\end{enumerate}
\end{lemma}

The remainder of this section concerns the proof of the following result.

\begin{lemma}[Lem.\ 7.2.1 \cite{farris}]\label{lem:connectortopology} Let $\conn:\dot{\Sigma} \to\R\times Y$ be a connector, where $\conn=\union_iC_i$ and each $C_i$ is connected. The ECH index $I(\conn)=0$ and the genus of each component $C_i$ is zero.{Further assuming that the Fredholm index $\op{ind}(\conn)\in\{0,1\}$}, then:
\begin{enumerate}[\em (i)]
\item If $\ind(\conn)=0$ then $\ind(C_i)=0$ for all $i$, and either
	\begin{enumerate}[\em a.]
	\item $C_i$ is an unbranched cover of a trivial cylinder.
	\item $C_i$ is branched, covers $\R\times e_+$, and has a single positive end.
	\item $C_i$ is branched, covers $\R\times e_-$, and has a single negative end.
	\end{enumerate}
\item If $\ind(\conn)=1$ then $\conn=C_0\cup\union_iC_i$ where $\ind(C_0)=1, \ind(C_i)=0$, and $C_0$ is a branched cover of $\R\times h_j$ for some $j\in\{1,\dots,2g\}$ with either one positive end and two negative ends, or two positive ends and one negative end.  {Each of the $C_i$ is an unbranched cover of a trivial cylinder.}
\end{enumerate}
\end{lemma}
\begin{proof}

Let $m_\pm,m_j$ denote the multiplicities of the ends of $\conn$ at the orbits $e_\pm,h_j$: the multiplicities at the positive and negative ends will be the same because $\conn$ covers a union of trivial cylinders. In particular, the difference between the total multiplicities at the positive and negative ends of $\conn$ will be zero. Therefore, from the ECH index formula (\ref{eqn:indexformula}), we have
\[
I(\conn)=\chi(\Sigma)\cdot0-0^2e+2\cdot0\cdot\left(m_++\sum_jm_j+m_-\right)+m_+-m_--m_++m_-=0.
\]
Note that in particular $c_\tau(\conn)=0$. These formulas also hold for each component $C_i$ of $\conn$.

Let $p_\pm(C_i)$ denote the number of positive and negative ends of $C_i$, respectively, and let $g(C_i)$ denote the genus of $C_i$. Recall that the Euler characteristic of a surface with $p$ punctures is $2-2g-p$.

\textbf{Case  (i)} If $\ind(\conn)=0$, then $\ind(C_i)=0$ for all $i$ by Lemma \ref{lem:ht}.

\textbf{Case (i.a)} Assume $u(C_i)$ is a branched cover of $\R\times h_j$. Because $c_\tau(C_i)=0$ and hyperbolic orbits have Conley-Zehnder index zero, we have
\[
0=\ind(C_i)=-\chi(C_i).
\]
The Euler characteristic of a cylinder is 0, therefore the Riemann-Hurwitz Theorem (\ref{eqn:RH}) gives us
\[
0=\chi(C_i)=-\sum_{p\in\dot\Sigma}(e(p)-1).
\]
Because $e(p)\geq1$ for all $p$, each term $e(p)-1\geq0$, so we must have $e(p)=0$ for all $p$. Therefore,  $C_i$ is unbranched. Moreover,
\[
0=2g(C_i)-2+p_+(C_i)+p_-(C_i)\Leftrightarrow 2=2g(C_i)+p_+(C_i)+p_-(C_i).
\]
Because $C_i$ is a cover of a cylinder, $p_\pm(C_i)\geq1$. Therefore $g(C_i)=0$, both $p_\pm(C_i)=1$, and $C_i$ unbranched cover of a cylinder.

\textbf{Case  (i.b)} Because the Conley-Zehnder index of a cover of $e_+$ is always 1, we have
\begin{align}
0&=\ind(C_i)\nonumber
\\&=2g(C_i)-2+p_+(C_i)+p_-(C_i)+p_+(C_i)-p_-(C_i),\label{eqn:coveringe+}
\end{align}
hence
\[
1=g(C_i)+p_+(C_i).
\]
Therefore, because $p_+(C_i)\geq1$, we have $g(C_i)=0$ and $p_+(C_i)=1$.

\textbf{Case (i.c)} By the same argument as for i.(b), using the fact that the Conley-Zehnder index of a cover of $e_-$ is always $-1$, we get $g(C_i)=0$ and $p_-(C_i)=1$.

\textbf{Case (ii)}   {If $\ind(\conn)=1$ then $\ind(C_i)\leq1$ for all $i$. Because $\ind(C_i)\geq0$ for all $i$ by Lemma \ref{lem:ht}, there must be one component $C_0$ with $\ind(C_0)=1$ and all other $C_i$ have $\ind(C_i)=0$.}

If $C_0$ were a branched cover of $\R\times e_+$, then setting the analogue of the right hand side of (\ref{eqn:coveringe+}) equal to $\ind(C_0)$ would imply that
\[
1=2(g(C_0)-1+p_+(C_0)),
\]
a contradiction. Similarly $C_0$ being a branched cover of $\R\times e_-$ would lead to a contradiction.

Therefore $C_0$ must be a branched cover of $\R\times h_j$. In this case, because hyperbolic orbits have Conley-Zehnder index zero, we have
\[
1=2g(C_0)-2+p_+(C_0)+p_-(C_0)\Leftrightarrow 3=2g(C_0)+p_+(C_0)+p_-(C_0).
\]
Because $p_\pm(C_0)\geq1$, this implies $1\geq2g(C_0)$, requiring $g(C_0)=0$. Therefore either $(p_+(C_0),p_-(C_0))=(1,2)$ or $(p_+(C_0),p_-(C_0))=(2,1)$.
\end{proof}

\subsection{Classification of connectors arising in buildings}

In this section we use intersection theory and higher asymptotics of holomorphic curves to rule out connectors from appearing at the top-most and bottom-most level of a building arising as a limit of a (sub)sequence of holomorphic curves defined in terms of a one parameter family of domain dependent almost complex structures, cf. Proposition \ref{notopbottomconn}. This result will be key in \S \ref{handleslides}.  We begin by recalling some needed results about the asymptotics of holomorphic curves from \cite[\S 3.1]{dc}.

Let $\gamma$ be an embedded Reeb orbit, and let $N$ be a tubular neighborhood of $\gamma$. We can identify $N$ with a disk bundle in the normal bundle to $\gamma$, and also with $\xi|_\gamma$. Let $\zeta$ be a braid in $N$, i.e.\ a link in $N$ such that that the tubular neighborhood projection restricts to a submersion $\zeta\to\gamma$. Given a trivialization $\tau$ of $\xi|_\gamma$, one can then define the {\em writhe\/} $w_\tau(\zeta)\in\Z$. To define this one uses the trivialization $\tau$ to identify $N$ with $S^1\times D^2$, then projects $\zeta$ to an annulus and counts crossings of the projection with (nonstandard) signs; see \S \ref{s:writhe}, \cite[\S2.6]{Hrevisit}, or \cite[\S3.3]{Hu2} for details. 

Now let $C$ be a $J$-holomorphic curve in $\R\times Y$. Suppose that $C$ has a positive end at $\gamma^d$ which is not part of a multiply covered component. Results of Siefring \cite[Cor.\ 2.5 and 2.6]{s1} show that if $s$ is sufficiently large, then the intersection of this end of $C$ with $\{s\}\times N\subset\{s\}\times Y$ is a braid $\zeta$, whose isotopy class is independent of $s$. We will need bounds on the writhe $w_\tau(\zeta)$, which are provided by the following lemma.

\begin{lemma}[Lemma 3.2 \cite{dc}]\label{lem:positivewrithe}
Let $\gamma$ be an embedded Reeb orbit, let $C$ be a $J$-holomorphic curve in $\R\times Y$ with a positive end at $\gamma^d$ which is not part of a trivial cylinder or a multiply covered component, and let $\zeta$ denote the intersection of this end with $\{s\}\times Y$. If $s \gg 0$, then the following hold:
\begin{description}
\item{\em (a)} $\zeta$ is the graph in $N$ of a nonvanishing section of $\xi|_{\gamma^d}$. Thus, using the trivialization $\tau$ to write this section as a map $\gamma^d\to\C\setminus\{0\}$, it has a well-defined winding number around $0$, which we denote by $\op{wind}_\tau(\zeta)$.
\item{\em (b)} $\op{wind}_\tau(\zeta) \le \floor{\op{CZ}_\tau(\gamma^d)/2}$.
\item{\em (c)} If $J$ is generic, $\op{CZ}_\tau(\gamma^d)$ is odd, and $\op{ind}(u)\le 2$, then equality holds in (b).
\item{\em (d)} $w_\tau(\zeta) \le (d-1)\op{wind}_\tau(\zeta)$.
\end{description}
\end{lemma}

Symmetrically to Lemma~\ref{lem:positivewrithe}, we also have the following: 

\begin{lemma}
\label{lem:negativewrithe}
Let $\gamma$ be an embedded Reeb orbit, let $C$ be a $J$-holomorphic curve in $\R\times Y$ with a negative end at $\gamma^d$ which is not part of a trivial cylinder or multiply covered component, and let $\zeta$ denote the intersection of this end with $\{s\}\times Y$. If $s\ll0$, then the following hold:
\begin{description}
\item{\em (a)} $\zeta$ is the graph of a nonvanishing section of $\xi|_{\gamma^d}$, and thus has a well-defined winding number $\op{wind}_\tau(\zeta)$.
\item{\em (b)} $\op{wind}_\tau(\zeta) \ge \ceil{\op{CZ}_\tau(\gamma^d)/2}$.
\item{\em (c)} If $J$ is generic, $\op{CZ}_\tau(\gamma^d)$ is odd, and $\op{ind}(u)\le 2$, then equality holds in (b).
\item{\em (d)} $w_\tau(\zeta) \ge (d-1)\op{wind}_\tau(\zeta)$.
\end{description}
\end{lemma}

\begin{remark}
\label{rem:improved}
Lemma~\ref{lem:positivewrithe}(b),(d) imply that
\[
w_\tau(\zeta) \le (d-1)\floor{\CZ_\tau(\gamma^d)/2}.
\]
{In fact one can improve this to}
\begin{equation}
\label{eqn:improved}
w_\tau(\zeta) \le (d-1)\floor{\CZ_\tau(\gamma^d)/2} - \op{gcd}\left(d,\floor{\CZ_\tau(\gamma^d)/2}\right)+1,
\end{equation}
see \cite{s2}.   Recent work of Cristofaro-Gardiner - Hutchings - Zhang obtains equality in \eqref{eqn:improved} in the following situation.
\end{remark}

\begin{lemma}[{\cite[Cor. 5.3]{CGHZ}}]\label{lem:writheCZequality} Let $\gamma$ be an embedded Reeb orbit, let $C$ be a $J$-holomorphic curve in $\R\times Y$ with only one positive end at $\gamma^d$, and let $\zeta$ denote the intersection of this end with $\{s\}\times Y$ for $s \gg 0$. Suppose $CZ_\tau(\gamma^d)$ is odd, the Fredholm index of $u$ is at most 2, and $J$ is generic. Then  $\zeta$ is isotopic to the braid given by a regular end and
\[
w_\tau(\zeta)=(d-1)\floor{\op{CZ}(\gamma^d)/2}-\gcd\left(d,\floor{\op{CZ}(\gamma^d)/2}\right)+1.
\]
\end{lemma}

The definition of a regular end is lengthy, see \cite[Def. 1.3]{CGHZ}.  It ensures that the topology of the braid near an embedded Reeb orbit is completely determined by the total multiplicity of the orbit and the corresponding partition numbers.  However, \cite[Thm. 1.4]{CGHZ}, guarantees that for generic $J$, every generic curve has regular positive and negative ends.  Symmetrically to Lemma \ref{lem:writheCZequality} we have the following result for a negative end.

\begin{lemma}\label{lem:writheCZequalityneg} 
Let $\gamma$ be an embedded Reeb orbit, let $C$ be a $J$-holomorphic curve in $\R\times Y$ with only one negative end at $\gamma^d$, and let $\zeta$ denote the intersection of this end with $\{s\}\times Y$ for $s \ll 0$. Suppose $CZ_\tau(\gamma^d)$ is odd, the index of $u$ is at most 2, and $J$ is generic. Then  $\zeta$ is isotopic to the braid given by a regular end and
\[
w_\tau(\zeta)=(d-1)\ceil{\op{CZ}(\gamma^d)/2}+\gcd\left(d,\ceil{\op{CZ}(\gamma^d)/2}\right)-1.
\]
\end{lemma}

The proof of the main classification result, Proposition \ref{notopbottomconn}, requires the following direct computation of asymptotic writhes and linking numbers, which uses the preceding lemmas.
\begin{lemma}\label{lem:writhecomputations} Let $J$ be generic.  Let $\zeta_i,\zeta_j$ be connected braids about an embedded Reeb orbit $\gamma$ with multiplicities $d_i, d_j$. If both $\zeta_i, \zeta_j$ arise from either the positive or the negative ends of a curve which covers $\gamma$, then
\begin{enumerate}[\em (i)]
\item Assuming $\gamma=e_+$:
	\begin{enumerate}[\em a.]
	\item There is only one positive end $\zeta_+$, and $w_\tau(\zeta_+)=1-d_+$.
	\item If the $\zeta_i,\zeta_j$ are negative ends, then $w_\tau(\zeta_i)=d_i-1$, $w_\tau(\zeta_j)=d_j-1$, and 
	\[\ell_\tau(\zeta_i,\zeta_j)=\min(d_i,d_j).\]
	\end{enumerate}
\item Assuming $\gamma=e_-$:
	\begin{enumerate} [\em a.]
	\item If the $\zeta_i,\zeta_j$ are positive ends, then $w_\tau(\zeta_i)=1-d_i$, $w_\tau(\zeta_j)=1-d_j$, and \[\ell_\tau(\zeta_i,\zeta_j)=-\min(d_i,d_j).\]
	\item There is only one negative end $\zeta_-$, and $w_\tau(\zeta_-)=d_--1$.
	\end{enumerate}
\end{enumerate}
\end{lemma}

\begin{proof} We proceed casewise.

\textbf{Case (i.a)} By Lemma \ref{lem:connectortopology} {(i.b)}, the end $\zeta_+$ is the only positive end. Therefore Lemma \ref{lem:writheCZequality} applies, giving us
\begin{align*}
w_\tau(\zeta_+)&=(d_+-1)\floor{\op{CZ}(\gamma^{d_+})/2}-\gcd\left(d_+,\floor{\op{CZ}(\gamma^{d_+})/2}\right)+1
\\&=(d_+-1)\floor{\frac{1}{2}}-\gcd\left(d_+,\floor{\frac{1}{2}}\right)+1
\\&=0-\gcd(d_+,0)+1
\\&=1-d_+.
\end{align*}

\textbf{Case (i.b)} Firstly, we immediately have $\op{wind}_\tau(\zeta_i)=1$ by Lemma \ref{lem:negativewrithe} (b,c):
\[
\op{wind}_\tau(\zeta_i)=\ceil{CZ_\tau(\zeta_i^{d_i})/2}=\ceil{\frac{1}{2}}=1.
\]

Therefore $\gcd(d_i,\op{wind}_\tau(\zeta_i))=\gcd(d_i,1)=1$, which is a sub-case in the proof of \cite[Lemma 6.7]{Hindex}.  There the equality
\[
w_\tau(\zeta_i)=(d_i-1)\op{wind}_\tau(\zeta_i)
\]
is proven by showing that the $\zeta_i$ are isotopic to $(d_i,1)$ torus braids when $\gcd(d_i,\op{wind}_\tau(\zeta_i))=1$. Therefore $w_\tau(\zeta_i)=d_i-1$.

For the claim on linking, let $\lambda_i$ denote the smallest eigenvalue of the asymptotic operator $L_{d_i}$ associated to $\gamma^{d_i}$ in the expansion of $\zeta_i$. The proof of \cite[Lemma 6.9]{Hindex} proceeds without loss of generality by considering three cases: when $\lambda_i<\lambda_j$, when $\lambda_i=\lambda_j$ and the coefficients of the corresponding eigenfunctions are different, and when $\lambda_i=\lambda_j$ and the coefficients of the corresponding eigenfunctions are the same.  We are guaranteed by \cite[Proposition 3.9]{obg2} that we are in either of the first two cases, while the proof of \cite[Lemma 6.9]{Hindex} gives the equality
\[
\ell_\tau(\zeta_i,\zeta_j)=\min\{d_i,d_j\}
\]
in both of those cases, which is stronger than its general result.

\textbf{Case (ii.a)} We immediately have $\op{wind}_\tau(\zeta_i)=-1$ by Lemma \ref{lem:positivewrithe}(b,c):
\[
\op{wind}_\tau(\zeta_i)=\floor{CZ_\tau(\zeta_i^{d_i})/2}=\floor{-\frac{1}{2}}=-1.
\]

The proof that $w_\tau(\zeta_i)=(d_i-1)\op{wind}_\tau(\zeta_i)$ and hence that $w_\tau(\zeta_i)=1-d_i$ is a virtual repeat of the proof for negative ends from \cite[Lemma 6.7]{Hindex} as in Case (i.b).

For the claim on linking, we can repeat the proof in Case (i.b). Note that \cite[Lemma 6.9]{Hindex} only applies to negative ends, but the proof will work using the asymptotic expansion of a positive end from \cite{hwz1}, written in our notation as \cite[Lemma 5.2]{Hu2}. If $\lambda_i<\lambda_j$, or $\lambda_i=\lambda_j$ with corresponding eigenfunctions having different multiplicities in the $\zeta_i$, we know that the braid $\zeta_j$ must be nested inside $\zeta_i$, therefore
\[
\ell_\tau(\zeta_i,\zeta_j)=\op{wind}_\tau(\zeta_i)d_j=-d_j.
\]
We have
\[
-d_j=-\min\{d_i,d_j\}
\]
because, by pulling back both $\zeta_i$ to covers of $\gamma^{d_id_j}$, we multiply their winding numbers by $d_j$ and $d_i$, respectively, and can apply the analytic perturbation theory of \cite[\S3]{hwz2}, written in our notation as \cite[Lemma 2.11 (a)]{obg1}, to obtain
\[
d_j\op{wind}_\tau(\zeta_i)\geq d_1\op{wind}_\tau(\zeta_j)\Leftrightarrow d_j\leq d_i.
\]

\textbf{Case (ii.b)} By Lemma \ref{lem:connectortopology}(i.c), the end $\zeta_-$ is the only negative end. Therefore Lemma \ref{lem:writheCZequalityneg} applies, giving us
\begin{align*}
w_\tau(\zeta_-)&=(d_--1)\ceil{\op{CZ}(\gamma^{d_-})/2}+\gcd\left(d_1,\ceil{\op{CZ}(\gamma^{d_-})/2}\right)-1
\\&=(d_--1)\ceil{-\frac{1}{2}}+\gcd\left(d_-,\ceil{-\frac{1}{2}}\right)-1
\\&=0+\gcd(d_-,0)-1
\\&=d_--1.
\end{align*}
\end{proof}

Finally, we need the following inequality from intersection theory of holomorphic curves, cf. \cite[\S 3.2]{dc}, which is proven similarly to the relative adjunction formula, Lemma \ref{lem:adjunction}.  As before, let $\gamma$ be an embedded Reeb orbit with tubular neighborhood $N$, and let $\tau$ be a trivialization of $\xi|_\gamma$.

\begin{lemma}\label{lem:adj}
Let $C$ be a $J$-holomorphic curve in $[s_-,s_+] \times N$ with no multiply covered components and with boundary $\zeta_+ - \zeta_-$ where $\zeta_\pm$ is a braid in $\{ s_\pm \} \times N$.  Then
\[
\chi(C) + w_\tau(\zeta_+) - w_\tau(\zeta_-) = 2\delta(C) \geq 0,
\]
where $\chi(C)$ denotes the Euler characteristic of the domain of $C$ and $\delta(C)$ is a count of the singularities of $C$ in $Y$ with positive integer weights.
\end{lemma}

With these preliminaries in place, we are now ready to prove the key classification result which excludes connectors from appearing in the top most or bottom most level of a building arising as a limit in the sense of \cite{BEHWZ}.
\begin{proposition}\label{notopbottomconn}
Let  $\{\cj_t\}_{t\in [0,1]}$ be a generic family of domain dependent almost complex structures and $\alpha$ and $\beta$ be admissible Reeb currents with $I(\alpha,\beta)=1$.  Let $C(t) \in \M^{\cj_t}(\alpha,\beta)$ be a sequence of Fredholm index 1 curves, which, as $t \to 1$, converges in the sense of \cite{BEHWZ}  to a building  with $n$ levels given by $C_{i} \in \M^{\cj_1}(\gamma_{i-1},\gamma_{i})$, $i=1,...,n$, where $\gamma_0 = \alpha$ and $\gamma_n = \beta$.  Then neither the top most level $C_1$ nor the bottom most level $C_n$ are connectors.
\end{proposition}
\begin{proof}

We assume that the proposition is false and set up some notation.  Suppose to get a contradiction that there exists a sequence of $\{\cj_t\}$-holomorphic curves $\{ C(t)\} \in \M^{\cj_t}(\alpha,\beta)$ which converges in the sense of \cite{BEHWZ} to a $n$-level building which has either $C_1$ or $C_n$ as a connector.   {Recall that $C_i$ is an equivalence class of holomorphic curves in $\R \times Y$, where two holomorphic curves are equivalent iff they differ by $\R$-translation in $\R \times Y$.}  In the following, we will choose a representative of this equivalence class and still denote it by $C_i$. If necessary, translate the holomorphic curve $C_1$ upward and $C_n$ downward so that Lemmas \ref{lem:positivewrithe}-\ref{lem:writheCZequalityneg} apply, cf. \cite[\S 3.3]{dc}.

Without loss of generality, we can work under the assumption that the connector appears in the top most level,  $C_1 \in \M^{\cj_1}(\alpha, \alpha)$.  Consider an embedded Reeb orbit $\gamma$ appearing in the orbit set $\alpha$.  Let $N_\gamma$ be a tubular neighborhood of the Reeb orbit $\gamma$.  For some sufficiently large  $s_0 \gg 0$ and some $t$ close to 1, the intersection $C(t) \cap ( [s_0, \infty) \times N_\gamma)$  can be identified with the union of components of $C_1$ that cover $\R \times \gamma$.  Denote both by $C$.  Note that as a subset of $C(t)$, $C$ is not a trivial cylinder, but rather an embedding in the complement of a finite number of singular points.  

While intersection positivity is not true in general for domain dependent almost complex structures, by Remark \ref{ddacsintpos} if
\[
C: (\ds, j) \to (\R \times Y, \bj)
\]
is a $(j,\bj)$-holomorphic curve, then $C \vert_{C^{-1}(\R \times N)}$ is $(j,\fj_0)$-holomorphic.  Since $\fj_0$ is domain independent, the subset $C(\ds) \cap \R \times N \subset C(\ds)$ satisfies intersection positivity.  Thus we will be in a situation to apply relative adjunction as in Lemma \ref{lem:adj} because intersection positivity holds.  Moreover, the count of singularities of $C$, satisfies $\delta(C) \geq 0$ with equality if and only if $C$ is embedded.  

We will show that if $C$ arises from a nontrivial connector appearing at the top most level, then relative adjunction as in Lemma \ref{lem:adj} will imply that $\delta(C) < 0$, a contradiction.  Note that a connector cannot be trivial in the sense that it exclusively consists of unbranched components, e.g. trivial cylinders, as explained in \cite[Remark 9.26]{wendl-sft}.

There are three cases to consider, corresponding to connectors containing components satisfying the conclusions of Lemma \ref{lem:connectortopology}, (i.b), (i.c), and (ii).

\textbf{Case (i.b)} Assume $C$ is a component of a connector covering $\R\times e_+$. Then by Lemma \ref{lem:connectortopology} we have $g(C)=0$ and by Lemma \ref{lem:connectortopology}(i.b) $C$ has a single positive end. Let $d_+$ denote the covering multiplicity of this end, and let $d_i$ denote the covering multiplicity of the $i^\text{th}$ negative end of $C$. Because $C$ covers a trivial cylinder, $d_+=\sum_{i=1}^{p_-(C)}d_i$. For $t$ sufficiently close to 1, there is a representative $C$ with the following properties.

\begin{enumerate}
\item $C^{-1}([0,\infty) \times Y)$ is an annulus with one puncture, which is mapped by $C$ to $[0,\infty) \times N_{e_+}$
\item $C^{-1}((-\infty, 0] \times Y)$ consists of as many half cylinders $C_i$ as there are $p_-(C)$ negative ends of $C$.  
\item $C(C_i)$ is contained in $(-\infty, 0] \times N_{e_+}$ and $C(C_i) \cap (\{ 0 \} \times N)$ is a braid $\zeta_i$ which projects to $e_+$ with degree $d_i$ and has distance at most $\frac{\varepsilon}{p_-(C) +1}$ from $e_+$.
\end{enumerate}

\noindent Also let $\zeta_+$ denote the braid corresponding to the positive end of $C$ at $e_+^{d_+}$.   It follows that the union $\bigcup_i \zeta_i$ is a braid.  We obtain a contradiction:
\begin{align*}
2\delta(C)&=2-p_+(C)-p_-(C)+w_\tau(\zeta_+) - {w_\tau\left(\bigcup_i \zeta_i\right)}
\\&=1-p_-(C)+(1-d_+)-\left(\sum_{i=1}^{p_-(C)}(d_i-1)+\sum_{i\neq j}\min(d_i,d_j)\right)\text{ by Lemma \ref{lem:writhecomputations}(i) and (\ref{eqn:writhelinking})}
\\&=2-2d_+-\sum_{i\neq j}\min(d_i,d_j)
\\&\leq -2.
\end{align*}
(Note that our notation $\sum_{i\neq j}\min(d_i,d_j)$ accounts for the factor of two in (\ref{eqn:writhelinking}).) In the inequality we have used the fact that $\sum_{i\neq j}\min(d_i,d_j)\geq2$ whenever there are at least two negative ends, because $d_i\geq1$. There can never be just one negative end lest $C$ be topologically a cylinder and therefore unbranched, by the Riemann-Hurwitz Theorem.

\textbf{Case (i.c)} Assume $C$ is a component of a connector covering $\R\times e_-$. Then by Lemma \ref{lem:connectortopology} we have $g(C)=0$ and by Lemma \ref{lem:connectortopology}(i.c) $C$ has a single negative end. Let $d_-$ denote the covering multiplicity of this end, and let $d_i$ denote the covering multiplicity of the $i^\text{th}$ positive end of $C$. For $t$ sufficiently close to 1, there is a representative $C$ with the following properties.

\begin{enumerate}
\item $C^{-1}((-\infty, 0] \times Y)$ is an annulus with one puncture, which is mapped by $C$ to $(\infty,0] \times N_{e_-}$
\item $C^{-1}([0,\infty) \times Y)$ consists of as many half cylinders $C_i$ as there are $p_+(C)$ positive ends of $C$.  
\item $C(C_i)$ is contained in $[0,\infty) \times N_{e_-}$ and $C(C_i) \cap (\{ 0 \} \times N)$ is a braid $\zeta_i$ which projects to $e_-$ with degree $d_i$ and has distance at most $\frac{\varepsilon}{p_+(C) +1}$ from $e_-$.
\end{enumerate}

\noindent Also let $\zeta_-$ denote the braid corresponding to the negative end of $C$ at $e_-^{d_-}$.  It follows that the union $\bigcup_i \zeta_i$ is a braid. We obtain a contradiction:
\begin{align*}
2\delta(C)&=2-p_+(C)-p_-(C) + {w_\tau\left(\bigcup_i \zeta_i\right)} - w_\tau(\zeta_-) 
\\&=1-p_+(C)+\left(\sum_{i=1}^{p_+(C)}(1-d_i)-\sum_{i\neq j}\min(d_i,d_j)\right)-(d_--1)\text{ by Lemma \ref{lem:writhecomputations}(ii) and (\ref{eqn:writhelinking})}
\\&=2-2d_--\sum_{i\neq j}\min(d_i,d_j)
\\&\leq-2.
\end{align*}
As in Case (i.b), we must have $p_+(C)\geq2$, hence $\sum_{i\neq j}\min(d_i,d_j)\geq2$.

\textbf{Case (ii)} If a branched component of the connector at the top (respectively, the bottom) covers $\R\times h$, where $h$ is hyperbolic, then by Lemma \ref{lem:connectortopology}(ii), its ends must be asymptotic to $h^2$. Therefore $\alpha$ (respectively, $\beta$) must include the pair $(h,m)$ with $m\geq2$, which contradicts the fact that $\alpha$ (respectively, $\beta$) is an ECH chain complex generator.
\end{proof}

\section{From domain dependent $\bj$ to domain independent $J$}\label{handleslides}
In Corollary \ref{nogenus}, we saw that for a generic $S^1$-invariant domain dependent almost complex structure $\bj \in  \sj_\ds^{S^1}$ that ECH index one moduli spaces of nonzero genus curves are empty.  However ECH is defined using a domain independent generic $\lambda$-compatible $J$, so we must prove the analogous result when $J$ is a generic $\lambda$-compatible almost complex structure.   In order to do so, we consider a generic one parameter family $\{ \cj_t\}_{t\in[0,1]}$ of domain dependent almost complex structures interpolating between a generic $\cj_0 :=\bj \in \sj_\ds^{S^1}$ and a domain independent generic $\lambda$-compatible $\cj_1:=J \in \sj_{reg}(Y,\lambda)$ and show that the computation of ECH is not affected.

\subsection{Overview and sketch of proof}
Our main result is the following.

\begin{proposition}\label{ECH-ok}
Let $\alpha$ and $\beta$ be admissible Reeb currents with $I(\alpha,\beta)=1$  and $\op{deg}(\alpha,\beta) >0.$  For generic paths $\{ J_t\}_{t\in[0,1]}$ connecting $\cj_0 := \bj \in \sj_\ds^{S^1}$ and $\cj_1 := J \in \mathscr{J}_{reg}(Y,\lambda)$, the moduli space $\mathcal{M}_t := \mathcal{M}^{\cj_t}(\alpha, \beta)$  is cut out transversely save for a discrete number of times $t_0,...,t_\ell \in (0,1)$.  At each such $t_i$,  the ECH differential can change either by:
 \begin{enumerate}[\em(a)]
 \item The creation or destruction of a pair of oppositely signed  curves.\footnote{Because we are using $\Z_2$-coefficients, we will not sort through the signs.}
 \item An ``ECH handleslide.''
\end{enumerate}
However, in either case, the homology is unaffected. 
 \end{proposition}

 For each $\cj_t$ we consider the moduli space $\M_t(\alpha,\beta)$ of $\cj_t$-holomorphic curves where $\alpha$ and $\beta$ are admissible Reeb current satisfying $I(\alpha,\beta)=1$ and $\op{deg}(\alpha,\beta) >0$.  That  $\op{deg}(\alpha,\beta) >0$ rules out moduli spaces of $\cj_t$-holomorphic cylinders, for which the domain dependent almost complex structures cannot be used, cf. Lemma \ref{lem:d0g0}.  We have that $\M_t(\alpha,\beta)$ is cut out transversely save a discrete number of times $t_i\in[0,1]$ and at such a nonregular $\cj_{t_i}$, the differential can be impacted by either the creation or destruction of a pair of oppositely signed curves or by an ``ECH handleslide."  In the former case, the signed and mod 2 counts of curves in $\M_{t_i-\varepsilon}$ and $\M_{t_i+\varepsilon}$ are the same. 
 
  The differential can change at an ``ECH handleslide," at which a sequence of Fredholm and ECH index 1 curves $\{C(t)\}$ breaks into a holomorphic building in the sense of \cite{BEHWZ} into components consisting of an ECH and Fredholm index 0 curve, an ECH and Fredholm index 1 curve, and some ``connectors," which are Fredholm index 0 branched covers of a trivial cylinder $ \R \times \gamma$.    In \S \ref{sec:connectors} we previously demonstrated that connectors cannot appear at the top most or bottom most level of the building via intersection theory arguments similar to \cite[\S 4]{dc}.  As a result, in \S \ref{sec:ECHok} we can appeal to the obstruction bundle gluing theorems of \cite{obg1, obg2} in conjunction with an inductive argument involving the degree of a completed curve, cf. Definition \ref{degree} and Proposition \ref{prop-degree}, to show that the the presence of an ECH handleslide does not have an impact after passing to homology.  
  
Proposition \ref{ECH-ok} yields the the domain independent analogue of Corollary \ref{nogenus}:

\begin{corollary}\label{nogenusJ}
Let $\alpha$ and $\beta$ be nondegenerate admissible Reeb currents and $J \in \sj(Y,\lambda)$ be generic.  If $\op{deg}(\alpha,\beta) >0$ and $I(\alpha,\beta)=1$ then the mod 2 count $\#_{\Z_2}\M^{J}(\alpha,\beta) = 0$.  If $\alpha$ and $\beta$ are associated to $\lambda_\varepsilon$ as in Lemma \ref{lem:efromL} and $\mathcal{A}(\alpha), \mathcal{A}(\beta) <L(\varepsilon)$ then $\langle \partial^{L(\varepsilon)} \alpha, \beta \rangle =0$.\end{corollary}

In \S \ref{sec:handles} we demonstrate that the classification results for connectors in \S \ref{sec:connectors} ensure that at an ECH handleslide $t_i$,  a sequence of $\cj_{t}$-holomorphic curves $\{ C_k \ | \ 
\mbox{ind}(C_k) =1 \}$ breaks into an \emph{ECH handleslide building} $(C_+, \conn, C_-)$ wherein:
\begin{enumerate}[(i)]
 \item The top most curve $C_+$ has either index 1 or index 0;
\item Connectors $\conn$ with  $\mbox{ind}(\conn)=0$ appearing in the middle; 
\item The bottom most curve $C_-$ has $\ind(C_-) = 1-\mbox{ind}(C_+)$.
\end{enumerate}
Moreover, the index 0 curve occurring at either the top most or bottom most level cannot contain any connectors.
\begin{definition}
We define an \emph{ECH handleslide} to be the index 0 curve which is not a connector in an \emph{ECH handleslide building} $(C_+, \conn, C_-)$, in analogy with Morse theory.  
\end{definition}

As observed in \cite[\S 7.1.1]{farris}, because connectors, the branched covers of trivial cylinders, cannot appear as the top-most or bottom-most level by Proposition \ref{notopbottomconn}, we can appeal to the obstruction bundle gluing theorems \cite{obg1,obg2} to relate the curve counts occurring immediately prior to and following the appearance of an ECH handleslide at time $t_i$.    If we assume that the ECH handleslide is $C_-$, then as explained in Remark \ref{rem:handleslide} we obtain:
\begin{equation}\label{eq:handle}
{\# \mathcal{M}_{t_i + \epsilon}(\alpha,\beta)= \# \mathcal{M}_{t_i - \epsilon}(\alpha, \beta) + \# G(C_+, C_-) \cdot \# \mathcal{M}_{t_i }(\alpha, \gamma_+),}
\end{equation}
where $\gamma$ is another (admissible) Reeb current such that $I(\alpha,\gamma_+)=1$ with $C_+ \in \mathcal{M}_{t_i }(\alpha, \gamma_+)$. Note that the connector $\conn \in \mathcal{M}_{t_i }(\gamma_+, \gamma_-)$ and the ECH handleslide curve $C_- \in \mathcal{M}_{t_i }(\gamma_-, \beta)$.  As explained in \S \ref{sec:obg}, obstruction bundle gluing gives a combinatorial formula for $\#G(C_+, C_-) \in \Z$, based on the negative asymptotic ends of $C_+$, the positive asymptotic ends of $C_-$, and the partitions associated to the ends of the connectors $\conn$.  For each embedded Reeb orbit $\gamma$, the total covering multiplicity of Reeb orbits covering $\gamma$ in the list $\gamma_+$ is the same as the total for $\gamma_-$.  (In contrast, for the usual form of Floer theory gluing, one would assume that $\gamma_+=\gamma_-$.) In \S \ref{sec:ECHok} we complete the proof of Proposition \ref{ECH-ok} by way of an inductive argument involving the degree, which precludes the need to explicitly compute $\#G(C_+, C_-) $ as we obtain $\# \mathcal{M}_{t_i}(\alpha,\gamma_+) =0$ for all admissible $\gamma_+$ such that $I(\alpha,\gamma_+)=1$.

\subsection{Handleslides and bifurcations}\label{sec:handles}

 An \emph{ECH handleslide building} is a building arising as a limit of $I(C)=1$, $\op{ind}(C)=1$ curves in $\R\times Y$ as the complex structure varies through domain-dependent almost complex structures. The terminology arises from the fact that such a building might include levels with $I(C)=0$ which do not consist solely of trivial cylinders, in analogy to the Morse index zero gradient trajectories which arise during a handleslide in a generic homotopy of Morse functions.  Note that the characterization from Proposition \ref{lowiprop} does not apply to moduli spaces defined using domain-dependent almost complex structures.

\begin{lemma}[Configuration of an ECH handleslide]\label{lem:slide}
Fix admissible Reeb currents $\alpha$ and $\beta$ with $\op{deg}(\alpha,\beta)>0$ and $I(\alpha,\beta)=1$.  Let $\{\cj_t\}_{t\in [0,1]}$ be a one generic parameter family of almost complex structures.  Consider the corresponding moduli spaces $\mathcal{M}_t := \mathcal{M}^{\cj_t}(\alpha, \beta)$; label the times at which $\mathcal{M}_t$ is not cut out transversely by $t_0,...,t_\ell \in (0,1)$.  Let $C(t) \in \M_t(\alpha,\beta)$ with $t \to t_i$.  Then after passing to a subsequence, $\{ C(t) \}$ converges in the sense of \cite{BEHWZ} either to a curve in $\M_{t_i}(\alpha,\beta)$ or to
an \emph{ECH handleslide building} with
\begin{enumerate}[\em (i)]
        \setlength{\itemsep}{0pt}
 \item An index 1 curve  at the top most level $C_+$ (or at bottom most level $C_-$);
\item Connectors $\conn$ with  $\mbox{ind}(\conn)=0$; 
\item An index 0 curve, the ECH handleslide, at the bottom most level $C_-$ (or at the top most level $C_+$).
\end{enumerate}
\end{lemma}

\begin{proof}
By the compactness theorem in \cite{BEHWZ}, any sequence of ECH and Fredholm index 1 curves in $\M^{\cj_t}(\alpha,\beta)$ has a subsequence which converges to some broken curve as $t \to t_i$.  Moreover, the indices of the levels of the broken curve sum to 1.  By Proposition \ref{notopbottomconn} we cannot have connectors appear at the top most or bottom most level.  Moreover by compactness and the conservation of Fredholm index, a Fredholm index one connector cannot appear as a middle level in a handleslide building.   If the sequence is close to breaking, cf. Definition \ref{def:close}, then by Lemma \ref{lem:ht} and the definition of $\mathcal{G}_\delta$, one of the following two scenarios occurs:
\begin{enumerate}[(i)]
\item The top most level of the broken curve contains the index 1 component $C_+$ and some lower level contains the index 0 ECH handleslide $C_-$. 
\item The bottom most level of the broken curve contains the index 1 component, $C_-$, and the top most level contains the index 0 ECH handleslide,  $C_+$. 
\end{enumerate}
Moreover, all other components of all levels are index zero branched covers of $\R$-invariant cylinders, e.g. connectors.  By analogy with condition (d) in the definition of a gluing pair, Definition \ref{def:gluepair}, any covers of $\R$-invariant cylinders in the top and bottom levels of the broken curve must be unbranched.  
\end{proof}

Finally, we review the possible bifurcations that appear in a generic 1-parameter family $\{ \cj_t \}_{t\in[0,1]}$:
\begin{proposition}
Fix  a nondegenerate contact form $\lambda$. Then for a 1-parameter family $\{ \cj \}_{t\in[0,1]}$ of $\lambda$-compatible domain dependent almost complex structures with fixed endpoints we may arrange that the only possible bifurcations are:
\begin{enumerate}[\em(a)]
        \setlength{\itemsep}{0pt}
\item A cancellation of two oppositely signed holomorphic curves.
\item An ECH handleslide. 
\end{enumerate}
\end{proposition}

In the case of Morse theory, the corresponding transversality statement is \cite[Lemma 2.11(b)]{torsion}.  In the context of Seiberg-Witten Floer homology, Taubes completes this bifurcation analysis at the end of \cite{taubes-wconj}.  Note that cancellation of two oppositely signed curves does not change the differential.  The presence of an ECH handleslide does, change the differential, but in \S \ref{sec:ECHok}, we show that it does not have an impact after passing to homology.

\subsection{Recap of obstruction bundle gluing}\label{sec:obg}
In this section we collect the results from \cite{obg1} that will be used in the proof of Proposition \ref{ECH-ok}. We state everything in the context considered in proving $\partial^2=0$, and explain in a subsequent remark the difference and continued applicability in the setting under consideration.   In connection with the index calculations for branched covered cylinders over an elliptic embedded Reeb orbit, cf. Lemma \ref{lem:ht}, we can define a partial order on the associated set of partitions, which will be used in the construction of a gluing pair.

\begin{definition}[Partial order $\geq_\vartheta$]\label{def:partial}
Let $\gamma$ be a nondegenerate elliptic embedded Reeb orbit with a fixed irrational rotation number $\vartheta$, cf. \S \ref{cz-sec}. Writing $\alpha = \left(\gamma^{a_1},...,\gamma^{a_k}\right)$ and $\beta=\left(\gamma^{b_1},...,\gamma^{b_\ell}\right)$, consider $C \in \M^J(\alpha,\beta)$a branched cover of $\R \times \gamma$.  We say
\[
(a_1,...,a_k) \geq_\vartheta (b_1,...,b_\ell)
\]
whenever there exists an index zero branched cover of $\R \times \gamma \in \M^J\left(\left(\gamma^{a_1},...,\gamma^{a_k}\right),\left(\gamma^{b_1},...,\gamma^{b_\ell}\right)\right).$
\end{definition}

Following \cite[\S 1.3-1.4]{obg1} we define a gluing pair, prepare for the definition of the count $\#G(C_+, C_-)$, and state the main obstruction bundle gluing theorem.  \begin{definition}\label{def:gluepair}
A \emph{gluing pair} is a pair of immersed $J$-holomorphic curves $C_+(\alpha,\gamma_+)$ and $C_-(\gamma_-,\beta)$ such that:
\begin{enumerate}[(a)]
\item $\ind(C_+)=\ind(C_-)=1$.
\item $C_+$ and $C_-$ are not multiply covered, except that they may contain unbranched covers of $\R$-invariant cylinders.  
\item For each embedded Reeb orbit $\gamma$, the total covering multiplicity of Reeb orbits covering $\gamma$ in the list $\gamma_+$ is the same as the total for $\gamma_-$.  (In contrast, for the usual form of Floer theory gluing, one would assume that $\gamma_+=\gamma_-$.)
\item If $\gamma$ is an elliptic embedded Reeb orbit with rotation angle $\vartheta$, let $m_1',...,m_k'$ denote the covering multiplicities of the $\R$-invariant cylinders over $\gamma$ in $C_+$ and let $n_1',...,n_j'$  denote the corresponding multiplicities in $C_-$.  Then under the partial order $\geq_\vartheta$ in Definition \ref{def:partial}, the partition $(m_1',...,m_k')$ is minimal, and the partition $(n_1',...,n_j')$ is maximal.
\end{enumerate}
\end{definition}

Let $(C_+,C_-)$ be a gluing pair.  The main gluing result of \cite{obg1,obg2} computes an integer $\#G(C_+,C_-)$ which, roughly speaking is a signed count of ends of the index two part of the moduli space $\M^J(\alpha,\beta)/\R$ that break into $C_+$ and $C_-$ along with some index zero connectors (branched covers of $\R$-invariant cylinders between them).  When $C_\pm$ contain covers of $\R$-invariant cylinders, there are some subtleties which require the use of condition (d) above in showing that $\#G(C_+,C_-)$ is well-defined.

Before giving the definition of the count $\#G(C_+,C_-)$, we first define a set $\mathcal{G}_\delta(C_+,C_-)$ of index two curves in $\M^J(\alpha,\beta)$ which, are close to breaking in the above manner.  For the following definition, choose an arbitrary product metric on $\R \times Y$.

\begin{definition}\label{def:close}
For $\delta >0$, define $\mathfrak{C}_\delta(C_+,C_-)$ to be the set of immersed (except possibly at finitely many singular points) surfaces in $\R \times Y$ that can be decomposed as $\mathbf{C}_- \cup \mathbf{C}_0 \cup \mathbf{C}_+$ such that the following hold:
\begin{itemize}
\item There is a real number $R_-$, and a section $\psi_-$ of the normal bundle to $C_-$ with $|\psi_-|<\delta$, such that $\mathbf{C}_-$ is the set $s \mapsto s + R_-$ translate of the $s \leq \frac{1}{\delta}$ portion of the image of $\psi_-$ under the exponential map.
\item Similarly, there is a real number $R_+$, and a section $\psi_+$ of the normal bundle to $C_+$ with $|\psi_+|<\delta$, such that $\mathbf{C}_+$ is the set $s \mapsto s + R_+$ translate of the $s \geq -\frac{1}{\delta}$ portion of the image of $\psi_+$ under the exponential map.
\item $R_+ - R_- > \frac{2}{\delta}.$
\item  $\mathbf{C}_0 $ is contained in the union of the radius $\delta$ tubular neighborhoods of the cylinders $\R \times \gamma$, where $\gamma$ ranges over the embedded Reeb orbits covered by orbits in $\gamma_\pm$
\item $\partial \mathbf{C}_0 = \partial \mathbf{C}_- \sqcup \partial \mathbf{C}_+$, where the positive boundary circles of $\mathbf{C}_-$ agree with the negative boundary circles of $\mathbf{C}_0$, and the positive boundary circles of $\mathbf{C}_0$ agree with the negative boundary circles of $\mathbf{C}_+$.
\end{itemize}
Let $\mathcal{G}_\delta(C_+,C_-)$ denote the set of index two curves in $\M^{J}(\alpha,\beta) \cap \mathfrak{C}_\delta(C_+,C_-)$.
\end{definition}

To see that this definition works as expected, we have the following lemma.  We include the proof, as it elucidates why we can invoke the obstruction bundle gluing formalism in the setting under consideration.

\begin{lemma}\label{lem:glue}
Given a gluing pair $(C_+,C_-)$, there exists $\delta_0 >0$ with the following property.  Let $\delta \in (0,\delta_0)$ and let $\{C(k) \}_{k=1,2,...}$ be a sequence in $\mathcal{G}_\delta(C_+,C_-)/\R$.  Then there is a subsequence which converges in the sense of \cite{BEHWZ} either to a curve in $\M^{J}(\alpha,\beta)/\R $ or to a broken curve in which the top level is $C_+$, the bottom level is $C_-$, and all intermediate levels are unions of index zero branched covers of $\R$-invariant cylinders.  
\end{lemma}

\begin{proof}
By the compactness theorem in \cite{BEHWZ}, any sequence of index 2 curves in $\M^J(\alpha,\beta)/\R$ has a subsequence which converges to some broken curve.  Moreover, the indices of the levels of the broken curve sum to 2.  If the sequence is in $\mathcal{G}_\delta(C_+,C_-)/\R$ with $\delta>0$ sufficiently small then by Lemma \ref{lem:ht} and the definition of $\mathcal{G}_\delta$, one of the following two scenarios occurs:
\begin{enumerate}[(i)]
\item One level of the broken curve contains the index 1 component of $C_+$, and some lower level contains the index 1 component of $C_-$.
\item Some level contains two index 1 components or one index 2 component.
\end{enumerate}
Moreover, all other components of all levels are index zero branched covers of $\R$-invariant cylinders.  By condition (d) in the definition of a gluing pair, any covers of $\R$-invariant cylinders in the top and bottom levels of the broken curve must be unbranched.  It follows that in Case (i), the top level is $C_+$ and the bottom level is $C_-$, while in Case (ii), there are no other levels.
\end{proof}

\begin{definition}\label{def:U}
Fix coherent orientations and generic $\lambda$-compatible $J$ and let $(C_+,C_-)$ be a gluing pair.  If $\delta \in (0,\delta_0)$, then by Lemma \ref{lem:glue}, one can choose an open set $U \subset \M^J(\alpha,\beta)/\R$ such that:
\begin{itemize}
\item $\mathcal{G}_{\delta'}(C_+,C_-)/\R \subset U \subset \mathcal{G}_\delta(C_+,C_-)/\R$ for some $\delta' \in (0,\delta)$.
\item The closure $\overline{U}$ has finitely many boundary points.
\end{itemize} 
Define $\# G(C_+,C_-) \in \Z$ to be minus the signed count of boundary points of $\overline{U}$.  By Lemma \ref{lem:glue}, this does not depend on the choice of $\delta$ or $U$.
\end{definition}

Note that by Lemma \ref{lem:ht}, if $\#G(C_+,C_-)\neq 0$ then for each hyperbolic Reeb orbit $\gamma$, the multiplicities of the negative ends of $C_+$ at covers of $\gamma$ agree, up to reordering, with the multiplicities of the positive ends of $C_-$ at covers of $\gamma$.  When this is the case, assume that the orderings of the negative ends of $C_+$ and of the positive ends of $C_-$ are such that for each positive hyperbolic orbit $\gamma$, the aforementioned multiplicities appear in the same order for $C_+$ and for $C_-$.  With this ordering convention, the statement of the main gluing theorem is as follows:

\begin{theorem}{\em \cite[Theorem 1.13]{obg1}}\label{thm:obg}\label{thm:obg1}
Fix coherent orientations.  If $J$ is generic and if $(C_+,C_-)$ is a gluing pair then
\begin{equation}
\#G(C_+,C_-) = \epsilon(C_+)\epsilon(C_-)\prod_\gamma c_\gamma(C_+,C_-).
\end{equation}
Here the product is over embedded Reeb orbits $\gamma$ such that $C_+$ has a negative end at a cover of $\gamma$.  The integer $c_\gamma(C_+,C_-)$ depends only on $\gamma$ and on the multiplicities of the $\R$-invariant and non-$\R$-invariant negative ends of $C_+$ and positive ends of $C_-$ at covers of $\gamma$.
\end{theorem}

We omit the discussion of the explicit computation of the gluing coefficient $c_\gamma(C_+,C_-)$, as we will multiply this by zero; the budding obstruction bundle enthusiast can find further details in \cite[\S 1.5-1.6]{obg1}.  

\begin{remark}
We have that $c_\gamma(C_+,C_-)=1$ if and only if the ECH partition conditions are met, cf. \S \ref{ECH-ineq}.  However, we cannot guarantee this in practice, as such an argument typically relies on the ECH index inequality, Theorem \ref{thm:indexineq}, which does not apply for domain dependent curves, because its proof requires intersection positivity.
\end{remark}

As a result of \S \ref{sec:handles} (and the classification of connectors in \S \ref{sec:connectors}), we can analogously define  a gluing pair for an ECH handleslide building and invoke Theorem \ref{thm:obg1}.  It remains to explain how this yields \eqref{eq:handle}.

\begin{remark}\label{rem:handleslide}
Roughly speaking and along the lines of \cite{fl}, \cite[\S 3.3]{torsion}, for $\epsilon$ small, the results of \S \ref{sec:handles} yield a one-dimensional cobordism:
\[
\partial \left( \bigcup_{t\in[t_i - \epsilon, t_i + \epsilon]}  \mathcal{M}_{t}(\alpha,\beta) \right) =  \mathcal{M}_{t_i + \epsilon}(\alpha,\beta) -  \mathcal{M}_{t_i + \epsilon}(\alpha,\beta) \mp  \mathcal{M}_{t_i}(\alpha,\gamma_+) \bigsqcup_{\conn \in \M_{t_i}(\gamma_+,\gamma_-)} \mathcal{M}_{t_i}(\gamma_-,\beta).
\]
However, technically speaking we cannot compute the boundary of the compactified moduli space on the left hand side.  Instead, we must truncate this moduli space in order to invoke the obstruction bundle gluing theorem and  obtain \eqref{eq:handle} similarly to the proof that the ECH differential squares to zero, cf.  \cite[Theorem 7.20]{obg1}.  In particular, if $(C_+,C_-)$ is a gluing pair arising from an ECH handleslide in which $C_+ \in \mathcal{M}_{t_i}(\alpha,\gamma_+) $ and $C_- \in \mathcal{M}_{t_i}(\gamma_-,\beta)$, let $V(C_+,C_-) \subset \M_t(\alpha,\beta)$  for $t\in[t_i - \epsilon, t_i + \epsilon]$ be an open set like the open set $U$ in Definition \ref{def:U}, but where the curves do not have any asymptotic markings or orderings of the ends.  We truncate the interior of the cobordism by removing curves which are close to breaking,
\[
\overline \M: = \bigcup_{t\in[t_i - \epsilon, t_i + \epsilon]}  \mathcal{M}_{t}(\alpha,\beta) \ \setminus \bigsqcup_{(C_-,C_+)} V(C_-,C_+).
\]
By the analogue of \cite[Lem.~7.23]{obg1}, namely Lemma \ref{lem:slide}, we have that $\overline \M$ is compact. Because the handleslide is isolated, the signed count of truncated boundary points is
\[
0 = \# \partial \overline{\M} = - \# \partial \overline{V(C_+,C_-)}.
\]
The count $\#G(C_+,C_-) = -\#\partial\overline{U}$ distinguishes curves in $\partial\overline{U}$  that have different asymptotic markings and orderings of the ends, but represent the same element of $\partial\overline{V(C_+,C_-)}$, resulting in \eqref{eq:handle}.

\end{remark}

\subsection{Proof of Proposition \ref{ECH-ok}}\label{sec:ECHok}
Similarly to \cite[\S 7.1.1]{farris}, we complete the proof of Proposition \ref{ECH-ok}, by demonstrating that the occurrence of ECH handleslides do not impact the homology.

\begin{proof}[Proof of Proposition~\ref{ECH-ok}.]
Number the ECH handleslides $t_0,...,t_k$.  We omit the cancellation bifurcations as they do not change the curve counts, and note that one should occur before the first ECH handleslide, as otherwise the moduli spaces under consideration are empty {by Corollary \ref{nogenus}}. Without loss of generality we will always assume that  $C_+$ is the index 1 curve and $C_-$ is the index 0 ECH handleslide curve in an ECH handleslide building.

By Remark \ref{rem:handleslide} we have that at each handleslide $t_i$,
\begin{equation}\label{obg-handleslide}
{\# \mathcal{M}_{t_i + \epsilon}(\alpha,\beta)= \# \mathcal{M}_{t_i - \epsilon}(\alpha, \beta) + \# G(C_+, C_-) \cdot \# \mathcal{M}_{t_i }(\alpha, \gamma_+),}
\end{equation}
where $\gamma_+$ is another (admissible) Reeb current such that $I(\alpha,\gamma_+)=1$ with $C_+ \in \mathcal{M}_{t_i }(\alpha, \gamma_+)$. Note that the connector $\conn \in \mathcal{M}_{t_i }(\gamma_+, \gamma_-)$ and the ECH handleslide curve $C_- \in \mathcal{M}_{t_i }(\gamma_-, \beta)$.  

Since $\cj_0:=\bj$ is a generic $S^1$-invariant domain dependent almost complex structure, by Corollary \ref{nogenus},
\[
\M_{t_0-\epsilon}(\alpha,\beta) = \emptyset.
\]
If we can show for all possible $\gamma_+$ that 
\begin{equation}\label{desire}
\# \M_{t_0}(\alpha, \gamma_+) =0, 
\end{equation}
then \eqref{obg-handleslide} yields
\[
\# \M_{t_0+\epsilon}(\alpha,\beta) = \# \M_{t_0-\epsilon}(\alpha,\beta) + 0 =0. 
\]
An inductive argument on $k \in {0,...,\ell}$, where each $t_k$ realizes an ECH handleslide, would then complete the proof.

It remains to show \eqref{desire}; this will be done by a reductive degree argument. The degree of a connector, as defined in \S \ref{degree}, is always zero.  The degree is also additive, hence 
\[
\op{deg}(C(t)) = \op{deg}(C_+) + \op{deg}(C_-).
\]
Moreover, if the ECH handleslide curve $C_-$ has degree 0, then it is a union of branched covers of cylinders, at least one of which is not a trivial cylinder, as otherwise $C_-$ contain a connector.  But a nontrivial cylinder, and hence the union of cylinders including it, has positive Fredholm index.  Thus the ECH handleslide curve $C_-$ must have positive degree (and positive genus) by Lemma \ref{lem:d0g0}.  Hence, if at a handleslide, $\{C(t) \} \in \M_t(\alpha,\beta)$ converges to an ECH handleslide building $(C_+,\conn, C_-)$ then 
\begin{equation}\label{eq:deg}
\op{deg}(C_+) < \op{deg}(C(t)).
\end{equation}
Going back to the task at hand, consider $\M_{t_0}(\alpha,\gamma_+)$, a smooth one dimensional moduli space of Fredholm and ECH index 1 curves.  We wish to show that 
\[
\# \M_{t_0}(\alpha, \gamma_+) =0.  
\]
This will be accomplished by inductive iteration:  
\begin{enumerate}
\item  Consider another generic one parameter family  $\{\cj'_t\}_{t \in [0,t_0]}$ of domain dependent almost complex structures from $\bj_0'$ to  $\bj_{t_0}$.  Note that before the first handleslide, call it $t_0'$, in this new deformation we have 
\[
\M_{t_0'-\epsilon}(\alpha,\gamma_+) = \emptyset.
\]
\item Use \eqref{obg-handleslide} to relate the curve counts occurring immediately prior to and following the first appearance of an ECH handleslide at $t_0'$ in this new deformation $\{\cj_t'\}_{t\in [0,t_0]}$:
\[
\begin{array}{lcl}
\# \mathcal{M}_{t_0' + \epsilon}(\alpha,\gamma_+) &= &\# \mathcal{M}_{t_0' - \epsilon}(\alpha, \gamma_+) + \# G(C_+', C_-') \cdot \# \mathcal{M}_{t_0' }(\alpha, \gamma'_+) \\
&= & \# G(C_+', C_-') \cdot \# \mathcal{M}_{t_0' }(\alpha, \gamma'_+). \\
\end{array}
\]
\item  Observe the degree reduction  \eqref{eq:deg} for the resulting ECH index 1 component $C_+' \in \M_{t_0'}(\alpha,\gamma'_+)$ in the handleslide, namely,
\[
\op{deg}(C_+')< \op{deg}(C_+)
\]
 because the index 0 ECH handleslide curve must always have positive degree.  
\end{enumerate}
We repeat this process until either the resulting index 1 curve $C_+^{(k)}\in\M_{{t}^{(k)}_0}(\alpha, \gamma^{(k)}_+)$ arising from \eqref{obg-handleslide} associated to $\left\{\cj^{(k)}_t\right\}$ for ${t \in \left[0,t_0^{(k-1)}\right]}$ at the ``next" handleslide at time $t_0^{(k)}$ can no longer degenerate via handleslides or is a degree 1 curve. (Note that we could also stop at degree 2, since transversality for degree 1 curves can be achieved by $S^1$-invariant domain dependent almost complex structures.)  This permits us to conclude that \[
\# \M_{t_0}(\alpha, \gamma_+) =0.  
\]
for all admissible $\gamma_+$ with $I(\alpha,\gamma_+)=1$ as desired.

\end{proof}


\section{Computation of $ECH$}\label{sec:finalcomp}

In this section we prove Theorem \ref{thm:mainthm}. 

Before invoking the results in \S\ref{sec:modspcs}-\S\ref{handleslides} to prove the first and second conclusions of Theorem \ref{thm:mainthm}, which relate the chain complex $\lim_{L\to\infty}ECH_*^L(Y,\lambda_{\varepsilon(L)},\Gamma)$ to $\Lambda^*H_*(\Sigma_g,\Z_2)$, we must first relate $\lim_{L\to\infty}ECH_*^L(Y,\lambda_{\varepsilon(L)},\Gamma)$ to the ECH of the original prequantization bundle. In \S\ref{subsec:directlimit}, we prove 
\begin{theorem}\label{thm:dlisECH} With $Y,\lambda,\varepsilon(L)$ as in Lemma \ref{lem:efromL}, for any $\Gamma \in H_1(Y;\Z)$,
\[
\lim_{L\to\infty}ECH_*^L(Y,\lambda_{\varepsilon(L)},\Gamma)=ECH_*(Y,\xi,\Gamma).
\]
\end{theorem}

In \S\ref{subsec:pfmainthm}, we will use Theorem \ref{thm:dlisECH} to prove Theorem \ref{thm:mainthm} by showing the $\Z_2$-graded isomorphism of $\Z_2$-modules
\begin{equation}\label{eqn:computedl}
\bigoplus_{\Gamma\in H_1(Y)}\lim_{L\to\infty}ECH_*^L(Y,\lambda_{\varepsilon(L)},\Gamma)\cong\Lambda^*H_*(\Sigma_g;\Z_2).
\end{equation}
The idea of the proof is as follows. As a consequence of Lemma \ref{lem:HYreps} (i), the groups $ECC^L_*(Y,\lambda_{\varepsilon(L)},\Gamma;J)$ are all zero unless $\Gamma$ is in the $\Z_{-e}$ summand of $H_1(Y)$. We abuse notation by using $\Gamma$ to also indicate the corresponding element of $\Z_{-e}$. By Proposition \ref{prop:directlimitcomputesfiberhomology} we can restrict attention to orbits above critical points of the Morse function $H$, and by the analysis of \S\ref{sec:modspcs}-\S\ref{handleslides}, the ECH differential ``agrees" with the Morse differential in the sense that if a $J$-holomorphic curve count contributing to the ECH differential were nonzero, then it must equal a count of gradient flow lines defining the Morse differential on $\Sigma_g$. However, because $H$ is perfect, all such counts of gradient flow lines are zero. By analyzing the implications of index parity and the fact that hyperbolic orbits can appear with multiplicity at most one, we prove that the chain complexes have zero differential and satisfy
\begin{equation}\label{eqn:dlG}
\lim_{L\to\infty}ECH^L_*(Y,\lambda_{\varepsilon(L)},\Gamma)\cong\bigoplus_{d\in\Z_{\geq0}}\Lambda^{\Gamma+(-e)d}H_*(\Sigma_g;\Z_2)
\end{equation}
as $\Z_2$-graded $\Z_2$-modules, which implies (\ref{eqn:computedl}).

We conclude this introductory section by proving the third \eqref{eqn:0grading} and fourth \eqref{eqn:relgrading} conclusions of Theorem \ref{thm:mainthm}, assuming the first and second.  

\begin{proof}[Theorem \ref{thm:mainthm}, $\Z$-grading]
In \S\ref{subsec:pfmainthm} we will show that (\ref{eqn:dlG}) follows from the natural map taking an admissible Reeb current to the wedge product the homology classes of the critical points to which $\fp$ sends the orbits. Note that the images of generators with $\op{deg}(\alpha,e_-^\Gamma)=d$ lie in $\Lambda^{\Gamma+(-e)d}H_*(\Sigma_g;\Z_2)$. The $\Z$-valued grading of the image of $e_-^{m_-}h_1^{m_1}\cdots h_{2g}^{m_{2g}}e_+^{m_+}$ under this map is
\[
|e_-^{m_-}h_1^{m_1}\cdots h_{2g}^{m_{2g}}e_+^{m_+}|_\bullet = m_1+\cdots+m_{2g}+2m_+ = M+m_+-m_- = \Gamma-ed+m_+-m_-,
\]
where $d=\op{deg}(e_-^{m_-}h_1^{m_1}\cdots h_{2g}^{m_{2g}}e_+^{m_+},e_-^\Gamma)=\frac{M-\Gamma}{-e}$.
On the other hand,
\begin{align*}
I(e_-^{m_-}h_1^{m_1}\cdots h_{2g}^{m_{2g}}e_+^{m_+},e_-^\Gamma)&=-ed^2+(\chi(\Sigma_g)+2\Gamma)d+m_+-m_-+\Gamma
\\&=-ed^2+(\chi(\Sigma_g)+2\Gamma)d+|e_-^{m_-}h_1^{m_1}\cdots h_{2g}^{m_{2g}}e_+^{m_+}|_\bullet+ed.
\end{align*}
\end{proof}

\subsection{Direct limits and Seiberg-Witten Floer cohomology}\label{subsec:directlimit}

Because there are no Morse-Bott methods for ECH, we must compute $ECH_*(Y,\xi,\Gamma)$ by relating it to $\lim_{L\to\infty}ECH^L_*(Y,\lambda_{\varepsilon(L)},\Gamma)$. Our discussion so far allows us to understand the latter. In this section we prove Theorem \ref{thm:dlisECH}, obtaining
\[
\lim_{L\to\infty}ECH^L_*(Y,\lambda_{\varepsilon(L)},\Gamma)=ECH_*(Y,\xi,\Gamma).
\]

It is not possible to prove Theorem \ref{thm:dlisECH} solely within the realm of ECH. This is because we cannot relate the filtered homologies computed for $\varepsilon>0$ to any chain complex when $\varepsilon=0$, since the contact form is degenerate and the filtered homology is not defined for any $L$. However, ECH is isomorphic to Seiberg-Witten Floer cohomology (including at the level of filtrations), and in the latter case the degeneracy of $\lambda$ is not an issue. The key result is the last equation in \cite[\S3.5]{cc2}, whose role we explain in the proof of Theorem \ref{thm:dlisECH} in \S\ref{sssec:pfThm7.1} below.

In \S\ref{sssec:SWFc} we review Seiberg-Witten Floer cohomology.  In \S\ref{sssec:cfpert} we describe the perturbation of the Seiberg-Witten equations which allows us to incorporate a contact form into the chain complex. In \S\ref{sssec:HML} we explain an energy filtration analogous to that in ECH and collect results from \cite{cc2} about cobordism maps on filtered Seiberg-Witten Floer cohomology necessary for our arguments in \S\ref{sssec:HML}. Our review follows \S\cite[\S2]{cc2}. The proof of Theorem \ref{thm:dlisECH} is completed in \S\ref{sssec:pfThm7.1}.

\subsubsection{Seiberg-Witten Floer cohomology}\label{sssec:SWFc}

For a definition of Seiberg-Witten Floer cohomology in our setting, we refer the reader to \cite[\S2]{cc2}. In this section we quickly review the necessary notation, following \cite[\S2]{cc2}.

Let $Y$ be a closed oriented connected three-manifolds equipped with a Riemannian metric $g$. Recall that a \textit{spin-c structure} on $Y$ is a rank two Hermitian vector bundle $\mathbb{S}$ on $Y$ together with a \textit{Clifford multiplication} $\text{cl}:TY\to\text{End}(\mathbb{S})$.
 We denote a spin-c structure by $\mathfrak{s}=(\mathbb{S},\text{cl})$. Sections of $\mathbb{S}$ are called \textit{spinors}. A \textit{spin-c connection} on a spin-c structure $\mathfrak{s}$ is a connection $\mathbb{A}_\mathbb{S}$ on $\mathbb{S}$ which is compatible with the Clifford multiplication, meaning that if $v$ is a vector field on $Y$ and $\psi$ is a spinor, then
\[
\nabla_{\mathbb{A}_\mathbb{S}}(\text{cl}(v)\psi)=\text{cl}(\nabla v)\psi+\text{cl}(v)\nabla_{\mathbb{A}_\mathbb{S}}\psi,
\]
where $\nabla v$ denotes the covariant derivative of $v$ with respect to the Levi-Civita connection of $g$. Such a spin-c connection is equivalent to a Hermitian connection $\mathbb{A}$ on the determinant line bundle $\det(\mathbb{S})$. The \textit{Dirac operator} $D_\mathbb{A}$ of $\mathbb{A}_\mathbb{S}$ is the composition
\begin{equation}\label{eqn:Do}
C^\infty(Y;\mathbb{S})\overset{\nabla_{\mathbb{A}_\mathbb{S}}}{\to}C^\infty(Y;T^*Y\otimes\mathbb{S})\overset{\text{cl}}{\to}C^\infty(Y;\mathbb{S}),
\end{equation}
where $\text{cl}$ is the composition $T^*Y\overset{g}{\cong}TY\overset{\text{cl}}{\to}\text{End}(\mathbb{S})$.

Let $\eta$ be an exact 2-form on $Y$. Let $\mathbb{A}$ be a Hermitian connection on $\det(\mathbb{S})$ and $\Psi$ be a spinor. Define a bundle map $\tau:\mathbb{S}\to iT^*Y$ by $\tau(\Psi)(v)=g(\text{cl}(v)\Psi,\Psi)$. The \textit{Seiberg-Witten equations} for $(\mathbb{A},\Psi)$ with perturbation $\eta$ are
\begin{equation}\label{eqn:SWeqns}
D_\mathbb{A}\Psi=0 \text{ and }*F_\mathbb{A}=\tau(\Psi)+i*\eta.
\end{equation}
Seiberg-Witten Floer cohomology is generated by certain solutions $(\mathbb{A},\Psi)$ to the (\ref{eqn:SWeqns}). The \textit{gauge group} $\mathcal{G}:=C^\infty(Y;S^1)$ acts on the set of all pairs $(\mathbb{A},\Psi)$ by
\[
u\cdot(\mathbb{A};\Psi):=(\mathbb{A}-2u^{-1}du,u\Psi),
\]
and if $(\mathbb{A},\Psi)$ is a solution to the Seiberg-Witten equations, so is $u\cdot(\mathbb{A},\Psi)$. Solutions are \textit{(gauge) equivalent} if they are equivalent under the action of $\mathcal{G}$. If $\eta$ is generic then modulo gauge equivalence there are only finitely many solutions with $\Psi\not\equiv0$, and each is cut out transversely. Such solutions are called \textit{irreducible} (those with $\Psi\equiv0$ are \textit{reducible}). We assume $\eta$ is sufficiently generic in this way.

Denote by $\widehat{CM}^*_\text{irr}$ the free $\Z_2$-module generated by the irreducible solutions to the Seiberg-Witten equations, modulo gauge equivalence.

We next describe the part of the Seiberg-Witten differential which maps $\widehat{CM}^*_\text{irr}$ to itself. Let $(\mathbb{A}_\pm,\Psi_\pm)$ be solutions to the Seiberg-Witten equations. An \textit{instanton} from $(\mathbb{A}_-,\Psi_-)$ to $(\mathbb{A}_+,\Psi_+)$ is a smooth one-parameter family $(\mathbb{A}(s),\Psi(s))$ parameterized by $s\in\R$, for which
\begin{align}
&\frac{\partial}{\partial s}\Psi(s)=-D_{\mathbb{A}(s)}\Psi(s) \nonumber
\\&\frac{\partial}{\partial s}\mathbb{A}(s)=-*F_{\mathbb{A}(s)}+\tau(\Psi(s))+i*\eta
\\&\lim_{s\to\pm\infty}(\mathbb{A}(s),\Psi(s))=(\mathbb{A}_\pm,\Psi_\pm). \nonumber
\end{align}
The gauge group and $\R$ both act on the space of instantons.

If $(\mathbb{A}_\pm,\Psi_\pm)$ are irreducible, then the differential coefficient $\langle\partial(\mathbb{A}_+,\Psi_+),(\mathbb{A}_-,\Psi_-)\rangle$ is a count of instantons from $(\mathbb{A}_-,\Psi_-)$ to $(\mathbb{A}_+,\Psi_+)$, modulo the actions of $\mathcal{G}$ and $\R$, living in a moduli space of local expected dimension one. This local expected dimension defines a relative $\Z/d(c_1(\mathfrak{s}))$ grading on the chain complex, and the differential increases this grading by one. More generally, so long as there is no moduli space of instantons to $(\mathbb{A}_+,\Psi_+)$ from a reducible solution to the Seiberg-Witten equations of local expected dimension one, then $\partial(\mathbb{A}_+,\Psi_+)\in\widehat{CM}^*_\text{irr}$. For a discussion of further abstract perturbations necessary to obtain the transversality required to fully define the differential, see \cite[\S2.1]{cc2}; they are not necessary in our arguments.

There is a differential on the chain complex generated over $\Z_2$ by all solutions to the Seiberg-Witten equations, modulo gauge equivalence, whose differential extends the differential from $\widehat{CM}^*_\text{irr}$ to itself discussed above. We will not need to discuss this extension further here, because the key to the proof of Theorem \ref{thm:dlisECH} is a filtered version of Seiberg-Witten Floer cohomology whose generators are all irreducible, introduced in \S\ref{sssec:HML}. However, we do mention that the homology of the chain complex including reducible solutions is denoted by $\widehat{HM}^*(Y,\mathfrak{s};g,\eta)$. Because this homology is independent of the choices of $(g,\eta)$, we denote the canonical isomorphism class of all such homologies by $\widehat{HM}^*(Y,\mathfrak{s})$.

By $\mathfrak{s}_{\xi,\Gamma}$ we denote the spin-c structure $\mathfrak{s}_\xi+PD(\Gamma)$ on $Y$, where $\mathfrak{s}_\xi$ is the spin-c structure determined by $\xi$ as in \cite[Example 2.1]{cc2}. Taubes \cite{taubesechswf1}-\cite{taubesechswf4} has shown
\begin{equation}\label{eqn:fullECHHMiso}
ECH_*(Y,\lambda,\Gamma;J)\cong\widehat{HM}^{-*}(Y,\mathfrak{s}_{\xi,\Gamma});
\end{equation}
here we use the notation for ECH emphasizing the roles of $\lambda$ and $J$, although inherent in the result is the fact that both sides do not depend on $\lambda$ or $J$ but only on $(Y,\xi,\Gamma)$.

\subsubsection{The contact form perturbation of the Seiberg-Witten equations}\label{sssec:cfpert}

If $Y$ has a contact form $\lambda$, let $J$ be an almost complex structure on $\xi$ which extends to a $\lambda$-compatible almost complex structure on $\R\times Y$. From $\lambda$ and $J$ we obtain a metric $g$ for which $g(R,R)=1$ and $g(R,\xi)=0$. In particular, on $\xi$, we have $g(v,w)=\frac{1}{2}d\lambda(v,Jw)$. 
Any spin-c structure $\mathfrak{s}=(\mathbb{S},\text{cl})$ can be canonically decomposed into eigenbundles of $\text{cl}(\lambda)$, i.e. $\mathbb{S}=E\oplus\xi E$, where $E$ is the $i$ eigenbundle and concatenation denotes the tensor product of line bundles. In this decomposition, a connection $\mathbb{A}$ on $\det(\mathbb{S})=\xi E^2$ can be written $\mathbb{A}=A_\xi+2A$ for some connection $A$ on $E$. Similarly to (\ref{eqn:Do}) we can define the Dirac operator $D_A$ of $A$.

Let $r>0$ and $\mu$ be an exact 2-form satisfying the genericity conditions described in \cite[\S2.2]{cc2}. Replacing $D_\mathbb{A}$ with $D_A$ and setting
\[
\eta=-rd\lambda+2\mu,\; \psi=\frac{1}{\sqrt{2r}}\Psi
\]
in (\ref{eqn:SWeqns}) gives us the perturbed Seiberg-Witten equations 
for the pair $(A,\psi)$:
\begin{equation}\label{eqn:pSWeqns}
D_A\Psi=0 \text{ and } *F_A=r(\tau(\Psi)-i\lambda)-\frac{1}{2}*F_{A_\xi}+i*\mu.
\end{equation}

Taking into account abstract perturbations as in \cite[\S2.2]{cc2}, the chain complex $\widehat{CM}^*(Y,\mathfrak{s};\lambda,J,r)$ is generated by the solutions to (\ref{eqn:pSWeqns}), modulo gauge equivalence, and its homology is denoted $\widehat{HM}^*(Y,\mathfrak{s};\lambda,J,r)$. As in the original chain complex, if $(A_\pm,\psi_\pm)$ are irreducible solutions, then $\langle\partial(A_+,\psi_+),(A_-,\psi_-)\rangle$ is a count (modulo the actions of $\mathcal{G}$ and $\R$) of solutions to the perturbed instanton equations
\begin{align}
&\frac{\partial}{\partial s}\psi(s)=-D_{A(s)}\psi(s) \nonumber
\\&\frac{\partial}{\partial s}A(s)=-*F_{A(s)}+r(\tau(\psi(s))-i\lambda)-\frac{1}{2}*F_{A_\xi}+i*\mu
\\&\lim_{s\to\pm\infty}(A(s),\psi(s))=(A_\pm,\psi_\pm), \nonumber
\end{align}
which live in a moduli space of local expected dimension one. Again we denote by $\widehat{CM}^*_\text{irr}$ the component of $\widehat{CM}^*(Y,\mathfrak{s};\lambda,J,r)$ generated by irreducible solutions to (\ref{eqn:pSWeqns}). Although we have used the notation $\widehat{CM}^*_\text{irr}$ for two different subcomplexes (as is the convention in \cite{cc2}), from now on we only use it to denote the subcomplex generated by the irreducible solutions subject to the contact form perturbation.

\subsubsection{The energy filtration on $\widehat{HM}^*$}\label{sssec:HML}
Analogous to the action of Reeb currents, the \textit{energy} of a solution $(A,\psi)$ to the perturbed Seiberg-Witten equations (\ref{eqn:pSWeqns}) is defined as
\[
E(A):=i\int_Y\lambda\wedge F_A.
\]
In analogy to $ECH^L_*$, for $L>0$, define $\widehat{CM}^*_L$ to be the submodule of $\widehat{CM}^*_\text{irr}$ generated by the irreducible solutions to (\ref{eqn:pSWeqns}) with $E(A)<2\pi L$.

Note that the energy of a reducible solution $(A,0)$ to (\ref{eqn:pSWeqns}) is a linear increasing function in $r$, so if $r$ is sufficiently large then the condition that elements of $\widehat{CM}^*_L$ be elements of $\widehat{CM}^*_\text{irr}$ is redundant: if $E(A)<2\pi L$ then if $r$ is large enough, the pair $(A,0)$ cannot be a solution to (\ref{eqn:pSWeqns}).

We quote a lemma necessary for defining the homology of the submodule $\widehat{CM}^*_L$:
\begin{lemma}[{\cite[Lem. 2.3]{cc2}}]\label{lem:CMLsubcx}
Fix $Y,\lambda, J$ as above and $L\in\R$. Suppose that $\lambda$ has no Reeb current of action exactly $L$. Fix $r$ sufficiently large, and a 2-form $\mu$ so that all irreducible solutions to (\ref{eqn:pSWeqns}) are cut out transversely. Then for every $\mathfrak{s}$ and for every sufficiently small generic abstract perturbation, $\widehat{CM}^*_L(Y,\mathfrak{s};\lambda,J,r)$ is a subcomplex of $\widehat{CM}^*(Y,\mathfrak{s};\lambda,J,r)$.
\end{lemma}

When the hypotheses of Lemma \ref{lem:CMLsubcx} apply, we denote the homology of $\widehat{CM}^*_L(Y,\mathfrak{s};\lambda,J,r)$ by $\widehat{HM}^*_L(Y,\lambda,\mathfrak{s})$. In particular, if $r$ is sufficiently large then this homology is independent of $\mu$ and $r$, and it is also independent of $J$, as shown in \cite[Cor. 3.5]{cc2}. We will use the notation $\widehat{HM}^*_L(Y;\lambda,J,r)$ when we wish to emphasize the roles of $J$ and $r$.

Filtered Seiberg-Witten Floer cohomology is isomorphic to ECH:
\begin{lemma}[{\cite[Lem. 3.7]{cc2}}] Suppose that $\lambda$ is $L$-nondegenerate and $J$ is $ECH^L$-generic (see Lemma \ref{lem:nicecobmap}). Then for all $\Gamma\in H_1(Y)$, there is a canonical isomorphism of relatively graded $\Z_2$-modules
\begin{equation}\label{eqn:filterediso}
\Psi^L:ECH^L_*(Y,\lambda,\Gamma;J)\overset{\cong}{\to}\widehat{HM}^{-*}_L(Y,\lambda,\mathfrak{s}_{\xi,\Gamma}).
\end{equation}
\end{lemma}

Analogous to the cobordism maps on $ECH^L_*$, there are cobordism maps on $\widehat{HM}^*_L$. The following is a modified version of \cite[Cor. 5.3 (a)]{cc2} which keeps track of the spin-c structures in our setting. Note that therefore our notation for the cobordism maps on $\widehat{HM}^*_L$ differs slightly from that of \cite{cc2}.
\begin{lemma}\label{lem:HMLcobmapdef} Let $(X,\lambda)$ be an exact symplectic cobordism from $(Y_+,\lambda_+)$ to $(Y_-,\lambda_-)$ where $\lambda_\pm$ is $L$-nondegenerate. Let $\mathfrak{s}$ be a spin-c structure on $X$ and let $\mathfrak{s}_\pm$ denote its restrictions to $Y_\pm$, respectively. Let $J_\pm$ be $\lambda_\pm$-compatible almost complex structures. Suppose $r$ is sufficiently large. Fix 2-forms $\mu_\pm$ and small abstract perturbations sufficient to define the chain complexes $\widehat{CM}^*(Y_\pm,\mathfrak{s}_\pm;\lambda_\pm,J_\pm,r)$. Then there is a well-defined map
\begin{equation}\label{eqn:HMLcobmap}
\widehat{HM}^*_L(X,\lambda,\mathfrak{s}):\widehat{HM}^*_L(Y_+,\mathfrak{s}_+;\lambda_+,J_+,r)\to\widehat{HM}^*_L(Y_-,\mathfrak{s}_-;\lambda_-,J_-,r),
\end{equation}
depending only on $X,\mathfrak{s},\lambda,L,r,J_\pm,\mu_\pm$, and the perturbations, such that if $L'<L$ and if $\lambda_\pm$ are also $L'$-nondegenerate, then the diagram
\begin{equation}\label{eqn:HMLcd}
\xymatrixcolsep{5pc}\xymatrix{
\widehat{HM}^*_{L'}(Y_+,\mathfrak{s}_+;\lambda_+,J_+,r) \ar[r]^{\widehat{HM}^*_{L'}(X,\lambda,\mathfrak{s})} \ar[d] & \widehat{HM}^*_{L'}(Y_-,\mathfrak{s}_-;\lambda_-,J_-,r) \ar[d]
\\\widehat{HM}^*_L(Y_+,\mathfrak{s}_+;\lambda_+,J_+,r) \ar[r]_{\widehat{HM}^*_L(X,\lambda,\mathfrak{s})} & \widehat{HM}^*_L(Y_-,\mathfrak{s}_-;\lambda_-,J_-,r)
}
\end{equation}
commutes, where the vertical arrows are induced by inclusions of chain complexes.
\end{lemma}

Finally, in order to define direct systems, we will need to compose cobordism maps on $\widehat{HM}^*_L$. For certain cobordisms (e.g. those defining the direct system $\lim_{\varepsilon\to0}\widehat{HM}^{-*}_{L(\varepsilon)}(Y,\lambda_\varepsilon,\mathfrak{s}_{\xi,\Gamma})$ it is enough to use the composition property \cite[Lem. 3.4 (b)]{cc2}, but we will need to understand cobordism maps on slightly more complex cobordisms as well. The next lemma is a version of \cite[Prop. 5.4]{cc2} explaining the composition law for $\widehat{HM}^*_L$ in our setting.

In the following lemma, we will consider the following composition. Assume $\varepsilon''<\varepsilon'<\varepsilon$. We consider the exact symplectic cobordism $([\varepsilon'',\varepsilon]\times Y,(1+s\fp^*H)\lambda)$ from $(Y,\lambda_\varepsilon)$ to $(Y,\lambda_{\varepsilon''})$. It is the composition of the exact symplectic cobordism $([\varepsilon'',\varepsilon']\times Y,(1+s\fp^*H)\lambda)$ from $(Y,\lambda_{\varepsilon'})$ to $(Y,\lambda_{\varepsilon''})$ with the exact symplectic cobordism $([\varepsilon',\varepsilon]\times Y,(1+s\fp^*H)\lambda)$ from $(Y,\lambda_\varepsilon)$ to $(Y,\lambda_{\varepsilon'})$ in the sense of \cite[\S1.5]{cc2}, where $\lambda_\varepsilon, \lambda_{\varepsilon'}$, and $\lambda_{\varepsilon''}$ are $L$-nondegenerate. We also assume $J, J'$, and $J''$ are $\lambda_\varepsilon$-, $\lambda_{\varepsilon'}$-, and $\lambda_{\varepsilon''}$-compatible almost complex structures, respectively. Further, we choose a spin-c structure $\mathfrak{s}''$ on $[\varepsilon'',\varepsilon]\times Y$ which restricts to spin-c structures $\mathfrak{s}'$ and $\mathfrak{s}$ on $[\varepsilon'',\varepsilon']\times Y$ and $[\varepsilon',\varepsilon]\times Y$, respectively, where $\mathfrak{s}'$ restricts to $\mathfrak{s}_2$ on $\{\varepsilon''\}\times Y$, $\mathfrak{s}$ restricts to $\mathfrak{s}_0$ on $\{\varepsilon\}\times Y$, and both $\mathfrak{s}'$ and $\mathfrak{s}$ restrict to $\mathfrak{s}_1$ on $\{\varepsilon'\}\times Y$. Finally we choose abstract perturbations and $r$ large enough to define the chain complexes $\widehat{CM}^*_L$.
\begin{lemma}\label{lem:HMLcomp} The maps of Lemma \ref{lem:HMLcobmapdef} for the above data satisfy
\[
\widehat{HM}^*_L([\varepsilon'',\varepsilon]\times Y,(1+s\fp^*H)\lambda,\mathfrak{s})=\widehat{HM}^*_L([\varepsilon'',\varepsilon']\times Y,(1+s\fp^*H)\lambda,\mathfrak{s}^+)\circ\widehat{HM}^*_L([\varepsilon',\varepsilon]\times Y,(1+s\fp^*H)\lambda,\mathfrak{s}^-).
\]
\end{lemma}
Note that \cite[Prop. 5.4]{cc2} does not discuss the spin-c structures, but since it is proved with a neck-stretching argument for holomorphic curves whose ends must be homologous, it will preserve spin-c structures in the case considered in Lemma \ref{lem:HMLcomp}, see \cite[Rmk. 1.10]{cc2}.

\subsubsection{$ECH_*$ via $\widehat{HM}^*_L$}\label{sssec:pfThm7.1}

In this section we prove Theorem \ref{thm:dlisECH} using the machinery from Seiberg-Witten theory reviewed in the previous sections.

\begin{proof}[Proof of Theorem \ref{thm:dlisECH}]
Because all $\lambda_\varepsilon$ have the same contact structure $\xi$ as $\lambda$, we have
\[
\lim_{L\to\infty}ECH^L_*(Y,\lambda_\varepsilon,\Gamma)=ECH_*(Y,\xi,\Gamma).
\]
However, if $\varepsilon>\varepsilon(L)$ then we cannot compute $ECH^L_*(Y,\lambda_\varepsilon,\Gamma)$ using our methods, because the chain complex $ECC^L_*(Y,\lambda_\varepsilon,\Gamma;J)$ may contain orbits which do not project to critical points of $H$. If $\varepsilon$ is fixed and only $L$ is sent to $\infty$, then because $\varepsilon(L)\sim\frac{1}{L}$, there will be some $L$ beyond which $\varepsilon>\varepsilon(L)$ and we can no longer compute $ECH^L_*(Y,\lambda_\varepsilon,\Gamma)$.

Instead, we will explain how to obtain $ECH_*(Y,\xi,\Gamma)$ from $\lim_{L\to\infty}ECH^L_*(Y,\lambda_{\varepsilon(L)},\Gamma)$. Let $L(\varepsilon)$ denote the value of $L$ for which $\varepsilon(L)=\varepsilon$. Note that for all $L<L(\varepsilon)$, the generators of $ECH^L_*(Y,\lambda_\varepsilon,\Gamma)$ all project to critical points of $H$. In particular,
\[
ECH^{L(\varepsilon)}_*(Y,\lambda_\varepsilon,\Gamma)=ECH^L_*(Y,\lambda_{\varepsilon(L)},\Gamma)
\]
and therefore
\[
\lim_{L\to\infty}ECH^L_*(Y,\lambda_{\varepsilon(L)},\Gamma)=\lim_{\varepsilon\to0}ECH^{L(\varepsilon)}_*(Y,\lambda_\varepsilon,\Gamma).
\]

To prove Theorem \ref{thm:dlisECH} we will prove the following sequence of isomorphisms:
\begin{align}
\lim_{\varepsilon\to0}ECH_*^{L(\varepsilon)}(Y,\lambda_\varepsilon,\Gamma)&\cong\lim_{\varepsilon\to0}\widehat{HM}^{-*}_{L(\varepsilon)}(Y,\lambda_\varepsilon,\mathfrak{s}_{\xi,\Gamma})\label{eqn:ECHeHMe}
\\&\cong\lim_{\varepsilon\to0}\lim_{L\to\infty}\widehat{HM}^{-*}_L(Y,\lambda_\varepsilon,\mathfrak{s}_{\xi,\Gamma})\label{eqn:factoringdleL}
\\&\cong\lim_{\varepsilon\to0}\widehat{HM}^{-*}(Y,\mathfrak{s}_{\xi,\Gamma})\label{eqn:sendLtoinfty}
\\&\cong ECH_*(Y,\xi,\Gamma).\label{eqn:HMtoECH}
\end{align}
Note that the groups $\widehat{HM}^{-*}_L(Y,\lambda_\varepsilon,\mathfrak{s}_{\xi,\Gamma})$ on the right hand side of the second equation (\ref{eqn:factoringdleL}) are only defined for $L$ and $\varepsilon$ such that $\lambda_\varepsilon$ has no Reeb currents of action exactly $L$. This includes all $L\leq L(\varepsilon)$ (similarly all $\varepsilon<\varepsilon(L)$); for a given $\varepsilon$, this is still a full measure set of $L$ because for generic perfect $H$, the set of actions of orbits of $X_H$ is discrete.

The direct limit on the right hand side of the first equation (\ref{eqn:ECHeHMe}) is defined using either composition in the commutative diagram (\ref{eqn:HMLcd}) and the exact symplectic cobordisms $([\varepsilon',\varepsilon]\times Y,(1+s\fp^*H)\lambda)$ discussed in the proof of Proposition \ref{prop:directlimitcomputesfiberhomology} in \S\ref{subsec:PQBgens}. That these maps compose properly is derived from the composability of the cobordism maps in Lemma \ref{lem:HMLcomp}, following the same logic in the proof of Proposition \ref{prop:directlimitcomputesfiberhomology} to prove that the maps in the direct system on the left hand side of (\ref{eqn:ECHeHMe}) compose. (\ref{eqn:ECHeHMe}) follows from the isomorphism (\ref{eqn:filterediso}).

The direct limit in $L$ on the right hand side of the second equation (\ref{eqn:factoringdleL}) is defined using the maps induced on homology by the inclusion of chain complexes. 
The direct limit in $\varepsilon$ is defined using the cobordism maps (\ref{eqn:HMLcobmap}) and the same exact symplectic cobordisms as in the above paragraph.

To obtain (\ref{eqn:factoringdleL}), we first show that the ``obvious" map is well-defined:
\[
F:\lim_{\varepsilon\to0}\widehat{HM}^{-*}_{L(\varepsilon)}(Y,\lambda_\varepsilon,\mathfrak{s}_{\xi,\Gamma})\to\lim_{\varepsilon\to0}\lim_{L\to\infty}\widehat{HM}^{-*}_L(Y,\lambda_\varepsilon,\mathfrak{s}_{\xi,\Gamma}).
\]
Note that every element of the direct system on the left hand side appears on the right hand side; the map $F$ is given by sending the equivalence class of $a\in\widehat{HM}^{-*}_{L(\varepsilon)}(Y,\lambda_\varepsilon,\mathfrak{s}_{\xi,\Gamma})$ as a member of the left hand direct limit to its equivalence class as a member of the right hand direct limit. Throughout the proof of (\ref{eqn:factoringdleL}) we will use the notation
\[
\widehat{HM}^{-*}_L(\varepsilon):=\widehat{HM}^{-*}_{L}(Y,\lambda_\varepsilon,\mathfrak{s}_{\xi,\Gamma}).
\]

To show that $F$ is well-defined, assume $a\in\widehat{HM}^{-*}_{L(\varepsilon)}(\varepsilon), b\in\widehat{HM}^{-*}_{L(\varepsilon')}(\varepsilon')$, and $a\sim b$ as elements of $\lim_{\epsilon\to0}\widehat{HM}^{-*}_{L(\varepsilon)}(\varepsilon)$. That means there is some $\varepsilon''$ for which the image of $a$ under the map
\begin{equation}\label{eqn:aLHS}
\widehat{HM}^{-*}_{L(\varepsilon)}(\varepsilon)\to\widehat{HM}^{-*}_{L(\varepsilon'')}(\varepsilon'')
\end{equation}
equals the image of $b$ under the map
\[
\widehat{HM}^{-*}_{L(\varepsilon')}(\varepsilon')\to\widehat{HM}^{-*}_{L(\varepsilon'')}(\varepsilon'').
\]
Call this shared image $c$. On the right hand side, we also have $F(a)\sim F(c)$ under the composition
\begin{equation}\label{eqn:RHScomp}
\widehat{HM}^{-*}_{L(\varepsilon)}(\varepsilon)\to\widehat{HM}^{-*}_{L(\varepsilon'')}(\varepsilon)\to\widehat{HM}^{-*}_{L(\varepsilon'')}(\varepsilon''),
\end{equation}
where the first map comes from the first direct limit (and is defined because $\varepsilon''<\varepsilon$) and the second from the second, because (\ref{eqn:RHScomp}) is precisely one of the compositions defining (\ref{eqn:aLHS}) in the commutative diagram (\ref{eqn:HMLcd}). Similarly, we have $F(b)\sim F(c)$.

Next we show that $F$ is an injection. Let $a\in\widehat{HM}^{-*}_{L(\varepsilon)}(\varepsilon), b\in\widehat{HM}^{-*}_{L(\varepsilon')}(\varepsilon')$, but we do not assume $a\sim b$ on the left hand side. Assume $F(a)\sim F(b)$ on the right hand side. We want to show $a\sim b$. Because $F(a)\sim F(b)$, there are $L''\geq L(\varepsilon), L(\varepsilon')$ and $\varepsilon''\leq\varepsilon,\varepsilon'$ for which the image of $a$ under
\[
\widehat{HM}^{-*}_{L(\varepsilon)}(\varepsilon)\to\widehat{HM}^{-*}_{L''}(\varepsilon)\to\widehat{HM}^{-*}_{L''}(\varepsilon'')
\]
equals the image of $b$ under
\[
\widehat{HM}^{-*}_{L(\varepsilon')}(\varepsilon')\to\widehat{HM}^{-*}_{L''}(\varepsilon)\to\widehat{HM}^{-*}_{L''}(\varepsilon'').
\]
Call this shared image $d$. Let $\varepsilon'''=\min\{\varepsilon'',\varepsilon(L'')\}$. We have $d=F(c)$, where $c$ is the image of $a$ under
\[
\widehat{HM}^{-*}_{L(\varepsilon)}(\varepsilon)\to\widehat{HM}^{-*}_{L(\varepsilon''')}(\varepsilon'''),
\]
as well as the image of $b$ under
\[
\widehat{HM}^{-*}_{L(\varepsilon')}(\varepsilon')\to\widehat{HM}^{-*}_{L(\varepsilon''')}(\varepsilon'''),
\]
both again by the definition (\ref{eqn:HMLcd}) of the maps in the direct limit on the left hand side. Therefore $a\sim b$ on the left hand side.

Finally we show that $F$ is a surjection, essentially because the direct systems on the right hand side of (\ref{eqn:factoringdleL}) are very simple. If $d\in\widehat{HM}^{-*}_L(\varepsilon)$ with $L<L(\varepsilon)$, then it is eventually equivalent to some element of $\widehat{HM}^{-*}_{L(\varepsilon)}(\varepsilon)$ because there is an inclusion map sending $\widehat{HM}^{-*}_L(\varepsilon)$ to $\widehat{HM}^{-*}_{L(\varepsilon)}(\varepsilon)$. If $d\in\widehat{HM}^{-*}_L(\varepsilon)$ with $L>L(\varepsilon)$, then it is eventually equivalent to some element of $\widehat{HM}^{-*}_L(\varepsilon(L))$ because there is a cobordism map sending $\widehat{HM}^{-*}_L(\varepsilon)$ to $\widehat{HM}^{-*}_{L}(\varepsilon(L))$.

We have that (\ref{eqn:sendLtoinfty}) follows from the last equation of \cite[\S3.5]{cc2}, which itself follows from \cite[Thm. 4.5]{taubesechswf1}. Note that although the equation in \cite[\S3.5]{cc2} is only required to hold for nondegenerate $\lambda$, it is true for all $\lambda$.

We obtain (\ref{eqn:HMtoECH}) by the fact that the groups $\widehat{HM}^{-*}(Y,\mathfrak{s}_{\xi,\Gamma})$ on the right hand side of (\ref{eqn:sendLtoinfty}) are all equal and independent of $\varepsilon$, together with the isomorphism (\ref{eqn:fullECHHMiso}).

\end{proof}

\subsection{Proof of the main theorem}\label{subsec:pfmainthm}

We split the proof of the first two conclusions of Theorem \ref{thm:mainthm} into the case where $g>0$ and the case where $g=0$. This is because our methods for $g>0$ with domain-dependent almost complex structures do not work when $g=0$: see Remark \ref{unstable-cases}. Instead, because a perfect Morse function on $S^2$ has only elliptic critical points, the differential vanishes by index parity (Theorem \ref{thm:Iproperties}), as explained in \S\ref{sssec:pfg=0}.

\subsubsection{Proof of the main theorem when $g>0$}\label{sssec:pfg>0}

In this section we prove the first two conclusions of Theorem \ref{thm:mainthm} in the case $g>0$.

\begin{proof}[Proof of Theorem \ref{thm:mainthm}, assuming $g>0$]
By Theorem \ref{thm:dlisECH} and the discussion in the introduction to \S\ref{sec:finalcomp}, it is enough to show (\ref{eqn:dlG}). By Proposition \ref{prop:directlimitcomputesfiberhomology}, the direct limit on the left hand side of (\ref{eqn:dlG}) is the homology of the chain complex generated by Reeb currents of the form $e_-^{m_-}h_1^{m_1}\cdots h_{2g}^{m_{2g}}e_+^{m_+}$ in the class $\Gamma$. We will denote this chain complex by $(C_*,\partial)$.

Recall that we are abusing notation by thinking of $\Gamma$ as an element of $\Z_{-e}$ rather than as an element of the $\Z_{-e}$ summand of $H_1(Y)$. Therefore $e_-^{m_-}h_1^{m_1}\cdots h_{2g}^{m_{2g}}e_+^{m_+}$ is in the class $\Gamma$ precisely when $M=m_-+m_1+\cdots+m_{2g}+m_+=\Gamma+(-e)d$ for some $m\in\Z_{\geq0}$.

We claim that the chain complex $C_*$ splits into the submodules $\Lambda^{\Gamma+(-e)d}H_*(\Sigma_g;\Z_2)$ on the right hand side of (\ref{eqn:dlG}). Let $E_-$ denote the index zero generator of $H_*(\Sigma_g;\Z_2)$, let $H_i$ denote the $i^\text{th}$ index one generator, and let $E_+$ denote the index two generator. If $A$ is a generator of $H_*(\Sigma_g;\Z_2)$ let $A^m$ denote the $m$-fold wedge product $A\wedge\cdots\wedge A$, where $m=0$ indicates that there is no factor of $A$ in the wedge product. The $\Z_2$ grading on $E_-^{m_-}\wedge H_1^{m_1}\wedge\cdots\wedge H_{2g}^{m_{2g}}\wedge E_+^{m_+}$ is the $\Z_2$ equivalence class of
\begin{align}
m_-|E_-|+m_1|H_1|+\cdots+m_{2g}|H_{2g}|+m_+|E_+|&\equiv_2 0+m_1+\cdots+m_{2g}+2m_+ \nonumber
\\&\equiv_2 m_1+\cdots+m_{2g} \nonumber
\\&\equiv_2 I(e_-^{m_-}h_1^{m_1}\cdots h_{2g}^{m_{2g}}e_+^{m_+},e_-^\Gamma) \label{eqn:Imod2}
\end{align}
where (\ref{eqn:Imod2}) follows from the Index Parity property of the ECH index, see Theorem \ref{thm:Iproperties} (iv). Moreover, the exterior product $\Lambda^{\Gamma+(-e)d}H_*(\Sigma_g;\Z_2)$ consists precisely of wedge products of the $E_\pm$ and $H_i$ of total multiplicity $\Gamma+(-e)d$, where for all $A,B$ generators of $H_*(\Sigma_g;\Z_2)$,
\[
B\wedge A=(-1)^{|A|\cdot|B|}A\wedge B.
\]
Therefore all elements of the exterior product can be rearranged so that all $E_-$ terms occur first and all $E_+$ terms occur last. However, no $H_i$ term can occur twice, and because we are using $\Z_2$ coefficients,
\[
H_j\wedge H_i=(-1)^{1\cdot1}H_i\wedge H_j=(-1)H_i\wedge H_j=H_i\wedge H_j,
\]
hence all $H_i$ terms can be arranged in the order of ascending index. Therefore the map
\[
e_-^{m_-}h_1^{m_1}\cdots h_{2g}^{m_{2g}}e_+^{m_+}\leftrightarrow E_-^{m_-}\wedge H_1^{m_1}\wedge\cdots\wedge H_{2g}^{m_{2g}}\wedge E_+^{m_+}
\]
defines a bijection
\[
C_*\leftrightarrow \bigoplus_{d\in\Z_{\geq0}}\Lambda^{\Gamma+(-e)d}H_*(\Sigma_g;\Z_2),
\]
which respects the splitting over $d$ in the sense that there is a splitting of $C_*$ over $d$ such that $M=\Gamma+(-e)d$, and respecting the mod two gradings.

Finally, we show that the differential vanishes, so that these submodules are in fact subcomplexes, and (\ref{eqn:dlG}) holds. By Corollary \ref{nogenusJ} we can assume that all all ECH index one curves contributing to $\partial$ have degree zero and thus consist of a union of cylinders by Lemma \ref{lem:d0g0}. By Proposition \ref{cylinder-to-morse}, we know that any moduli space of cylinders which could contribute a nonzero coefficient to $\partial$ must have the same mod two count as the space of currents consisting of trivial cylinders together with a single ECH index one cylinder above a gradient flow line of $H$ which contributes to the Morse differential, and that the count must be the same as the corresponding count for the Morse differential. Specifically, if there is a pseudoholomorphic current contributing to $\langle\partial(e_-^{m_-}h_1^{m_1}\cdots h_{2g}^{m_{2g}}e_+^{m_+}), e_-^{m'_-}h_1^{m'_1}\cdots h_{2g}^{m'_{2g}}e_+^{m'_+}\rangle$ then either
\begin{itemize}
\item There is some $h_i$ for which $m_i=1, m'_i=0$, and $m'_-=m_-+1, m'_+=m_+$, and for $j\neq i$, $m'_j=m_j$.
\item There is some $h_i$ for which $m_i=0, m'_i=1$, and $m'_+=m_+-1, m'_-=m_-$, and for $j\neq i$, $m'_j=m_j$.
\end{itemize}
In either case, the mod two count of all such pseudoholomorphic currents must equal the count of Morse flow lines from $\fp(h_i)$ to $\fp(e_-)$ or $\fp(e_+)$ to $\fp(h_i)$, respectively. Because $H$ is perfect, both of the latter counts are zero.

\end{proof}

\subsubsection{Proof of the main theorem when $g=0$}\label{sssec:pfg=0}

In this section we prove the first two conclusions of Theorem \ref{thm:mainthm} in the case $g=0$. Note that the manifolds in question are the lens spaces $L(-e,1)$.

\begin{theorem}\label{thm:echS2} Let $(Y,\lambda)$ be the prequantization bundle of Euler number $e\in\Z_{<0}$ over $(S^2,\omega)$ with contact structure $\xi$.  (We assume the cohomology class $\frac{[\omega]}{2\pi}$ admits an integral lift in $H^2(S^2,\Z)$.  Moreover $e=\frac{-1}{2\pi}[\omega]$).  Let $\Gamma$ be a class in $H_1(Y;\Z)=\Z_{-e}$. Then
\begin{equation}\label{eqn:mainthmg=0}
\bigoplus_{\Gamma\in H_1(Y)}ECH_*(Y,\xi,\Gamma)\cong\Lambda^*H_*(S^2;\Z_2)
\end{equation}
as $\Z_2$-graded $\Z_2$-modules. Furthermore, as a $\Z$-module,
\begin{equation}\label{eqn:ECHg=0}
ECH_*(Y,\xi,\Gamma)=\begin{cases}
\Z&\text{ if $*\in2\Z_{\geq0}$},
\\0&\text{ else}.
\end{cases}
\end{equation}
\end{theorem}

\begin{remark}
Recall that we may use $\Z$-coefficients rather than $\Z_2$-coefficients when the moduli spaces defining $\partial$ are empty, a fact we prove here. See Remark \ref{rmk:Zcoeff} for further discussion on coefficients in ECH.
\end{remark}

\begin{proof}
By Theorem \ref{thm:dlisECH}, it is enough to understand each $\lim_{L\to\infty}ECH^L_*(Y,\lambda_{\varepsilon(L)},\Gamma)$, which, by Proposition \ref{prop:directlimitcomputesfiberhomology}, is the homology of the chain complex generated by Reeb currents of the form $e_-^{m_-}e_+^{m_+}$ in the class $\Gamma$. We will denote this chain complex by $(C_*,\partial)$.

Because a perfect Morse function on $S^2$ has only elliptic critical points, index parity (Theorem \ref{thm:Iproperties}) tells us that $\partial=0$. Therefore
\begin{equation}\label{eqn:g=0diffl0}
ECH_*(Y,\xi,\Gamma)\cong C_*.
\end{equation}
Recall that we are abusing notation by thinking of $\Gamma$ as an element of $\Z_{-e}$ rather than as an element of $H_1(Y)$. The generator $e_-^{m_-}e_+^{m_+}$ is therefore in the class $\Gamma$ if $m_-+m_+=\Gamma+(-e)d$ for some $d\in\Z_{\geq0}$.

If $E_-,E_+$ denote the grading zero and two homology classes in $H_*(S^2;\Z_2)$, respectively, then $\Lambda^\delta H_*(S^2;\Z_2)$ is the group generated over $\Z_2$ by terms $E_-^{m_-}E_+^{m_+}$ with $m_-+m_+=\delta$, where $E^m$ denotes the $m$-fold wedge product $E\wedge\cdots\wedge E$. By a simplification of the proof of the analogous fact in \S\ref{sssec:pfg>0}, we obtain
\[
C_*\cong\bigoplus_{d\in\Z_{\geq0}}\Lambda^{\Gamma+(-e)d}H_*(S^2;\Z_2)
\]
as $\Z_2$-graded $\Z_2$-modules. Invoking (\ref{eqn:g=0diffl0}) and taking the sum over all $\Gamma\in H_1(Y)$ proves (\ref{eqn:mainthmg=0}).

To prove the improvement (\ref{eqn:ECHg=0}), we will prove that the ECH index is a bijection from the generators of $(C_*,\partial)$ to $2\Z_{\geq0}$ which sends $e_-^\Gamma$ to zero. Therefore
\[
C_*=\begin{cases}
\Z&\text{ if $*\in2\Z_{\geq0}$}
\\0&\text{ else}
\end{cases}
\]
and by (\ref{eqn:g=0diffl0}), we obtain (\ref{eqn:ECHg=0}).

The remainder of the proof of (\ref{eqn:ECHg=0}) therefore consists of proving that the ECH index is a bijection from generators $e_-^{m_-}e_+^{m_+}$ to $2\Z_{\geq0}$. Our perspective is similar to that of Choi \cite{choi}, but the contact forms are not the same (by choosing a specific perturbation function $H$ we could make them essentially the same, but do not need to do so).

The index bijection factors through a bijection to a lattice in the fourth quadrant in $\R^2$ determined by the vertical axis and the line through the origin of slope $\frac{1}{-e}$. We will first describe the bijection between generators of $C_*$ and this lattice, and then describe the bijection between the lattice and the nonnegative even integers.

The generators of $C_*$ in the class $\Gamma$ are of the form $e_-^{m_-}e_+^{m_+}$, where $m_\pm\in\Z_{\geq0}$ and $m_-+m_+\equiv_{-e}\Gamma$. See Figure \ref{fig:lattice_fig_722}. The union of all such generators over $\Gamma\in\Z_{-e}$ are in bijection with the intersection of the lattice spanned by $(1,0)$ and $\left(0,\frac{1}{-e}\right)$ with the fourth quadrant determined by the vertical axis and the line through the origin of slope $\frac{1}{-e}$, where the bijection is given by
\[
e_-^{m_-}e_+^{m_+}\mapsto\left(m_-,\frac{m_--m_+}{-e}\right).
\]
The image of $e_-^{m_-}e_+^{m_+}$ is to the right of the vertical axis, inclusive, because $m_-\geq0$, and is below the line through the origin of slope $\frac{1}{-e}$, inclusive, because
\[
\frac{m_--m_+}{-e}\leq\frac{m_-}{-e}.
\]
The map is a bijection because it has an inverse, which can be computed directly from the formula. Let $V(m_-,m_+)$ denote $\left(m_-,\frac{m_--m_+}{-e}\right)$.

\begin{figure}
 \begin{center}
 \begin{overpic}[width=.7\textwidth]{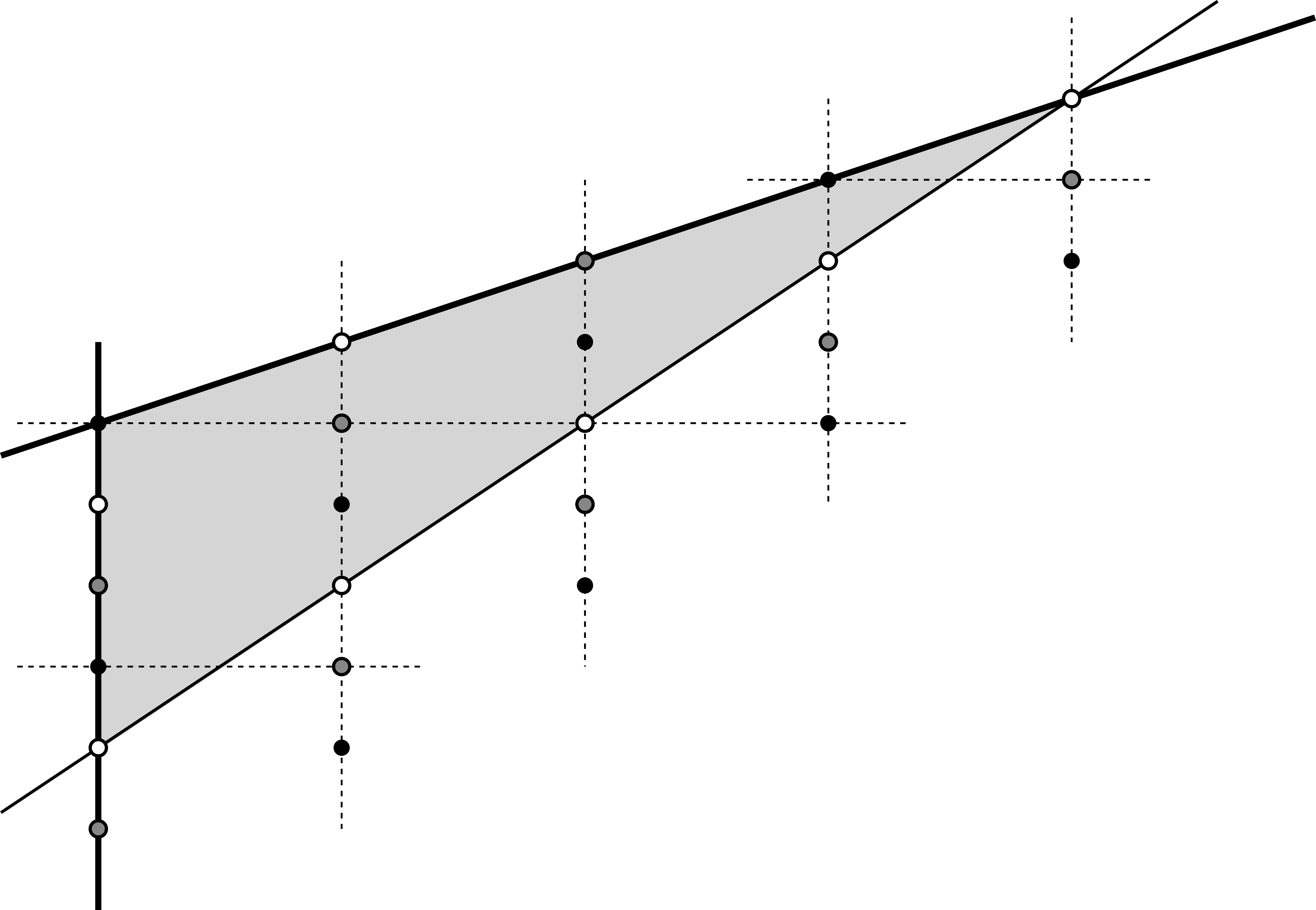}
\put(-8,128){\footnotesize $V(0,0)$}
\put(149,111){\footnotesize $V(2,2)$}
\put(88,73){\footnotesize $V(1,3)$}
\put(28,32){\footnotesize $V(0,4)$}
\put(269,193){\footnotesize $V(4,0)$}
\put(209,152){\footnotesize $V(3,1)$}
\end{overpic}
\end{center}
\caption{Depicted is the lattice for $e=-3$. The thicker solid lines indicate the axes while the dashed lines indicate the grid spanned by the standard lattice generated by $(1,0)$ and $(0,1)$. The $\Gamma=0$ sublattice is indicated by black points, the $\Gamma=1$ sublattice by white points, and the $\Gamma=2$ sublattice by dark grey points. The light grey triangle is $T(0,4)=T(1,3)=T(2,2)=T(3,1)=T(4,0)$.}
\label{fig:lattice_fig_722}
\end{figure}
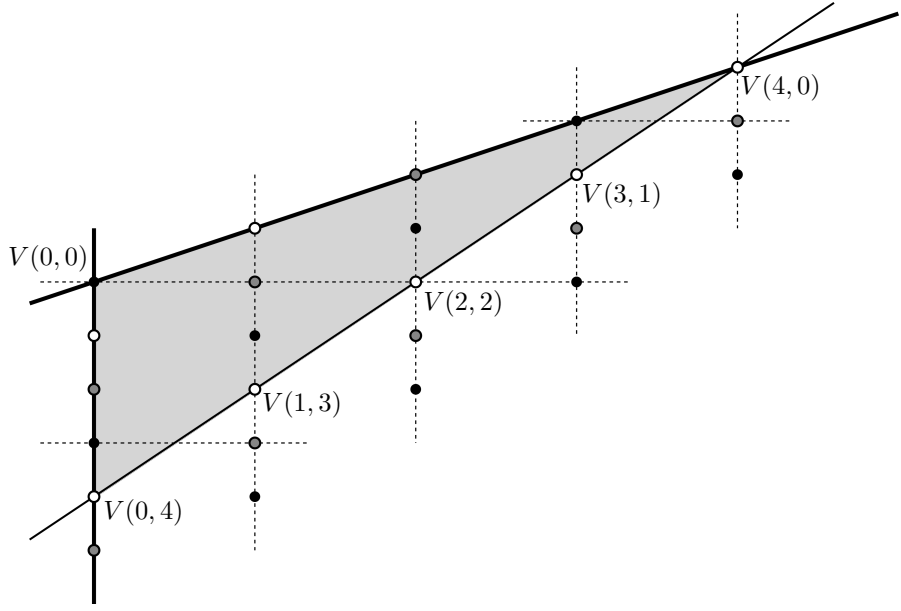

The lattice splits into $-e$ sublattices, each corresponding to the homology class $\Gamma$. Each is spanned by $\left(1,\frac{2}{-e}\right)$ and $(0,1)$, but they can be differentiated by a translation: they contain the points $\left(0,-\frac{\Gamma}{-e}\right)$, respectively.

Next we explain the bijections between each of these sublattices and the nonnegative even integers. The essential idea is the following. Let $T(m_-,m_+)$ denote the triangle bounded by the axes and the line through $V(m_-,m_+)$ of slope $\frac{2}{-e}$. The relative first Chern class records the approximate height of $T(m_-,m_+)$, the relative intersection pairing term records approximately twice its area; when this line moves to the right and/or down, both the height and area of $T(m_-,m_+)$ increase, so moving the line to the right and down groups points in the lattice into batches of roughly increasing ECH index. The Conley-Zehnder index differentiates between lattice points on the same line of slope $\frac{2}{-e}$ by increasing the index by two as $m_1$ increases and $m_0$ decreases, and makes the indices of the groups on different lines of slope $\frac{2}{-e}$ match up exactly.

Because these correspondences are only approximate when $\Gamma\neq0$, we will explain them in detail to show that the index is a bijection to $2\Z_{\geq0}$.

The index difference between $e_-^{m_-}e_+^{m_+}$ and $e_-^\Gamma$ is
\[
I(e_-^{m_-}e_+^{m_+},e_-^\Gamma)=\frac{2(m_-+m_+-\Gamma)}{-e}+\frac{(m_-+m_+-\Gamma)^2}{-e}+\frac{2\Gamma(m_-+m_+-\Gamma)}{-e}+m_+-m_-+\Gamma.
\]

The first term is the relative first Chern class $c_\tau(e_-^{m_-}e_+^{m_+},e_-^\Gamma)$. The triangle has vertices $(0,0),\left(-\frac{m_-+m_+}{-e},0\right),\left(m_-+m_+,\frac{m_-+m_+}{-e}\right)$, so its height is
\[
\frac{2(m_-+m_+)}{-e}=c_\tau(e_-^{m_-}e_+^{m_+},e_-^\Gamma)+\frac{2\Gamma}{-e}\Leftrightarrow c_\tau(e_-^{m_-}e_+^{m_+},e_-^\Gamma)=\text{Height}(T(m_-,m_+))-\frac{2\Gamma}{-e}.
\]

The second two terms comprise the relative intersection pairing $Q_\tau(e_-^{m_-}e_+^{m_+},e_-^\Gamma)$. Twice the area of the triangle is
\[
2\text{Area}(T(m_-,m_+))=\left(\frac{m_-+m_+}{-e}\right)(m_-+m_+),
\]
while
\begin{align*}
Q_\tau(e_-^{m_-}e_+^{m_+},e_-^\Gamma)&=\frac{(m_-+m_+-\Gamma)^2}{-e}+\frac{2\Gamma(m_-+m_+-\Gamma)}{-e}
\\&=\frac{1}{-e}\left((m_-+m_+)^2-2\Gamma(m_-+m_+)+\Gamma^2+2\Gamma(m_-+m_+)-2\Gamma^2\right)
\\&=2\text{Area}(T(m_-,m_+))-\frac{2\Gamma^2}{-e}.
\end{align*}

Notice that we can split the sublattices into lattices along the lines of slope $\frac{2}{-e}$ through $\left(-\frac{M+\Gamma}{-e},0\right)$, where $M\in\Z_{\geq0}$. Over each such line, the Conley-Zehnder term ranges from $-M+\Gamma$ to $M+\Gamma$, where $M=m_-+m_+$, and is strictly increasing in $m_+$. Since there is exactly one generator with each value of $m_+$ between zero and $m_-+m_+$ on this line, no values of $I(e_-^{m_-}e_+^{m_+},e_-^\Gamma)$ are repeated on a given line.

 Along each line, the triangle $T(m_-,m_+)$ is constant, therefore its height and area are constant, and thus both $c_\tau(e_-^{m_-}e_+^{m_+},e_-^\Gamma)$ and $Q_\tau(e_-^{m_-}e_+^{m_+},e_-^\Gamma)$ are constant. They are also both increasing in $M$. In order to prove the theorem it therefore suffices to show that the smallest value the index takes on the line corresponding to $M+(-e)$ must be two greater than the largest value the index takes on the line corresponding to $M$.
 
 The smallest value the index takes on the line corresponding to $M+(-e)$ is $I(e_-^{M+(-e)},e_-^\Gamma)$, while while the largest value the index takes on the line corresponding to $M$ is $I(e_+^M,e_-^\Gamma)$. It is a straightforward computation to show that $I(e_-^{M+(-e)},e_-^\Gamma)=I(e_+^M,e_-^\Gamma)+2$.
\end{proof}

\noindent \textsc{Jo Nelson \\  Rice University}\\
{\em email: }\texttt{jo.nelson@rice.edu}\\

\noindent \textsc{Morgan Weiler \\ Cornell University \\ University of California, Riverside}\\
{\em email: }\texttt{morgan.weiler@cornell.edu}\\

\end{document}